%% file: main.tex
\documentclass[11pt,a4paper]{article}
\usepackage{iftex}
\ifPDFTeX
  \usepackage[T1]{fontenc}
  \usepackage[utf8]{inputenc}
  \DeclareUnicodeCharacter{2265}{\ensuremath{\ge}}
\else
  \usepackage{fontspec}
  \catcode"2265=\active
  \def≥{\ensuremath{\ge}}
\fi
\usepackage[british]{babel}
\usepackage[a4paper,margin=28mm]{geometry}
\usepackage{amsmath,amssymb,amsthm,mathtools,mathrsfs}
\usepackage{enumitem,booktabs}
\usepackage{graphicx,tikz}
\usepackage{xcolor}
\usetikzlibrary{arrows.meta}
\PassOptionsToPackage{hyphens}{url}
\usepackage[hidelinks]{hyperref}
\numberwithin{equation}{section}
\newtheorem{theorem}{Theorem}[section]
\newtheorem{lemma}[theorem]{Lemma}

\newtheorem{corollary}[theorem]{Corollary}
\newtheorem{aoresult}[theorem]{Theorem}
\newtheorem{aoconjecture}[theorem]{Conjecture}
\newenvironment{restated}[2]{%
  \par\medskip\noindent\textbf{#1~\ref{#2}.}\ \itshape
}{\par\medskip}
\theoremstyle{definition}
\newtheorem{definition}[theorem]{Definition}

\newtheorem{example}[theorem]{Example}
\theoremstyle{remark}

\newcommand{\E}{\mathbb E}
\newcommand{\Pp}{\mathbb P}
\newcommand{\Var}{\operatorname{Var}}
\newcommand{\diam}{\operatorname{diam}}
\newcommand{\supp}{\operatorname{supp}}
\newcommand{\Tr}{\operatorname{Tr}}

\newcommand{\Reff}{\mathscr{R}}
\newcommand{\dimminkowski}{\dim_{\mathrm B}}
\newcommand{\dimresis}{\dim_{\mathrm R}}
\newcommand{\dimmass}{\dim_{\mathrm M}}
\newcommand{\dimspec}{\dim_{\mathrm S}}
\newcommand{\dist}{d}
\DeclareMathOperator{\esssup}{ess\,sup}
\title{Iterated Graph Systems (III): \\
Alexander-Orbach conjecture on hierarchical lattices}

\author{Ziyu Neroli\thanks{Email: \href{mailto:z5222549@zmail.unsw.edu.au}{\texttt{z5222549@zmail.unsw.edu.au}}; \href{mailto:ziyu.li21@imperial.ac.uk}{\texttt{ziyu.li21@imperial.ac.uk}}.}\\[3pt]
\small Department of Mathematics, Imperial College London, United Kingdom}
\date{}

\begin{document}
\maketitle

\begin{abstract}
Can the Alexander--Orbach conjecture hold on hierarchical lattices?
We construct Gromov--Hausdorff--Prokhorov scaling limits of critical percolation clusters on hierarchical lattices in the chemical metric and determine the typical spectral dimensions of the associated diffusions.
These dimensions are dense in $(1,2)$;
however, outside a narrow region in dimension space, we prove that the Alexander--Orbach conjecture fails on every hierarchical lattice under consideration.
\end{abstract}

\tableofcontents

\section{Introduction}\label{sec:introduction}

In 1982, Alexander and Orbach conjectured that the simple random walk on the incipient infinite cluster of critical Bernoulli percolation on $\mathbb Z^d$ has spectral dimension $4/3$ in every dimension $d\ge2$~\cite{alexander1982density}.
Kozma and Nachmias~\cite{kozma2009alexander} established this value in the high-dimensional mean-field regime; combined with the work of Fitzner and van der Hofstad~\cite{fitzner2017meanfield}, their result applies to nearest-neighbour percolation in dimensions $d\ge11$.
However, scaling arguments suggest that the conjecture fails in low dimensions~\cite{aharony1984breakdown}, and numerical estimates of the two-dimensional conductivity exponent deviate from its prediction~\cite{lobb1984percolative,grassberger1999conductivity}.
For further discussion of the conjecture and related work on spectral dimensions, diffusion and transport, see also~\cite{gefen1983anomalous,stauffer1994percolation,croydon2008selfsimilarity,croydon2010spectral,hambly2010diffusion}.
Despite this numerical evidence, the conjecture remains mathematically open for nearest-neighbour percolation on $\mathbb Z^d$ in dimensions $2\le d\le10$.

Around the 1970s, researchers began studying statistical-physics problems on graphs with a hierarchical structure rather than Euclidean geometry~\cite{berker1979renormalisation,kaufman1981exactly,griffiths1982spin}.
Such a model is called a \emph{hierarchical lattice}: a graph with self-similar, fractal structure on which an exact renormalisation can usually be carried out.
The classical example is the diamond hierarchical lattice (DHL), shown in Figure~\ref{fig:dhl-construction}.

\begin{figure}[htbp]
\centering
\resizebox{\linewidth}{!}{\input{figure/dhl_construction.tex}}
\caption{The diamond hierarchical lattice (DHL).}
\label{fig:dhl-construction}
\end{figure}
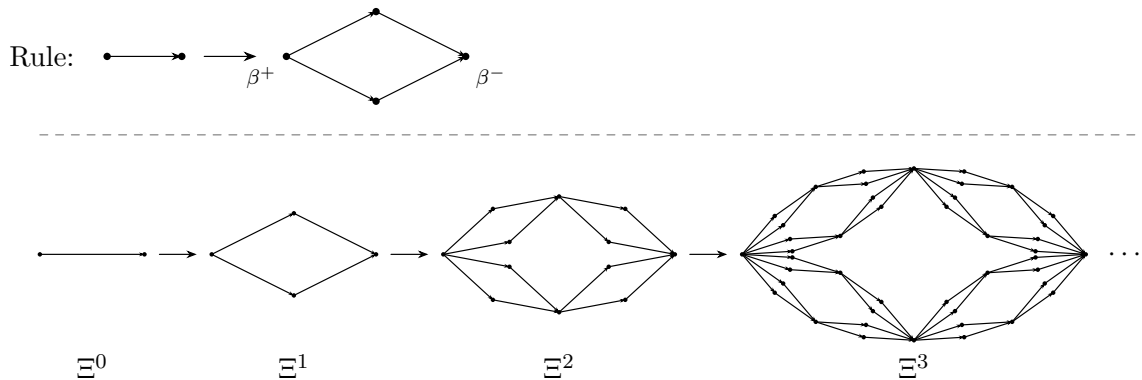

Hambly and Kumagai~\cite{hambly2010diffusion} in 2010 brought these two subjects together by studying diffusion on the scaling limit of the critical percolation cluster on the DHL.
They first obtained heat-kernel estimates on the DHL.
They then showed that the percolation cluster on the DHL has a graph-directed fractal representation, making it more amenable to analysis.
Finally, they determined the spectral dimension of a particular diffusion and proved that it differs from $4/3$.

We aim to examine this problem in depth through a series of papers.

In IGS~(I): \emph{Random walks and diffusion limits}~\cite{Neroli2026}, we studied random walks on a more general class of hierarchical lattices and proved their convergence to a diffusion limit.
Towards the end of that paper, we moved from diffusion on the ambient space to diffusion on its percolation cluster.
After determining the spectral dimension of the ambient space, we solved the open problem left by Hambly and Kumagai~\cite{hambly2010diffusion}, namely the almost-sure existence of the exponential growth rate of the two-terminal effective resistance on the DHL percolation cluster when every open edge has unit resistance.

In IGS~(II): \textcolor{black}{\emph{Bernoulli percolation and critical-exponent universality classes on hierarchical lattices}}~\cite{NeroliIGSII2026}, we continued from the end of IGS~(I) and studied the foundations of percolation in detail.
We then examined eight critical exponents and used four of them to define the critical-exponent universality class.
We rigorously proved that no finite family of substitution-multiplicative dimensions completely determines universality, providing a rigorous proof of the failure of dimension-based universality discussed by \textcolor{black}{Hauser and Saxena~\cite{hauser1986geometrical}}.

\textcolor{black}{In this third paper, we study the scaling limit of the full terminal-connected cluster at each lattice's own critical point, together with the natural diffusion obtained from simple random walks with unit edge resistances.
Here, the Alexander--Orbach conjecture on hierarchical lattices asserts that this diffusion has small-time spectral dimension $4/3$ at almost every point with respect to the limiting mass measure.
We are naturally led to the following questions:}
\begin{itemize}
    \item In IGS~(I), we proved the existence of the natural resistance growth exponent for the percolation cluster on the DHL, resolving an open problem. Can this existence result be extended to the critical percolation cluster on every classical EIGS hierarchical lattice?
    \item As asked at the end of IGS~(I), can we rigorously determine the spectral dimension of the critical percolation cluster on every classical EIGS hierarchical lattice?
    \item Returning to the question that motivated the series, can the critical percolation cluster on a classical EIGS hierarchical lattice have spectral dimension $4/3$?
\end{itemize}

\subsection{Edge iterated graph systems}
We describe hierarchical lattices through the constructions that generate them, treating the generating system, its finite substitution networks and the resulting lattice as distinct objects~\cite{Neroli2026}.
The general framework of Griffiths and Kaufman~\cite{griffiths1982spin} arranges units in successive levels, with uniform bounds on both the size of their boundaries and the number of units from which they are assembled.
Further work on iterated graph constructions can be found in~\cite{li2024on,xi2017fractality,li2025reducible,anttila2024replacement,anttila2026loewner,anttila2025energy}.

\begin{samepage}
\begin{definition}[{\cite[Definition~1.1]{NeroliIGSII2026}; \cite{neroli2024fractal}}]\label{definition:EIGS}
A \emph{random edge iterated graph system} (random EIGS) is specified by the following data and iteration.
{\small
\begin{enumerate}[label=\arabic*.,leftmargin=*,topsep=3pt,itemsep=1pt,parsep=0pt]
\item \textbf{Initial graph.} A finite directed graph $\Xi^0$ with edge colours in $[K]:=\{1,\ldots,K\}$; parallel edges are allowed. For a single initial edge, write its endpoints as $\mathfrak v_+,\mathfrak v_-$.
\item \textbf{Rule laws.} For each colour, a fixed probability law on finitely many finite directed rule graphs, with the same colour set and two ordered distinct \emph{planting vertices}.
\item \textbf{Substitution.} Replace each edge $(a,b)$ by a fresh rule sampled according to its colour, identifying the planting vertices with $a,b$ and making no further identifications. Samples are independent across edges, with fresh randomness at each step.
\item \textbf{Iteration.} Simultaneous substitution produces the \emph{substitution networks} $\Xi^0,\Xi^1,\ldots$. The EIGS is \emph{deterministic} if every rule law is concentrated on one graph.
\end{enumerate}}
\end{definition}
\end{samepage}

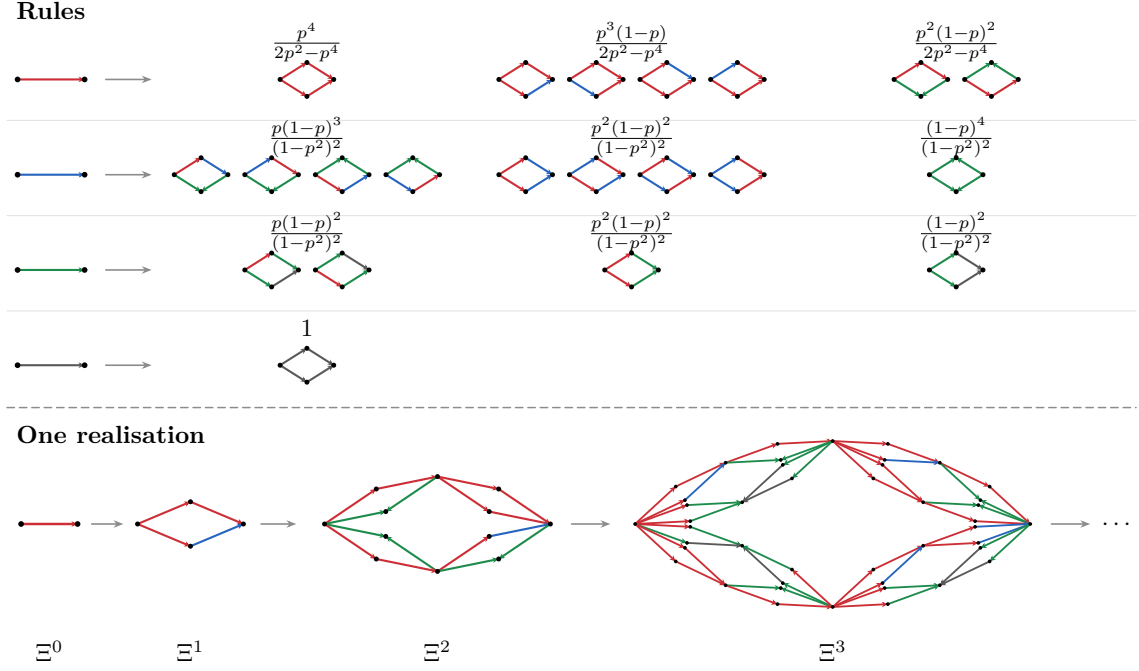
\begin{figure}[tbp]
\centering
\resizebox{\linewidth}{!}{\input{figure/random_DHL_rules.tex}}
\caption{A particular example of a random EIGS; it is induced by Bernoulli bond percolation on the DHL.}
\label{fig:random-eigs-dhl}
\end{figure}

For example, Figure~\ref{fig:random-eigs-dhl} gives the four-colour rules for critical percolation on the diamond hierarchical lattice, with $p=p_c=(\sqrt5-1)/2$.
Each displayed probability applies to each rule immediately below it; the bottom row shows one realisation of the first three substitutions.
Starting from one red edge, the red subgraph of $\Xi^n$ has the law of the level-$n$ percolation cluster conditioned on the two terminals being connected; the other colours record auxiliary terminal states.

Following~\cite{Neroli2026}, when the substitution networks admit a limit, we call it an \emph{EIGS hierarchical lattice}, distinguishing the combinatorial limit from the metric scaling limit when needed.
The EIGS is the generating system, and $(\Xi^n)_{n\in\mathbb N}$ is its sequence of finite approximating networks.
We retain the chosen EIGS construction when discussing a hierarchical lattice, so that its substitution steps and renormalisation scale are specified.

In this paper the underlying lattices are generated from one edge by deterministic single-colour systems.
Each edge is therefore replaced by the same two-terminal rule graph at every step.
Bernoulli percolation on such a deterministic system already produces a random coloured EIGS, whose colours record the two-terminal activity states of a cell.
This setting lets us study that mechanism while keeping the deterministic construction simple.
We use the term \emph{classical EIGS hierarchical lattice} for the following class.

\begin{definition}[{\cite[Definition~1.2]{NeroliIGSII2026}}]\label{def:standing}
A \textbf{classical EIGS hierarchical lattice} is generated from a single directed edge by a deterministic single-colour EIGS with a finite, connected, simple rule graph $R=(V(R),E(R);\beta^+,\beta^-)$.
Connectivity, distances and percolation refer to its underlying undirected graph.
{\small
\begin{enumerate}[label=\arabic*.,leftmargin=*,topsep=3pt,itemsep=1pt,parsep=0pt]
\item \textbf{Canonicality.} Every edge lies on a simple path between $\beta^+$ and $\beta^-$.
\item \textbf{Geometric bounds.} We require $\ell:=\dist_R(\beta^+,\beta^-)\ge2$ and $\operatorname{mincut}_R(\beta^+,\beta^-)\ge2$, where $\operatorname{mincut}_R(\beta^+,\beta^-)$ is the minimum size of a terminal-separating edge cut.
\item \textbf{Terminal symmetry.} The underlying undirected graph admits an involutive automorphism exchanging $\beta^+$ and $\beta^-$.
\end{enumerate}}
\end{definition}

The two geometric bounds in Definition~\ref{def:standing} characterise the existence of a nontrivial two-terminal crossing transition for a finite connected rule graph; see~\cite[Theorem~3.1 and Remark~5]{alves2026percolation}.
In particular, once $\ell\ge2$, the condition $\operatorname{mincut}_R(\beta^+,\beta^-)\ge2$ is necessary and sufficient for a critical crossing parameter in $(0,1)$ separating vanishing from asymptotically certain terminal connection.
This criterion does not require canonicality or terminal symmetry.

Unless stated otherwise, the hierarchical lattices studied below are classical EIGS hierarchical lattices in the sense of Definition~\ref{def:standing}.
Write $\Xi^n$ for the depth-$n$ substitution graph of the rule $R=(V(R),E(R);\beta^+,\beta^-)$, with inherited terminals $\mathfrak v_+,\mathfrak v_-$.
We use the notation $m:=|E(R)|$, $\ell:=\dist_R(\beta^+,\beta^-)$ and $d_R:=\deg_R(\beta^+)=\deg_R(\beta^-)$ throughout.
Let $L:=\max\{|P|:P\text{ is a simple terminal path in }R\}$.
Fix edge labels $1,\ldots,m$ and orientations specifying the ordered planting maps.
Cell addresses are words $u\in\{1,\ldots,m\}^{<\infty}$, with generation $|u|$.


\subsection{Main results}

Our aim is to show that the Alexander--Orbach conjecture fails on every classical EIGS hierarchical lattice.

We build on the percolation theory developed in IGS~II~\cite{NeroliIGSII2026}.
Readers may consult its definition of classical EIGS hierarchical lattices (Definition~1.2), critical point and renormalisation results (Theorem~2.4 and Corollary~2.5), random EIGS representation of the critical cluster (Definition~2.7 and Theorem~2.8), and mass-transfer matrix and mass dimension (Definition~2.10, Proposition~2.11 and Theorem~2.14) as needed.
We use these foundations throughout, together with Lemmas~3.2--3.3 and their proofs for the mass estimates and limits.

Let $(\mathfrak C,d_{\mathrm{ch}},\mu)$ be the Gromov--Hausdorff--Prokhorov scaling limit constructed in Section~\ref{sec:geometry} from the full terminal-connected critical clusters in the chemical metric.
Let $\ell_{\mathfrak C}^{\mathrm{ch}}$ be the chemical distance growth factor, and let $\rho(\mathbf M_{\mathfrak C})$ be the spectral radius of the critical mass-transfer matrix of IGS~II.
The chemical mass exponent agrees with the Minkowski dimension of this limit.
\begin{restated}{Theorem}{prop:chemical-minkowski-dimension}
Almost surely,
\[
\dimminkowski(\mathfrak C,d_{\mathrm{ch}})=\frac{\log\rho(\mathbf M_{\mathfrak C})} {\log\ell_{\mathfrak C}^{\mathrm{ch}}}.
\]
\end{restated}
This dimension uses the chemical scale, whereas the ambient mass dimension $\log\rho(\mathbf M_{\mathfrak C})/\log\ell$ studied in IGS~II uses the terminal distance $\ell$ of the rule graph.

Section~\ref{sec:resistance} next establishes the existence of the natural resistance growth exponent.
Let $\mathfrak C_n$ be the full critical percolation cluster in $\Xi^n$ conditioned on terminal connection, with the clusters coupled through their random EIGS representation.
Write $\Reff_{\mathfrak C_n}(\mathfrak v_+,\mathfrak v_-)$ for the effective resistance between the cluster's terminals when every open edge has unit resistance.
\begin{restated}{Theorem}{thm:resistance-exponent}
For every classical EIGS hierarchical lattice, there is a deterministic $\lambda_{\mathfrak C}>1$ such that, almost surely,
\[
\lim_{n\to\infty}\frac{\log\Reff_{\mathfrak C_n}(\mathfrak v_+,\mathfrak v_-)}{n}
=\log\lambda_{\mathfrak C}.
\]
\end{restated}
For the DHL, this is the resistance-growth question posed by Hambly and Kumagai~\cite[Section~6(i)]{hambly2010diffusion}.
The theorem extends its solution in IGS~I~\cite{Neroli2026} to every classical EIGS hierarchical lattice.
Using $\lambda_{\mathfrak C}$, Section~\ref{sec:resistance} constructs the limiting resistance form.

Section~\ref{sec:diffusion} then constructs the diffusion as the scaling limit of the natural random walks and determines its typical spectral dimension.
Write $\dimresis(\mathfrak C,d_{\mathrm{ch}}):=\log\lambda_{\mathfrak C}/\log\ell_{\mathfrak C}^{\mathrm{ch}}$ for the resistance dimension, and let $\dimspec^{(\mathrm N)}(\mathfrak C:x)$ denote the pointwise spectral dimension of this diffusion.
\begin{restated}{Theorem}{thm:heat-kernels}
Almost surely, for $\mu$-almost every $x\in \mathfrak C$,
\[
 \dimspec^{(\mathrm N)}(\mathfrak C:x)
 =
 \frac{2\dimminkowski(\mathfrak C,d_{\mathrm{ch}})}{\dimminkowski(\mathfrak C,d_{\mathrm{ch}})+\dimresis(\mathfrak C,d_{\mathrm{ch}})}.
\]
\end{restated}
Thus both the natural resistance exponent and the typical spectral dimension exist throughout the classical class.

Combining these geometric and diffusion results, Section~\ref{sec:ao} returns to our central question and proves the main theorem.
\begin{restated}{Theorem}{thm:ao-region}
For every classical EIGS rule, if
\[
 (\dimminkowski(\mathfrak C,d_{\mathrm{ch}}),\dimresis(\mathfrak C,d_{\mathrm{ch}}))
 \notin (\frac{4}{3},2)\times(\frac{2}{3},1),
\]
then the typical spectral dimension of the natural limiting diffusion on its critical cluster satisfies
\[
 \dimspec^{(\mathrm N)}(\mathfrak C)\neq\frac{4}{3}.
\]
\end{restated}
The remaining region is narrow, but excluding equality within it remains difficult.
We also seek an example satisfying the Alexander--Orbach equality by approximating prescribed spectral dimensions.
This leads to the following result.
\begin{restated}{Theorem}{thm:spectral-density}
The typical spectral dimensions $\dimspec^{(\mathrm N)}(\mathfrak C)$ of the natural limiting diffusions on critical percolation clusters of classical EIGS hierarchical lattices are dense in $(1,2)$.
\end{restated}
Despite this density result, we have not found an example with spectral dimension exactly $4/3$.
We therefore propose the following conjecture.
\begin{restated}{Conjecture}{conj:ao-global}
For every classical EIGS hierarchical lattice at its own critical point, the typical spectral dimension of the natural limiting diffusion on the full terminal-connected cluster satisfies
\[
\dimspec^{(\mathrm N)}(\mathfrak C) \neq \frac{4}{3}.
\]
That is, the Alexander-Orbach conjecture does not hold on classical EIGS hierarchical lattices.
\end{restated}

\subsection{Lean formalisation}

Sections~\ref{sec:geometry}--\ref{sec:ao} and Appendices~\ref{app:chemical}--\ref{app:spectral-density} have been formalised in Lean~4.32.1 using mathlib.
The proofs use only the standard axioms \texttt{propext}, \texttt{Classical.choice} and \texttt{Quot.sound}, with no \texttt{sorry} placeholders or additional axioms.
The global exclusion in Conjecture~\ref{conj:ao-global} remains open.

The code is available at \url{https://github.com/Nero-17/LEAN-Formalisation-IGS-Alexander-Orbach-conjecture-on-hierarchical-lattice}.

\section{Chemical geometry}\label{sec:geometry}

\textcolor{black}{In this section, we construct the GHP scaling limit of the full critical cluster in the chemical metric and identify its Minkowski dimension with the chemical mass exponent.
This replaces the ambient scale $\ell$ of IGS~II \cite{NeroliIGSII2026} by the intrinsic growth factor $\ell_{\mathfrak C}^{\mathrm{ch}}$, providing the geometry needed for the resistance and diffusion analysis below.}

For the foundations of Bernoulli percolation on classical EIGS hierarchical lattices, we refer to \cite[Definition~1.2 and Section~2.1]{NeroliIGSII2026}.
We use in particular the critical crossing parameter, the random EIGS representation of the full terminal-connected cluster, and the mass-transfer matrix that determines its expected mass growth.

Recall from \cite[Section~2.1]{NeroliIGSII2026} that $\varphi_R(p):=\Pp_p(\beta^+\leftrightarrow\beta^-\text{ in }R)$ is the terminal reliability polynomial, where each edge is independently open with probability $p$.
By \cite[Theorem~2.4 and Corollary~2.5]{NeroliIGSII2026}, its unique fixed point $p_c\in(0,1)$ is the critical crossing parameter.

\begin{definition}\label{def:critical-cluster}
As in \cite[Section~2.1]{NeroliIGSII2026}, let $\mathfrak C_n$ be the full terminal-connected cluster in $\Xi^n$, conditioned on crossing, including all dangling branches.
Its chemical distance $d_{\mathfrak C_n}$ is the intrinsic graph distance along open edges.
\end{definition}

The rule geometry first gives deterministic bounds on this distance.
\begin{lemma}\label{lem:chemical-elementary-bounds}
For every $n\ge0$ and every terminal-connected configuration,
\[
 \ell^n\le d_{\mathfrak C_n}(\mathfrak v_+,\mathfrak v_-)\le L^n.
\]
\end{lemma}
\begin{proof}
At depth zero the terminal distance is one.
Coarsening a terminal path at depth $n+1$ gives a crossing of the rule graph and hence contains a simple terminal path of at least $\ell$ edges.
Each corresponding depth-$n$ cell must be crossed, so induction gives the lower bound.
Conversely, choose a simple open terminal path in the coarse rule, which has at most $L$ edges, and cross each of its cells along a shortest open path.
Induction bounds the resulting path length by $L^{n+1}$.
\end{proof}

To identify the exact growth factor, we record how terminal distances compose under substitution.
\begin{definition}\label{def:chemical-recursion}
Let $\mathcal Q_c$ be the Bernoulli-$p_c$ rule law conditioned on crossing, and let $H$ be the union of its open edges lying on simple terminal paths.
For nonnegative edge lengths $x$, define the shortest-path functional
\[
 F_H^{\mathrm{ch}}(x):=\min_P\sum_{e\in P}x_e,
\]
where $P$ ranges over the simple terminal paths of $H$.
\end{definition}

\begin{lemma}\label{lem:chemical-path-functional}
The function $F_H^{\mathrm{ch}}$ is increasing, concave, positively homogeneous and superadditive, and satisfies
\[
 |F_H^{\mathrm{ch}}(x)-F_H^{\mathrm{ch}}(y)|\le\sum_e|x_e-y_e|.
\]
\end{lemma}
\begin{proof}
Each path cost is linear with nonnegative coefficients, so its minimum is increasing, concave and positively homogeneous.
For every path, its cost at $x+y$ is at least $F_H^{\mathrm{ch}}(x)+F_H^{\mathrm{ch}}(y)$, which proves superadditivity.
Evaluating the cost at $x$ along a minimising path for $y$, and then exchanging $x$ and $y$, gives the stated bound.
\end{proof}

Place independent crossing innovations $\xi_u^c$ with law $\mathcal Q_c$ at all cell addresses, and write $H_u$ for their terminal path cores.
For a connected address $u$, let $\mathfrak C_{u,n}$ be its depth-$n$ descendant cluster in the recursive representation of \cite[Definition~2.7 and Theorem~2.8]{NeroliIGSII2026}, with inherited terminals again denoted by $\mathfrak v_+,\mathfrak v_-$.
An omitted address denotes the root; an additional copy index denotes an independent descendant cluster.
Coarsening and expanding terminal paths gives the common-tree recursion
\[
 d_{\mathfrak C_{u,n+1}}(\mathfrak v_+,\mathfrak v_-)
 =F_{H_u}^{\mathrm{ch}}\bigl((d_{\mathfrak C_{ui,n}}(\mathfrak v_+,\mathfrak v_-))_i\bigr),
 \qquad d_{\mathfrak C_{u,0}}(\mathfrak v_+,\mathfrak v_-)=1.
\]
Unused coordinates may be filled by independent connected copies.
Refining the last generation also gives $\ell d_{\mathfrak C_{u,n}}(\mathfrak v_+,\mathfrak v_-)\le d_{\mathfrak C_{u,n+1}}(\mathfrak v_+,\mathfrak v_-)\le Ld_{\mathfrak C_{u,n}}(\mathfrak v_+,\mathfrak v_-)$.
The exact growth factor and its limiting field are given by the following theorem.

\begin{theorem}\label{thm:chemical-field}
There are a deterministic $\ell_{\mathfrak C}^{\mathrm{ch}}\in[\ell,L]$ and a strictly positive mean-one field $(X_u^{\mathrm{ch}})$, measurable with respect to the innovation tree, such that almost surely at every connected address,
\[
 F_{H_u}^{\mathrm{ch}}(X_{ui}^{\mathrm{ch}})=\ell_{\mathfrak C}^{\mathrm{ch}}X_u^{\mathrm{ch}}.
\]

\end{theorem}

\begin{proof}
Lemma~\ref{lem:chemical-moments} constructs the mean-one eigenlaw and proves its moment bounds. Lemma~\ref{lem:chemical-compactness} realises it on the original innovation tree. The recursion holds at each address by that construction, and a countable intersection gives the simultaneous assertion. The lower bound $\ell_{\mathfrak C}^{\mathrm{ch}}\ge\ell$ follows from the first-moment estimate in Lemma~\ref{lem:chemical-compactness}; the selected-path bound in Lemma~\ref{lem:chemical-moments} gives $\ell_{\mathfrak C}^{\mathrm{ch}}\le L$.
\end{proof}

The moment estimates use the normalised reliability function $g(u):=\varphi_R(p_cu)/p_c$.
By \cite[Theorem~2.4 and Corollary~2.5]{NeroliIGSII2026}, $\varphi_R'(p_c)>1$ and $g(u)<u$ for $0<u<1$; a union bound over simple terminal paths gives $g(u)=O(u^\ell)$ as $u$ tends to zero.

\begin{lemma}\label{lem:chemical-field-integrability}
The marginal law of the field in Theorem~\ref{thm:chemical-field} has every positive and negative moment and every finite positive exponential moment.
\end{lemma}

\begin{proof}
The field constructed in Lemma~\ref{lem:chemical-compactness} has the marginal law of Lemma~\ref{lem:chemical-moments}, where these integrability properties are proved.
\end{proof}

For the iterated path estimates, let $\mathcal H_n$ denote the realised depth-$n$ terminal path-core graph in this recursion, distinct from the deterministic host $\Xi^n$.
Write $F_{\mathcal H_n}^{\mathrm{ch}}$ for its shortest-path functional.
These estimates identify the limit of the actual terminal distances with the field above.

\begin{theorem}\label{thm:chemical-scalar}
For the field of Theorem~\ref{thm:chemical-field}, there is a deterministic $c_{\mathrm{ch}}\ge1$ such that
\[
 \lim_{n\to\infty}(\ell_{\mathfrak C}^{\mathrm{ch}})^{-n}d_{\mathfrak C_{u,n}}(\mathfrak v_+,\mathfrak v_-)=
 c_{\mathrm{ch}}X_u^{\mathrm{ch}}
 \quad\text{almost surely and in every finite positive }L^p
\]
simultaneously at every connected address.
\end{theorem}

\begin{proof}
Lemmas~\ref{lem:chemical-moments}--\ref{lem:chemical-endpoint}, proved in Appendix~\ref{app:chemical}, apply to the field of Theorem~\ref{thm:chemical-field} and give a deterministic $c\ge1$ and a subsequence along which $\bigl((\ell_{\mathfrak C}^{\mathrm{ch}})^{-n}d_{\mathfrak C_{n}}(\mathfrak v_+,\mathfrak v_-)\bigr)$ converges to $cX^{\mathrm{ch}}$ in $L^2$.
The remaining step is to pass from that subsequence to every depth.
Select $n_j$ with $\displaystyle\lim_{j\to\infty}\bigl((\ell_{\mathfrak C}^{\mathrm{ch}})^{-{n_j}}d_{\mathfrak C_{n_j}}(\mathfrak v_+,\mathfrak v_-)\bigr)=cX^{\mathrm{ch}}$ in $L^2$.
Fix $0<\varepsilon<c$ and choose $\theta>0$ so that $C_R\E e^{-\theta\varepsilon X^{\mathrm{ch}}}<1$.
Both $e^{\theta(\bigl((\ell_{\mathfrak C}^{\mathrm{ch}})^{-{n_j}}d_{\mathfrak C_{n_j}}(\mathfrak v_+,\mathfrak v_-)\bigr)-(c+\varepsilon)X^{\mathrm{ch}})}$ and $e^{\theta((c-\varepsilon)X^{\mathrm{ch}}-\bigl((\ell_{\mathfrak C}^{\mathrm{ch}})^{-{n_j}}d_{\mathfrak C_{n_j}}(\mathfrak v_+,\mathfrak v_-)\bigr))}$ converge in probability to $e^{-\theta\varepsilon X^{\mathrm{ch}}}$.
Uniform positive exponential moments make them uniformly integrable.
Choose one fixed $j$ for which both signed conditions in Lemma~\ref{lem:chemical-paths} hold.

Above this block, the pairs $(\bigl((\ell_{\mathfrak C}^{\mathrm{ch}})^{-{n_j}}d_{\mathfrak C_{n_j}}(\mathfrak v_+,\mathfrak v_-)\bigr),X^{\mathrm{ch}})$ are independent across the exposed leaves and retain their within-pair coupling.
Outside a summable sequence of exceptions, every terminal path satisfies both signed comparisons.
Taking minima gives $(c-\varepsilon)X^{\mathrm{ch}}\le \bigl((\ell_{\mathfrak C}^{\mathrm{ch}})^{-{n+n_j}}d_{\mathfrak C_{n+n_j}}(\mathfrak v_+,\mathfrak v_-)\bigr)\le(c+\varepsilon)X^{\mathrm{ch}}$ for every sufficiently large $n$.
This fixed shift covers the whole eventual sequence.
A countable sequence $\varepsilon\downarrow0$ gives almost-sure convergence.
Uniform moments give every finite positive $L^p$ convergence and convergence of means.
Set $c_{\mathrm{ch}}:=c$.
A countable address intersection proves the field assertion.
\end{proof}
In particular, the $L^1$ conclusion and $\E X^{\mathrm{ch}}=1$ give
\[
 \lim_{n\to\infty}(\ell_{\mathfrak C}^{\mathrm{ch}})^{-n}\E d_{\mathfrak C_n}(\mathfrak v_+,\mathfrak v_-)=c_{\mathrm{ch}}.
\]

\begin{example}\label{ex:chemical-scales}
If all simple terminal paths of the rule have length $\ell$, induction gives $d_{\mathfrak C_{n}}(\mathfrak v_+,\mathfrak v_-)=\ell^n$ on every crossing cluster; hence $\ell_{\mathfrak C}^{\mathrm{ch}}=\ell$ and $X^{\mathrm{ch}}=1$.
For the standard diamond, this gives $\ell_{\mathfrak C}^{\mathrm{ch}}=2$.

Strict inequality already occurs for a five-cycle with terminal paths of lengths two and three.
Its terminal exchange fixes the middle vertex of the shorter path and interchanges the two internal vertices of the longer path, so it is a classical rule.
Here $\ell=2$ and $\varphi_R(p)=p^2+p^3-p^5$; the interior critical point satisfies $p_c^3+p_c^2=1$.
Conditioned on a crossing, the two coarse branches are open with probabilities $p_c$ and $p_c^2$.
For independent copies of $d_{\mathfrak C_{n}}(\mathfrak v_+,\mathfrak v_-)$, convexity gives
$\E(d_{\mathfrak C_{n,1}}(\mathfrak v_+,\mathfrak v_-)+\cdots+d_{\mathfrak C_{n,j}}(\mathfrak v_+,\mathfrak v_-))^{-3}\le j^{-3}\E d_{\mathfrak C_{n}}(\mathfrak v_+,\mathfrak v_-)^{-3}$.
The inverse cube of the minimum over open branches is at most the sum of their inverse cubes, so
\[
 \E d_{\mathfrak C_{n+1}}(\mathfrak v_+,\mathfrak v_-)^{-3}\le\left(\frac{p_c}{8}+\frac{p_c^2}{27}\right)\E d_{\mathfrak C_{n}}(\mathfrak v_+,\mathfrak v_-)^{-3}.
\]
Iteration and Markov's inequality show that, for every
$c<(p_c/8+p_c^2/27)^{-1/3}$, the probabilities $\Pp(d_{\mathfrak C_{n}}(\mathfrak v_+,\mathfrak v_-)\le c^n)$ are summable.
Consequently
\[
 \ell_{\mathfrak C}^{\mathrm{ch}}\ge
 \left(\frac{p_c}{8}+\frac{p_c^2}{27}\right)^{-1/3}\approx2.0536>2.
\]
The strict inequality also follows without numerical approximation:
$p_c<4/5$ gives $p_c/8+p_c^2/27<167/1350<1/8$.
\end{example}

Terminal distances depend only on paths connecting the two terminals.
To recover the mass and geometry of the full cluster, we now retain its dangling branches and the components inside noncrossing cells.

\begin{definition}\label{def:retained-states}
We retain the states $c,b,u,0$ and the inherited endpoint markings of \cite[Definition~2.7 and Theorem~2.8]{NeroliIGSII2026}.
The state $c$ retains the crossing component, $b$ retains both terminal components of a noncrossing cell, $u$ retains one specified terminal component of a noncrossing cell, and $0$ retains nothing.
A cell of state $u$ carries its inherited active terminal as a mark; terminal symmetry identifies the two choices in law.
For a child edge, assign $c$ when it is open and both endpoints are active, $b$ when it is closed and both endpoints are active, $u$ when exactly one endpoint is active, and $0$ otherwise.
At depth zero the retained objects are an open edge, two isolated terminals, one marked terminal, or the empty set, according to their state.
\end{definition}

Let $\mathcal Q_N$ denote the Bernoulli-$p_c$ rule law conditioned on noncrossing.
Extend the crossing innovations to independent pairs $(\xi_u^c,\xi_u^N)$ at all addresses, with the second marginal $\mathcal Q_N$; the parent type selects the appropriate innovation.
The cited representation gives the law of $\mathfrak C_n$ at every depth, with stationary finite-support reproduction and conditional independence between distinct parents; siblings may be dependent.
Keeping only $c$-children does not in general give the terminal path core, since a connected child may lie in a dangling coarse branch.

We use the mass-transfer matrix $\mathbf M_{\mathfrak C}$ of \cite[Definition~2.10]{NeroliIGSII2026}, with rows and columns ordered as $(c,b,u)$.
Its entry $M_{\sigma,\tau}$ is the expected number of state-$\tau$ children of a state-$\sigma$ parent.
The representation theorem gives a three-type branching process, and the matrix is primitive by \cite[Proposition~2.11]{NeroliIGSII2026}.
We also use $2\le d_R<\rho(\mathbf M_{\mathfrak C})<m$ from \cite[Lemma~3.2]{NeroliIGSII2026}.
Since every state-$c$ parent has at least $\ell$ connected children along a terminal path, $\rho(\mathbf M_{\mathfrak C})\ge M_{c,c}\ge\ell$.

Following \cite[Lemmas~3.2--3.3 and their proofs]{NeroliIGSII2026}, let $X_{\sigma,n}$ count the active internal vertices in a depth-$n$ cell of state $\sigma$, excluding its outer terminals, and let $W_\sigma$ be the $L^2$ limit of $\rho(\mathbf M_{\mathfrak C})^{-n}X_{\sigma,n}$.
Write $\mathbf Z_k$ for the generation-$k$ population row vector in the order $(c,b,u)$.
Let $\mathbf h$ be the positive right Perron vector with $h_\sigma:=\E W_\sigma$, and normalise the positive left Perron vector $\mathbf v$ by $\mathbf v\mathbf h=1$.
This fixes the normalisation so that the population limit for a state-$c$ root is $W_c\mathbf v$, as in the branching-limit argument in \cite[proof of Lemma~3.3]{NeroliIGSII2026}.
For a cell at address $u$, let $T_u$ denote its inherited state; a superscript $(u)$ denotes its descendant realisation, so its mass limit is $W_{T_u}^{(u)}$.
The root has state $c$, and its mass limit is denoted by $W_c$.
The geometric construction below requires these limits almost surely on this same tree.

\begin{lemma}\label{lem:population-limit}
Almost surely,
\[
 \lim_{k\to\infty}\rho(\mathbf M_{\mathfrak C})^{-k}\mathbf Z_k=W_c\mathbf v,
 \qquad W_c>0.
\]
\end{lemma}

\begin{proof}
The finite-type Kesten--Stigum theorem gives almost-sure population convergence because $\mathbf M_{\mathfrak C}$ is primitive and the offspring are bounded \cite{KurtzLyonsPemantlePeres1997}.
The $L^2$ identification in \cite[proof of Lemma~3.3]{NeroliIGSII2026} identifies its scalar limit with $W_c$ under the stated Perron normalisation.
The same lemma and its proof give a density for the nonnegative mass limit, hence $W_c>0$ almost surely.
Applying the argument at each fixed descendant and taking a countable intersection gives the corresponding limits simultaneously throughout the tree.
\end{proof}

For the vertex counts, we use the same birth reward $J_\sigma$ as in \cite[proof of Lemma~3.2]{NeroliIGSII2026}: the number of active internal vertices introduced by one substitution of a state-$\sigma$ cell.
The inherited active-terminal mark is retained when $\sigma=u$; terminal symmetry makes the reward law independent of its choice.

\begin{lemma}\label{lem:vertex-mass-limit}
Almost surely,
\[
 \lim_{n\to\infty}\rho(\mathbf M_{\mathfrak C})^{-n}|V(\mathfrak C_n)|=W_c.
\]
\end{lemma}

\begin{proof}
Fix a live root state $\sigma$ and use the exact birth recursion in \cite[proof of Lemma~3.2]{NeroliIGSII2026}.
The total reward in generation $k$, after subtraction of its conditional mean, has variance at most $C\E|\mathbf Z_k|=O(\rho(\mathbf M_{\mathfrak C})^k)$.
Chebyshev's inequality and Borel--Cantelli show that division by $\rho(\mathbf M_{\mathfrak C})^k$ makes this centred reward tend to zero almost surely.
Its conditional mean is a fixed linear function of $\mathbf Z_k$.
Lemma~\ref{lem:population-limit} and summation of the resulting geometric series therefore give almost-sure convergence of the normalised accumulated rewards.
The limit equals $W_\sigma$ by the $L^2$ convergence already proved in \cite[proof of Lemma~3.3]{NeroliIGSII2026}; including the outer terminals changes the count by at most two.
For the root, $\sigma=c$ gives the asserted limit.
The same argument at every fixed live descendant gives the limit $W_{T_u}^{(u)}$ on one common event of probability one.
\end{proof}

When inspecting either one-arm mask of a $b$-cell, keep its original noncrossing innovation.
Its two components share that randomness but remain separate for intrinsic distance calculations.

\begin{lemma}\label{lem:mass-perron}
Every fixed retained-component cylinder has a compatible limiting normalised vertex mass, with positive mass for connected cylinders.
The two sides of $b$ use their original joint law.
\end{lemma}

\begin{proof}
Apply Lemma~\ref{lem:vertex-mass-limit} at every fixed descendant address.
For one side of $b$, retain its raw innovations and use the $u$ root mask with that side's active-terminal mark.
This gives precisely that component on the existing coupling, with the correct marginal one-arm law.
Its counting limit exists, and the two masked limits sum to the $b$ limit.
Assign shared frontier vertices by any fixed convention.
Different conventions affect finitely many vertices, whose normalised mass tends to zero.
A component at address $u$ therefore has limiting mass $\rho(\mathbf M_{\mathfrak C})^{-|u|}$ times its masked martingale limit divided by the positive root limit.
Finite-depth decomposition identities pass to these ratios and give compatibility.
Connected-cylinder positivity follows from the strict positivity in Lemma~\ref{lem:population-limit}.
\end{proof}

The uniform component estimate needed below is Lemma~\ref{lem:full-mesh} in Appendix~\ref{app:chemical}.
It controls both components of every noncrossing socket as well as all dangling branches.

Equip $V(\mathfrak C_n)$ with the uniform probability measure on its vertices.
The two-marked Gromov--Hausdorff--Prokhorov topology compares compact metric spaces, probability measures and ordered terminal marks in common metric embeddings.

\begin{theorem}\label{thm:full-ghp}
There is a random compact two-marked metric--measure space
\[
 (\mathfrak C,d_{\mathrm{ch}},\mu,\mathfrak v_+,\mathfrak v_-)
\]
with full-support measure such that
\[
 \lim_{n\to\infty}(\mathfrak C_n,(\ell_{\mathfrak C}^{\mathrm{ch}})^{-n}d_{\mathfrak C_n},\frac{1}{|V(\mathfrak C_n)|}\sum_{x\in V(\mathfrak C_n)}\delta_x,\mathfrak v_+,\mathfrak v_-)
 =(\mathfrak C,d_{\mathrm{ch}},\mu,\mathfrak v_+,\mathfrak v_-)
\]
almost surely along the whole sequence in two-marked Gromov--Hausdorff--Prokhorov topology.

\end{theorem}

\begin{proof}

Let $V_k$ contain the retained vertices born by generation $k$.
For $n\ge k$, replace each connected generation-$k$ cell by an edge of length $d_{\mathfrak C_{u,n-k}}(\mathfrak v_+,\mathfrak v_-)$.
Other frontier states supply no through-edge.
Connected dangling branches remain.
Compressing microscopic paths and removing one-attachment excursions, and conversely expanding skeleton edges along terminal geodesics, prove that this weighted distance is exactly the restriction of $d_{\mathfrak C_n}$ to $V_k$.

Each scaled skeleton edge tends to the positive value $c_{\mathrm{ch}}(\ell_{\mathfrak C}^{\mathrm{ch}})^{-k}X_u^{\mathrm{ch}}$.
Finite shortest-path matrices are continuous in their edge lengths.
The scalar recursion makes the limiting metrics compatible on successive $V_k$.
Every fixed skeleton has positive distances between distinct vertices.
We obtain a metric $d_{\mathrm{ch}}$ on $V_*=\bigcup_kV_k$.
Component anchors form $\Gamma_k$-nets both for later approximants and for $V_*$.
Thus its completion $\mathfrak C$ is compact.
Match skeleton labels and project the remaining points to anchors.
For fixed $k$ the finite distance errors vanish with $n$, and the projection errors are at most $\Gamma_k$.
Letting $n\to\infty$ and then $k\to\infty$ proves the marked metric convergence.

Project each component's mass to its planted anchor.
Lemma~\ref{lem:mass-perron} gives convergence of these finitely many masses.
The allocation error at old junction vertices vanishes.
Compatibility and the vanishing mesh make the limiting atomic measures a Cauchy family on $\mathfrak C$, defining $\mu$.
The same transports bound the Prokhorov error by the mesh and the vanishing allocation error.
This proves the asserted measured convergence.

Every skeleton vertex is incident to a connected frontier cell.
This follows from its active coarse connection at birth and the connected replacements of incident edges thereafter.
It has arbitrarily fine connected cylinders with positive limiting mass and vanishing diameter.
Every neighbourhood of each skeleton vertex therefore has positive mass.
Density proves full support.
Theorem~\ref{thm:chemical-scalar} also identifies the terminal distance as $d_{\mathrm{ch}}(\mathfrak v_+,\mathfrak v_-)=c_{\mathrm{ch}}X_\varnothing^{\mathrm{ch}}>0$.

Measurability follows from convergence of the measurable finite marked spaces.

\end{proof}

Construct the intrinsic completion of each retained component using only its internal skeleton paths, with metric in the global normalisation.
For a component $A$, write $d_A$ for this intrinsic metric and $\partial A$ for its planted boundary: two vertices for a connected component and one for a one-root component, including a side of $b$.

\begin{lemma}\label{lem:cell-localisation}
The intrinsic retained components embed continuously in $\mathfrak C$, with distinct components at the same generation intersecting only at shared planted vertices.
\end{lemma}

\begin{proof}
The intrinsic skeleton metric is positive and its completion compact by the component mesh.
Inclusion is $1$-Lipschitz.
For skeleton points in one component, first exit and last entrance give
\begin{equation}\label{eq:cell-route}
 d_{\mathrm{ch}}(x,y)=
 \min\left\{d_A(x,y),
 \min_{a,b\in\partial A}
 [d_A(x,a)+d_{\mathrm{ch}}(a,b)+d_A(b,y)]\right\}.
\end{equation}
Conversely these costs are approached by concatenating the indicated routes.
The equality passes from finite approximants to skeleton limits and then, by continuity and finiteness of the boundary, to completions.
If two distinct intrinsic points had the same image, this formula would force them to be distinct planted boundary points with the same image.
Their positive finite-skeleton distance excludes this.
The analogous first-exit calculation between distinct components shows that intersections are only planted ones.
\end{proof}

The embedding allows paths to leave a component, but every such path must first reach its planted boundary.
\begin{lemma}\label{lem:cell-separation}
For distinct retained components $A,B$ at the same generation and $x\in A$, $y\in B$,
\[
 d_{\mathrm{ch}}(x,y)\ge d_A(x,\partial A)+d_B(y,\partial B).
\]
\end{lemma}

\begin{proof}
For skeleton points, any route from $x$ to $y$ must leave $A$ through $\partial A$ and enter $B$ through $\partial B$.
Its initial and final segments therefore have total length at least the right-hand side.
This bound passes from finite approximants to the skeleton limit, and then to the intrinsic completions by continuity.
\end{proof}

The recursive construction also identifies the intrinsic laws of state-$c$ cells.
\begin{lemma}\label{lem:cell-scaling-law}
Conditional on the raw innovation pairs $(\xi_u^c,\xi_u^N)$ with $|u|<k$, the intrinsic completions of distinct generation-$k$ state-$c$ cells are independent, each with law $(\mathfrak C,(\ell_{\mathfrak C}^{\mathrm{ch}})^{-k}d_{\mathrm{ch}})$.
\end{lemma}

\begin{proof}
The generation-$k$ types are determined by the conditioned innovations.
The descendant innovation trees at distinct generation-$k$ addresses use disjoint coordinates and are conditionally independent.
Each state-$c$ intrinsic completion depends only on its descendant tree, and its edge lengths carry the deterministic factor $(\ell_{\mathfrak C}^{\mathrm{ch}})^{-k}$, giving the scaled root law.
The two sides of a state-$b$ cell use the same noncrossing innovation and retain their joint law.
\end{proof}

The metric in this section is the chemical metric of the full cluster.
The resistance metric and the natural processes require the further arguments below.

Let $\mathcal A_k$ be the collection of retained generation-$k$ components, with the inherited joint masks for the two sides of $b$.

\begin{definition}\label{def:chemical-dimension}
For a compact metric space, let $\mathcal N(r)$ be the minimum number of radius-$r$ balls needed to cover it.
Its Minkowski dimension is the common value of the lower and upper limits of $\log\mathcal N(r)/\log(1/r)$ as $r$ tends to zero, when these agree.
For $(\mathfrak C,d_{\mathrm{ch}})$ we call this the chemical Minkowski dimension.
\end{definition}

\begin{theorem}\label{prop:chemical-minkowski-dimension}
Almost surely,
\[
\dimminkowski(\mathfrak C,d_{\mathrm{ch}})=\frac{\log\rho(\mathbf M_{\mathfrak C})} {\log\ell_{\mathfrak C}^{\mathrm{ch}}}.
\]
\end{theorem}

\begin{proof}
The all-order chemical component-diameter bounds and a frontier union bound give
$\max_{A\in\mathcal A_k}\diam_{d_{\mathrm{ch}}}(A)
 \le C_{\omega,\varepsilon}(\ell_{\mathfrak C}^{\mathrm{ch}})^{-k}e^{\varepsilon k}$.
There are $O_\omega(\rho(\mathbf M_{\mathfrak C})^k)$ components and they cover $\mathfrak C$.
For $0<\varepsilon<\log\ell_{\mathfrak C}^{\mathrm{ch}}$, interpolation gives upper Minkowski dimension at most
$\log\rho(\mathbf M_{\mathfrak C})/(\log\ell_{\mathfrak C}^{\mathrm{ch}}-\varepsilon)$.
Let $\varepsilon\downarrow0$.

Choose an interior vertex $z$ on a first-generation crossing path in a fresh connected cell.
Since $\ell\ge2$, this vertex differs from both planted terminals, and the limiting first-skeleton edge lengths are positive.
Thus its intrinsic distance $\Delta$ to the planted boundary is strictly positive.
Choose deterministic $r_0>0$ with $\Pp_c(\Delta\ge r_0)>0$.
Lemma~\ref{lem:cell-scaling-law}, the connected population limit and Chernoff bounds give at least $c_\omega\rho(\mathbf M_{\mathfrak C})^k$ generation-$k$ connected cells whose corresponding vertices have intrinsic boundary distance at least $r_0(\ell_{\mathfrak C}^{\mathrm{ch}})^{-k}$, eventually.
By Lemma~\ref{lem:cell-separation}, their selected points are pairwise separated by at least $r_0(\ell_{\mathfrak C}^{\mathrm{ch}})^{-k}$ in the full-cluster chemical metric.
A ball of one third that radius contains at most one selected point.
Thus $\mathcal N(r_0(\ell_{\mathfrak C}^{\mathrm{ch}})^{-k}/3)\ge c_\omega\rho(\mathbf M_{\mathfrak C})^k$.
Interpolation gives the lower Minkowski dimension.
\end{proof}

The mass-growth formula of \cite[Theorem~2.14]{NeroliIGSII2026} also gives
\[
 \lim_{n\to\infty}\frac{\log\E|E(\mathfrak C_n)|}{n\log\ell_{\mathfrak C}^{\mathrm{ch}}}
 =\frac{\log\rho(\mathbf M_{\mathfrak C})}{\log\ell_{\mathfrak C}^{\mathrm{ch}}}
 =\dimminkowski(\mathfrak C,d_{\mathrm{ch}}).
\]
Thus the chemical mass exponent agrees with the chemical Minkowski dimension proved above.
The ambient mass dimension uses the same expected mass growth but the ambient scale $\ell$:
\[
 \dimmass(\mathfrak C)=\frac{\log\rho(\mathbf M_{\mathfrak C})}{\log\ell},
 \qquad
 \dimminkowski(\mathfrak C,d_{\mathrm{ch}})=\frac{\log\rho(\mathbf M_{\mathfrak C})}{\log\ell_{\mathfrak C}^{\mathrm{ch}}}.
\]
Since $\ell_{\mathfrak C}^{\mathrm{ch}}\ge\ell$, the chemical Minkowski dimension is no larger than the ambient mass dimension.
Example~\ref{ex:chemical-scales} shows that the inequality can be strict.
No identification of the expected mass exponent with an almost-sure ambient Hausdorff dimension is needed here.

\section{Resistance exponent}\label{sec:resistance}

With unit resistance on each microscopic open edge, we prove the existence of the resistance growth exponent and the convergence of the normalised resistances.
These limits define a resistance metric and a resistance form on the chemical space constructed above; together with the limiting mass measure, the form gives a Dirichlet form.
This extends the DHL result of IGS~I~\cite[Section~4.1]{Neroli2026}, which resolves the Hambly--Kumagai resistance-growth problem~\cite[Section~6(i)]{hambly2010diffusion}, to all classical EIGS hierarchical lattices.
Identifying the speed measure in Section~\ref{sec:diffusion} then connects this form to the original simple random walks and yields their limiting typical spectral dimension.

All scalar probabilities in this section refer to a connected root.
Give every microscopic open edge resistance one and use $\Reff_{\mathfrak C_n}(x,y)$ for the resulting effective resistance, as in IGS~I.
For its coarse terminal path core $H$, now write $F_H(x)$ for effective terminal resistance with edge resistances $x$.
The chemical functional $F_H^{\mathrm{ch}}$ remains the shortest-path cost; $F_H$ will henceforth denote electrical resistance.
The cited cluster representation gives $\Reff_{\mathfrak C_{0}}(\mathfrak v_+,\mathfrak v_-)=1$ and $\Reff_{\mathfrak C_{n+1}}(\mathfrak v_+,\mathfrak v_-)=F_H(\Reff_{\mathfrak C_{n,e}}(\mathfrak v_+,\mathfrak v_-))$, with independent descendant copies conditional on $H$.
Removing pieces outside the terminal path core preserves this resistance, because their components attach at a single vertex and carry no terminal current.
They remain part of the full space and its measure.

Retain $g(u)=\varphi_R(p_c u)/p_c$, and choose one simple path $P(H)$ in each core.
Then $\E|P(H)|\le L$.

\begin{lemma}\label{lem:electrical-eigenlaw}
There are $\lambda_{\mathfrak C}\in[\varphi_R'(p_c),\E|P(H)|]$ and a strictly positive mean-one variable $X$ with $\E(X^p+X^{-p})<\infty$ for every $p>0$, such that $F_H(X_e)$ has the law of $\lambda_{\mathfrak C} X$ for independent child copies, independent of $H$.
\end{lemma}

\begin{proof}
Thomson's principle writes $F_H(x)$ as the infimum of $\sum_ex_eI_e^2$ over unit terminal flows.
Thus it is increasing, concave, positively homogeneous and superadditive.
The unit electrical current satisfies $|I_e|\le1$: orient nonzero currents along decreasing potential and decompose this acyclic flow into terminal path flows with total weight one.
An optimal trial current therefore gives $|F_H(x)-F_H(y)|\le\sum_e|x_e-y_e|$, including the continuous extension to nonnegative resistances.

Every edge lying on all terminal paths carries unit current.
Their expected number is $\varphi_R'(p_c)$ by \textcolor{black}{\cite[proof of Lemma~2.16]{NeroliIGSII2026}}.
The chosen path supplies an upper trial flow.
For independent nonnegative mean-one inputs $Y_e$, therefore, $\varphi_R'(p_c)\le\E F_H(Y_e)\le \E|P(H)|$.

Normalise the output by its mean and denote the resulting law map by $\widehat T$.
The same selected-path multinomial argument as in Lemma~\ref{lem:chemical-moments} gives
\[
 \int_0^\infty x^p\,(\widehat T\nu)(dx)\le(\E|P(H)|/\varphi_R'(p_c)^p)\left(\int_0^\infty x^p\,\nu(dx)\right)
                   +A_p\left(\int_0^\infty x^p\,\nu(dx)\right)^{(p-2)/(p-1)}.
\]
For all sufficiently large integers $p$, choose invariant bounds $K_p\ge1$.
The mean-one laws satisfying all these bounds form a nonempty convex weakly compact set.
Higher moments preserve the mean under limits, the constraints are closed, and the electrical Lipschitz estimate makes $\widehat T$ continuous.
Schauder's theorem yields an eigenlaw with every positive moment and $\lambda_{\mathfrak C}\in[\varphi_R'(p_c),\E|P(H)|]$.

In a network of at most $m$ edges, effective resistance at most $a>0$ forces the edges of resistance at most $ma$ to contain a terminal path.
Otherwise the vertices reachable through such edges determine a terminal cut with conductance strictly less than $1/a$, contradicting the bound on effective conductance.

Independent marking of coarse core edges with probability $u$ gives crossing probability $g(u)$.
This follows by thinning the original Bernoulli-$p_c$ configuration; crossing of the thinned configuration already implies the conditioning event.
Consequently $z=\Pp(X=0)$ obeys $z=g(z)$.
The mean and the inequality $g(u)<u$ for $0<u<1$ give $z=0$.
For $F(t):=\Pp(X\le t)$ the small-edge observation gives $F(t)\le g(F(m\lambda_{\mathfrak C} t))$.
Starting with $F(t_0)<1$ and iterating yields
$F(t_0(m\lambda_{\mathfrak C})^{-k})\le\exp(-c\ell^k)$ after finitely many terms.
Integration proves every inverse moment.
\end{proof}

In the electrical arguments below, $X_u$ always denotes this electrical field, distinct from $X_u^{\mathrm{ch}}$.
Let $\mathcal F_k$ be the first $k$ generations of innovations.

\begin{lemma}\label{lem:electrical-field}
The eigenlaw has a stationary realisation $(X_u)$ measurable with respect to the original innovation tree.
\end{lemma}

\begin{proof}
Independent eigenlaw values on a finite boundary and upward recursion give consistent finite-dimensional distributions.
Their extension satisfies $F_{H_u}(X_{ue})=\lambda_{\mathfrak C} X_u$, with independent descendant subtrees.
For a fixed outer network $G$ with independent resistances $Y_i$, write $Z=F_G(Y)$.
Replacing $Y_i$ by an independent $Y_i'$ gives
$(F_G(Y^{(i)})-F_G(Y))_+\le I_i(Y)^2Y_i'$ by a trial current.
One-sided Efron--Stein gives
$\Var(Z\mid G)\le\E Y^2\,\E[\sum_iI_i(Y)^4\mid G]$.
Since $|I_i|\le1$ and $\sum_iY_iI_i^2=Z$, one has $\sum_iI_i^4\le Z\max_iY_i^{-1}$.
For at most $m^k$ edges and $p\ge2$,
\[
 \E\Var(Z\mid G)\le\E Y^2\,\|Z\|_2m^{k/p}\|Y^{-1}\|_p.
\]
With stationary leaves the output is $\lambda_{\mathfrak C}^kX$.
Choosing $p$ with $m^{1/p}/\lambda_{\mathfrak C}<1$ gives
$\E\Var(X_\varnothing\mid\mathcal F_k)\le C(m^{1/p}/\lambda_{\mathfrak C})^k$.
Thus $X_\varnothing$ is measurable with respect to the complete innovation tree.
A countable intersection gives this at every address.
\end{proof}

\begin{lemma}\label{lem:electrical-sandwich}
\[
 \E[X_\varnothing\mid\mathcal F_n]\le \bigl(\lambda_{\mathfrak C}^{-n}\Reff_{\mathfrak C_{n}}(\mathfrak v_+,\mathfrak v_-)\bigr)\le\frac{(\E X_\varnothing^{-1})}{\E[X_\varnothing^{-1}\mid\mathcal F_n]}\le(\E X_\varnothing^{-1}) \E[X_\varnothing\mid\mathcal F_n].
\]
\end{lemma}

\begin{proof}
Concavity in resistances gives $\E[X_\varnothing\mid\mathcal F_n]\le \bigl(\lambda_{\mathfrak C}^{-n}\Reff_{\mathfrak C_{n}}(\mathfrak v_+,\mathfrak v_-)\bigr)$.
The effective conductance $c\mapsto F_{\mathcal H_n}(c^{-1})^{-1}$ is concave by the Dirichlet principle.
Applying it to the stationary boundary conductances gives $\E[X_\varnothing^{-1}\mid\mathcal F_n]\le(\E X_\varnothing^{-1})/\bigl(\lambda_{\mathfrak C}^{-n}\Reff_{\mathfrak C_{n}}(\mathfrak v_+,\mathfrak v_-)\bigr)$.
Conditional Cauchy--Schwarz gives $\E[X_\varnothing\mid\mathcal F_n]\E[X_\varnothing^{-1}\mid\mathcal F_n]\ge1$.
This proves the sandwich.
\end{proof}

\begin{lemma}\label{lem:electrical-uniform-moments}
For every $p>0$, $\sup_n\E(\bigl(\lambda_{\mathfrak C}^{-n}\Reff_{\mathfrak C_{n}}(\mathfrak v_+,\mathfrak v_-)\bigr)^p+\bigl(\lambda_{\mathfrak C}^{-n}\Reff_{\mathfrak C_{n}}(\mathfrak v_+,\mathfrak v_-)\bigr)^{-p})<\infty$ and $\E(\sup_n\bigl(\lambda_{\mathfrak C}^{-n}\Reff_{\mathfrak C_{n}}(\mathfrak v_+,\mathfrak v_-)\bigr))^p<\infty$.
\end{lemma}

\begin{proof}
The bounds of Lemma~\ref{lem:electrical-sandwich} give the required comparison with conditional martingales.
The two conditional martingales converge to $X_\varnothing$ and $X_\varnothing^{-1}$.
The pathwise and mean assertions follow.
Conditional Jensen gives
$\bigl(\lambda_{\mathfrak C}^{-n}\Reff_{\mathfrak C_{n}}(\mathfrak v_+,\mathfrak v_-)\bigr)^{-p}\le \E[X_\varnothing\mid\mathcal F_n]^{-p}\le\E[X_\varnothing^{-p}\mid\mathcal F_n]$.
For positive moments use $\bigl(\lambda_{\mathfrak C}^{-n}\Reff_{\mathfrak C_{n}}(\mathfrak v_+,\mathfrak v_-)\bigr)\le(\E X_\varnothing^{-1}) \E[X_\varnothing\mid\mathcal F_n]$.
Doob's inequality at an order larger than $\max\{1,p\}$ gives the moment of $\sup_n\bigl(\lambda_{\mathfrak C}^{-n}\Reff_{\mathfrak C_{n}}(\mathfrak v_+,\mathfrak v_-)\bigr)$.
\end{proof}

The same sandwich also gives $1\le\E \bigl(\lambda_{\mathfrak C}^{-n}\Reff_{\mathfrak C_{n}}(\mathfrak v_+,\mathfrak v_-)\bigr)\le(\E X_\varnothing^{-1})$ and
$X_\varnothing\le\displaystyle \liminf \bigl(\lambda_{\mathfrak C}^{-n}\Reff_{\mathfrak C_{n}}(\mathfrak v_+,\mathfrak v_-)\bigr)\le\displaystyle \limsup \bigl(\lambda_{\mathfrak C}^{-n}\Reff_{\mathfrak C_{n}}(\mathfrak v_+,\mathfrak v_-)\bigr)\le(\E X_\varnothing^{-1}) X_\varnothing$ almost surely.

\begin{theorem}\label{thm:resistance-exponent}
There is a deterministic $\lambda_{\mathfrak C}>1$ such that
\[
 \lim_{n\to\infty}\frac{\log\Reff_{\mathfrak C_{n}}(\mathfrak v_+,\mathfrak v_-)}{n}
 =\lim_{n\to\infty}\frac{\E\log\Reff_{\mathfrak C_{n}}(\mathfrak v_+,\mathfrak v_-)}{n}
 =\lim_{n\to\infty}\frac{\log\E\Reff_{\mathfrak C_{n}}(\mathfrak v_+,\mathfrak v_-)}{n}
 =\log\lambda_{\mathfrak C},
\]
where the first limit holds almost surely.
The conclusions hold at every fixed connected address, and $\varphi_R'(p_c)\le\lambda_{\mathfrak C}\le \E|P(H)|$.
\end{theorem}

\begin{proof}

Take the value of $\lambda_{\mathfrak C}$ from Lemma~\ref{lem:electrical-eigenlaw}.
The sandwich bounds $\bigl(\lambda_{\mathfrak C}^{-n}\Reff_{\mathfrak C_{n}}(\mathfrak v_+,\mathfrak v_-)\bigr)$ eventually between strictly positive finite random constants.
Hence
\[
 \displaystyle\lim_{n\to\infty}n^{-1}\log \bigl(\lambda_{\mathfrak C}^{-n}\Reff_{\mathfrak C_{n}}(\mathfrak v_+,\mathfrak v_-)\bigr)=0\textcolor{black}{.}
\]

The bounds $1\le\E \bigl(\lambda_{\mathfrak C}^{-n}\Reff_{\mathfrak C_{n}}(\mathfrak v_+,\mathfrak v_-)\bigr)\le(\E X_\varnothing^{-1})$ and
$\sup_n\E|\log \bigl(\lambda_{\mathfrak C}^{-n}\Reff_{\mathfrak C_{n}}(\mathfrak v_+,\mathfrak v_-)\bigr)|<\infty$, the latter following from $|\log x|\le x+x^{-1}$, prove the first-moment and mean-log limits separately.
Take the countable intersection over connected addresses.

\end{proof}

\begin{corollary}\label{cor:resistance-moment-exponents}
For every fixed $t\in\mathbb R$,
$\displaystyle \lim_{n\to\infty} \frac{1}{n}\log\E\Reff_{\mathfrak C_{n}}(\mathfrak v_+,\mathfrak v_-)^t
=t\log\lambda_{\mathfrak C}$.
\end{corollary}

\begin{proof}
The uniform moment bounds in Lemma~\ref{lem:electrical-uniform-moments} apply.
For $t\ne0$, uniform positive and inverse moments give
$C_{-t}^{-1}\le\E \bigl(\lambda_{\mathfrak C}^{-n}\Reff_{\mathfrak C_{n}}(\mathfrak v_+,\mathfrak v_-)\bigr)^t\le C_t$, since
\[
\E \bigl(\lambda_{\mathfrak C}^{-n}\Reff_{\mathfrak C_{n}}(\mathfrak v_+,\mathfrak v_-)\bigr)^t\,\E \bigl(\lambda_{\mathfrak C}^{-n}\Reff_{\mathfrak C_{n}}(\mathfrak v_+,\mathfrak v_-)\bigr)^{-t}\ge1.
\]
Their logarithms divided by $n$ vanish.
The case $t=0$ is immediate.
\end{proof}

The inequality $\Reff_{\mathfrak C_{n}}(\mathfrak v_+,\mathfrak v_-)\le d_{\mathfrak C_{n}}(\mathfrak v_+,\mathfrak v_-)$ also gives $\lambda_{\mathfrak C}\le\ell_{\mathfrak C}^{\mathrm{ch}}$.
The three exponent identities follow from the two-sided control, without interchanging a logarithm and an expectation.

The exponent does not yet specify the conductance of a fixed limiting cell.
We first select the physical amplitudes on the same innovation tree.

We retain the electrical field $X$, its natural multiplier $\lambda_{\mathfrak C}$, and the normalised unit-input resistances $\bigl(\lambda_{\mathfrak C}^{-n}\Reff_{\mathfrak C_{n}}(\mathfrak v_+,\mathfrak v_-)\bigr)$.
The first estimate selects a scalar multiple of $X$ along a subsequence; the second compares its remaining blocks uniformly over the outer network.

\begin{lemma}\label{lem:electrical-compactness}
The sequence $\bigl(\lambda_{\mathfrak C}^{-n}\Reff_{\mathfrak C_{n}}(\mathfrak v_+,\mathfrak v_-)\bigr)$ is relatively compact in the original $L^2$ space.
Its omega-limit set contains $cX_\varnothing$ for a deterministic $c\in[1,(\E X_\varnothing^{-1})]$.
\end{lemma}

\begin{proof}
For $n\ge k$, expose $\mathcal F_k$ and place independent copies of $\bigl(\lambda_{\mathfrak C}^{-{n-k}}\Reff_{\mathfrak C_{n-k}}(\mathfrak v_+,\mathfrak v_-)\bigr)$ at its leaves.
The output is $\lambda_{\mathfrak C}^k\bigl(\lambda_{\mathfrak C}^{-n}\Reff_{\mathfrak C_{n}}(\mathfrak v_+,\mathfrak v_-)\bigr)$.
Uniform positive and inverse moments in the electrical variance calculation yield
$\sup_{n\ge k}\E\Var(\bigl(\lambda_{\mathfrak C}^{-n}\Reff_{\mathfrak C_{n}}(\mathfrak v_+,\mathfrak v_-)\bigr)\mid\mathcal F_k)\le C\theta^k$ for some $\theta<1$.
For $n<k$ the variance is zero.
The bounded conditional expectations lie in a finite-dimensional space because $\mathcal F_k$ has finitely many atoms.
Their uniform approximation proves relative compactness.

Define $Tf:=\lambda_{\mathfrak C}^{-1}F_H(f_1,\ldots,f_m)$ on functions of the innovation tree.
It is $L^2$-continuous and $\bigl(\lambda_{\mathfrak C}^{-n}\Reff_{\mathfrak C_{n}}(\mathfrak v_+,\mathfrak v_-)\bigr)=T^n1$.
The nonempty compact omega-limit set $\mathcal K$ satisfies $T\mathcal K=\mathcal K$.
Every $h\in\mathcal K$ obeys $X_\varnothing\le h\le(\E X_\varnothing^{-1}) X_\varnothing$ by further almost-sure subsequences.
Put $c:=\sup_{h\in\mathcal K}\esssup(h/X_\varnothing)$.

If $cX_\varnothing\notin\mathcal K$, the nonnegative deficits $cX_\varnothing-h$ have means bounded away from zero and uniformly bounded second moments.
Paley--Zygmund gives $Q(t_0)\le1-\varepsilon$ for some $t_0,\varepsilon>0$, where $Q(t)=\sup_h\Pp((cX_\varnothing-h)\le t)$.
Superadditivity gives $(cX_\varnothing-Th)\ge T(cX_\varnothing-h)$.
Predecessors and the electrical small-edge estimate yield $Q(t)\le g(Q(m\lambda_{\mathfrak C} t))$.
Its iteration gives a finite constant $M$ such that $\sup_h\E (cX_\varnothing-h)^{-1}\le M$.

For fixed $h$, choose deterministic predecessors $h_{-n}$ with $T^nh_{-n}=h$.
Repeated superadditivity gives $(cX_\varnothing-h)\ge T^n(cX_\varnothing-h_{-n})$.
The leaf deficits are independent conditional on $\mathcal F_n$, with reciprocal mean at most $M$.
Conductance concavity gives
\[
 \E[(T^n(cX_\varnothing-h_{-n}))^{-1}\mid\mathcal F_n]\le M/\bigl(\lambda_{\mathfrak C}^{-n}\Reff_{\mathfrak C_{n}}(\mathfrak v_+,\mathfrak v_-)\bigr).
\]
Conditional Jensen therefore implies
$\E[(cX_\varnothing-h)\mid\mathcal F_n]\ge \bigl(\lambda_{\mathfrak C}^{-n}\Reff_{\mathfrak C_{n}}(\mathfrak v_+,\mathfrak v_-)\bigr)/M\ge \E[X_\varnothing\mid\mathcal F_n]/M$.
Martingale convergence gives $(cX_\varnothing-h)\ge X_\varnothing/M$ almost surely.
Thus every $h$ has essential ratio at most $c-1/M$, contradicting the definition of $c$.
The same deterministic $M$ works for all $h$; no uncountable common null set is needed for the essential-supremum conclusion.
\end{proof}

For positive random variables $U,V$ with the required moments, define
\[
\varepsilon(U,V):=
 \frac{\E[(U-V)_+\max\{1,V^{-2}\}]}{\E\min\{V,V^{-1}\}}\textcolor{black}{.}
\]
\begin{lemma}\label{lem:electrical-comparison}
Let a finite connected two-terminal graph $G$ carry independent identically distributed pairs of positive edge resistances $(U_e,V_e)$.
Then
\[
 \E(F_G(U)-F_G(V))_+
 \le \varepsilon(U,V)\E F_G(V).
\]
The reversed bound holds after exchanging $U,V$.
The coefficient does not depend on the graph.
\end{lemma}

\begin{proof}

Hold all other resistances fixed and impose unit terminal current.
For a bridge, the current is constant as its resistance varies.
For a nonbridge, deleting the edge leaves a connected network, and the Thevenin formula gives $I_e(r)=A_e/(r+B_e)$ with $B_e\ge0$.
Thus, for $\vartheta_e(r)=I_e(r)^2$,
\[
 \min\{1,r^{-2}\}\vartheta_e(1)
 \le\vartheta_e(r)\le
 \max\{1,r^{-2}\}\vartheta_e(1).
\]
This also covers constant bridge currents and zero currents.

The optimal $V$ current is a trial flow for $U$, so
$(F_G(U)-F_G(V))_+\le\sum_e(U_e-V_e)_+\vartheta_e(V_e)$.
Conditioning on the other pairs bounds the expected summand above by
$\vartheta_e(1)\E[(U-V)_+\max\{1,V^{-2}\}]$.
The expected $V_e\vartheta_e(V_e)$ is at least
$\vartheta_e(1)\E\min\{V,V^{-1}\}$.
Eliminate $\vartheta_e(1)$ and sum, using $\sum_eV_eI_e(V)^2=F_G(V)$.

\end{proof}

\begin{theorem}\label{thm:electrical-amplitudes}
There is a deterministic $c_{\mathrm R}>0$ such that
$\displaystyle\lim_{n\to\infty}\lambda_{\mathfrak C}^{-n}\Reff_{\mathfrak C_{u,n}}(\mathfrak v_+,\mathfrak v_-)=c_{\mathrm R}X_u$ almost surely at every connected address and in every finite positive $L^p$.
\end{theorem}

\begin{proof}
The compactness and comparison estimates are Lemmas~\ref{lem:electrical-compactness} and~\ref{lem:electrical-comparison}.
Choose a deterministic subsequence $j_r$ with $\displaystyle\lim_{r\to\infty}\bigl(\lambda_{\mathfrak C}^{-{j_r}}\Reff_{\mathfrak C_{j_r}}(\mathfrak v_+,\mathfrak v_-)\bigr)=cX_\varnothing$ in $L^2$.
For selected $j$, put $\varepsilon_j=\varepsilon(\bigl(\lambda_{\mathfrak C}^{-j}\Reff_{\mathfrak C_{j}}(\mathfrak v_+,\mathfrak v_-)\bigr),cX_\varnothing)$ and $\varepsilon_j'=\varepsilon(cX_\varnothing,\bigl(\lambda_{\mathfrak C}^{-j}\Reff_{\mathfrak C_{j}}(\mathfrak v_+,\mathfrak v_-)\bigr))$.
Both tend to zero along this subsequence.
Cauchy--Schwarz and uniform fourth inverse moments control the numerators.
The reversed denominator tends to a positive limit by bounded continuity of $\min\{x,x^{-1}\}$ on $(0,\infty)$.

Conditional on $\mathcal F_n$, the boundary pairs $(\bigl(\lambda_{\mathfrak C}^{-j}\Reff_{\mathfrak C_{j}}(\mathfrak v_+,\mathfrak v_-)\bigr),cX)$ are independent across edges.
Their normalised outputs are $\bigl(\lambda_{\mathfrak C}^{-{n+j}}\Reff_{\mathfrak C_{n+j}}(\mathfrak v_+,\mathfrak v_-)\bigr)$ and $cX_\varnothing$.
Lemma~\ref{lem:electrical-comparison} gives
\[
 \E[(\bigl(\lambda_{\mathfrak C}^{-{n+j}}\Reff_{\mathfrak C_{n+j}}(\mathfrak v_+,\mathfrak v_-)\bigr)-cX_\varnothing)_+\mid\mathcal F_n]\le\varepsilon_jc\E[X_\varnothing\mid\mathcal F_n],
\]
and, in the other direction,
\[\begin{aligned}
 &\E[(cX_\varnothing-\bigl(\lambda_{\mathfrak C}^{-{n+j}}\Reff_{\mathfrak C_{n+j}}(\mathfrak v_+,\mathfrak v_-)\bigr))_+\mid\mathcal F_n]
 \\ &\le\varepsilon_j'\E[\bigl(\lambda_{\mathfrak C}^{-{n+j}}\Reff_{\mathfrak C_{n+j}}(\mathfrak v_+,\mathfrak v_-)\bigr)\mid\mathcal F_n]
 \\ &\le\varepsilon_j'(\E \bigl(\lambda_{\mathfrak C}^{-j}\Reff_{\mathfrak C_{j}}(\mathfrak v_+,\mathfrak v_-)\bigr))\bigl(\lambda_{\mathfrak C}^{-n}\Reff_{\mathfrak C_{n}}(\mathfrak v_+,\mathfrak v_-)\bigr).
\end{aligned}\]
The final inequality is conditional concavity.

For either positive-part error, the independent-copy variance identity and Lipschitz continuity of the positive part give conditional variance in expectation at most
\[
 2\E\Var(\bigl(\lambda_{\mathfrak C}^{-{n+j}}\Reff_{\mathfrak C_{n+j}}(\mathfrak v_+,\mathfrak v_-)\bigr)\mid\mathcal F_n)
 +2c^2\E\Var(X_\varnothing\mid\mathcal F_n)
 \le C\theta^n.
\]
Chebyshev and Borel--Cantelli show that each error differs from its conditional mean by a quantity tending to zero almost surely for fixed selected $j$.
Since $\displaystyle\lim_{n\to\infty}\E[X_\varnothing\mid\mathcal F_n]=X_\varnothing$, $\E \bigl(\lambda_{\mathfrak C}^{-j}\Reff_{\mathfrak C_{j}}(\mathfrak v_+,\mathfrak v_-)\bigr)\le(\E X_\varnothing^{-1})$, and $\displaystyle \limsup \bigl(\lambda_{\mathfrak C}^{-n}\Reff_{\mathfrak C_{n}}(\mathfrak v_+,\mathfrak v_-)\bigr)\le(\E X_\varnothing^{-1}) X_\varnothing$, the upper limiting errors are at most $\varepsilon_jcX_\varnothing$ and $\varepsilon_j'(\E X_\varnothing^{-1})^2X_\varnothing$.
Intersect over the countably many selected blocks and let $j\to\infty$ through them.
Each fixed shift covers the whole eventual sequence, proving convergence with $c_{\mathrm R}=c$.

Uniform positive moments give convergence in every finite positive $L^p$. Take the countable address intersection.
\end{proof}

\begin{corollary}\label{cor:electrical-moment-limits}
For every fixed $t\in\mathbb R$,
$\displaystyle\lim_{n\to\infty}\lambda_{\mathfrak C}^{-nt}\E\Reff_{\mathfrak C_{n}}(\mathfrak v_+,\mathfrak v_-)^t=c_{\mathrm R}^{\,t}\E X_\varnothing^t$.
\end{corollary}

\begin{proof}
Theorem~\ref{thm:electrical-amplitudes} gives almost-sure convergence to $c_{\mathrm R}X_\varnothing$. The uniform positive and inverse moment bounds in Lemma~\ref{lem:electrical-uniform-moments} give uniform integrability at every fixed real power, so expectations converge.
\end{proof}

\begin{theorem}\label{thm:electrical-uniqueness}
The mean-one eigenlaw and eigenvalue are unique among positive eigenlaws with every positive and inverse moment.
\end{theorem}

\begin{proof}
Apply Theorem~\ref{thm:electrical-amplitudes} to the natural eigenlaw.
For any other eigenlaw in the stated class, the sandwich bounds the physical mean between fixed multiples of its eigenvalue to the power $n$.
Its eigenvalue must therefore equal $\lambda_{\mathfrak C}$.
\textcolor{black}{The construction of Lemma~\ref{lem:electrical-field} realises its mean-one endogenous field on the same innovation tree.
The sandwich bounds apply to this field, so the endpoint argument of Lemma~\ref{lem:electrical-compactness} places a deterministic positive multiple of it in the omega-limit set of the same unit-input resistance sequence.
By Theorem~\ref{thm:electrical-amplitudes}, that set is the singleton $\{c_{\mathrm R}X_\varnothing\}$.
Taking expectations identifies the multiple as $c_{\mathrm R}$, and hence the two mean-one fields and their laws agree.}
\end{proof}

All component resistances in the following argument are initially calculated within the component.
Adding an exterior network can only decrease them.

\begin{lemma}\label{lem:full-resistance-mesh}
Almost surely there are exponentially decreasing bounds $\Gamma_k$ with $\displaystyle\lim_{k\to\infty}\Gamma_k=0$ for the diameters of every retained generation-$k$ component in $\lambda_{\mathfrak C}^{-n}\Reff_{\mathfrak C_n}$, simultaneously for all $n\ge k$.
Each component's normalised diameter supremum has every positive moment.
The assertion includes both sides of $b$ and all one-terminal decorations.
\end{lemma}

\begin{proof}
At connected $u$, put $V_u=\sup_{r\ge0}\lambda_{\mathfrak C}^{-r}\Reff_{\mathfrak C_{u,r}}(\mathfrak v_+,\mathfrak v_-)$.
It has all positive moments by Lemma~\ref{lem:electrical-uniform-moments}.
Define $H_{u,r}$ as the maximum resistance distance to available terminals in each vertex's own component, using the nearer terminal for $c$ and the maximum of the separate radii for $b$.
The seed radii vanish.
Reach a child terminal and follow a simple active coarse route to a parent terminal.
Resistance triangle inequalities and Rayleigh monotonicity give
\[
 H_{u,r+1}\le\max_{i\text{ retained}}H_{ui,r}
                +\sum_{i:T_{ui}=c}\Reff_{\mathfrak C_{ui,r}}(\mathfrak v_+,\mathfrak v_-).
\]
Apply this separately to each side of $b$.

As in Lemma~\ref{lem:full-mesh},
\[
 \sup_r\lambda_{\mathfrak C}^{-r}H_{u,r}
 \le A_u:=m\sum_{j\ge1}\lambda_{\mathfrak C}^{-j}
          \max_{\substack{|w|=j\text{ below }u\\T_w=c}}V_w.
\]
For arbitrarily large $p$, its $L^p$ norm is bounded by
$mC_p\sum_{j\ge1}(m^{1/p}/\lambda_{\mathfrak C})^j$.
Thus $G_u=2A_u+\mathbf1_{\{T_u=c\}}V_u$ bounds every normalised component diameter and has all positive moments.
The maximum over at most $m^k$ frontier sockets gives
$\Pp(\lambda_{\mathfrak C}^{-k}\max G_u>\eta^k)\le C[m/(\lambda_{\mathfrak C}\eta)^p]^k$
for $m^{1/p}/\lambda_{\mathfrak C}<\eta<1$.
Borel--Cantelli proves the result.
\end{proof}

Retain the increasing finite skeletons $V_k$ and their union $V_*$ from Section~\ref{sec:geometry}.

\begin{theorem}\label{thm:resistance-metric}
There is a compact resistance metric $R_*$ on the chemical carrier $\mathfrak C$ such that
$(V(\mathfrak C_n),\lambda_{\mathfrak C}^{-n}\Reff_{\mathfrak C_n})$ converges almost surely to $(\mathfrak C,R_*)$ in marked Gromov--Hausdorff topology.
The metrics $R_*$ and $d_{\mathrm{ch}}$ induce the same topology.
\end{theorem}

\begin{proof}
Work on the event of probability one where the amplitude limits at all addresses and both component mesh bounds hold.
For each connected cell $u$, let $s_u,t_u$ be its planted terminals, and define $\mathcal E_n^0(f):=\sum_{\{x,y\}\in E(\mathfrak C_n)}(f(x)-f(y))^2$, counting each undirected edge once.
Eliminating the interiors of generation-$k$ cells gives
\[
 \Tr_{V_k}(\lambda_{\mathfrak C}^n\mathcal E_n^0)(f)
 =\sum_{\substack{|u|=k\\T_u=c}}
 \frac{\lambda_{\mathfrak C}^n}{\Reff_{\mathfrak C_{u,n-k}}(\mathfrak v_+,\mathfrak v_-)}
       (f(s_u)-f(t_u))^2.
\]
A one-root component contributes zero because its boundary value extends constantly over the component.
This applies separately to the two components of a $b$-cell; their interiors reappear on finer skeletons, while connected cells in dangling branches already contribute to the displayed sum.

Theorem~\ref{thm:electrical-amplitudes} gives the limiting finite form
\begin{equation}\label{eq:limiting-trace}
 \mathcal E^{(k)}(f):=
 \sum_{\substack{|u|=k\\T_u=c}}
 \frac{\lambda_{\mathfrak C}^k}{c_{\mathrm R}X_u}
       (f(s_u)-f(t_u))^2.
\end{equation}
The skeleton network is connected and its conductances are positive.
On such a finite network, tracing to a fixed vertex set depends continuously on the conductances, as follows from the finite energy-minimisation problem.
We may therefore pass to the limit in the finite-level trace identity to obtain $\Tr_{V_k}\mathcal E^{(k+1)}=\mathcal E^{(k)}$.
Effective resistance is preserved by tracing, so these forms define a single resistance metric $R_*$ on $V_*$.

To prove compactness, pass the bound in Lemma~\ref{lem:full-resistance-mesh} to each fixed finite set of limiting born vertices.
The generation-$k$ component anchors then form a finite $\Gamma_k$-net in $V_*$, and hence in its resistance completion $\mathfrak C_{\mathrm R}$.
Since $\displaystyle\lim_{k\to\infty}\Gamma_k=0$, this completion is compact.
Project the finite approximants and the completion to the same labelled anchors.
For fixed $k$, their anchor resistance matrices converge, while each projection moves points by at most $\Gamma_k$.
First taking the limit in $n$ and then in $k$ proves marked resistance Gromov--Hausdorff convergence to $\mathfrak C_{\mathrm R}$.

We next construct the form needed to identify this completion with the chemical carrier.
For a function on $V_*$, define $\mathcal E(f):=\sup_k\mathcal E^{(k)}(f|_{V_k})$; trace compatibility makes these finite energies increasing.
The finite resistance inequalities give $|f(x)-f(y)|^2\le R_*(x,y)\mathcal E(f)$, so every finite-energy function extends continuously to $\mathfrak C_{\mathrm R}$.
Let $\mathcal D$ consist of these extensions.
Harmonic extension from $V_k$ to successive skeletons preserves energy, and uniqueness of each finite minimiser makes the extensions agree on their common vertices.
Clipping a function to $[0,1]$ decreases every finite-network energy; taking the supremum therefore gives the Markov property for $\mathcal E$.
The compatible-sequence theorem \cite[Theorem~2.13]{Kigami2003} supplies the resistance-form axioms on $V_*$ and identifies the associated metric with $R_*$.
Theorem~2.12 of the same reference extends the form to $\mathfrak C_{\mathrm R}$, and Theorem~2.14 shows that its traces on the dense skeletons determine it uniquely.

It remains to identify the two completions.
Fix $\varepsilon>0$ and choose a generation whose resistance mesh is at most $\varepsilon$.
By Lemma~\ref{lem:cell-localisation}, its chemical component closures meet only at shared planted vertices.
There are finitely many such compact sets, so the distance between any disjoint pair has a positive lower bound, when disjoint pairs exist.
Consequently, sufficiently chemically close born vertices lie in the same component or in two components sharing a planted vertex, and their resistance distance is at most $2\varepsilon$.
Thus the identity on $V_*$ is uniformly continuous from the chemical metric to $R_*$ and extends to a continuous map $\mathfrak C\to\mathfrak C_{\mathrm R}$.
Its image is compact and contains the dense set $V_*$, so the map is surjective.

For injectivity, take distinct $x,y\in\mathfrak C$ and a generation whose chemical mesh is less than $d_{\mathrm{ch}}(x,y)/4$.
The components containing $x$ and $y$ have disjoint planted boundary sets, since a shared boundary point would put $x$ and $y$ at distance less than $d_{\mathrm{ch}}(x,y)/2$.
Prescribe zero on the first boundary, one on the second, and arbitrary values at the other skeleton vertices, then extend harmonically to the finer skeletons.
The extension has finite energy and is constant on each of the two components: constant boundary data have the constant extension as their energy minimiser.
Its resistance-continuous extension therefore separates the images of $x$ and $y$.
The map is a continuous bijection between compact metric spaces, hence a homeomorphism.
Identifying $\mathfrak C_{\mathrm R}$ with $\mathfrak C$ gives the claimed metric and common topology; the same labelled anchors give the marked convergence on this carrier.
\end{proof}

\begin{theorem}\label{thm:full-resistance-form}
The finite electrical traces determine a unique regular strongly local resistance form $(\mathcal E,\mathcal D)$ on $\mathfrak C$.
\end{theorem}

\begin{proof}
The compatible-trace construction in the proof of Theorem~\ref{thm:resistance-metric} supplies the resistance form and proves uniqueness from its dense traces.

Compact resistance forms are regular \cite[Corollary~6.4]{Kigami2012}.
Here this also follows by harmonic interpolation of a continuous function's values on $V_k$: the maximum principle and vanishing component mesh give uniform approximation.

Suppose $f,g\in\mathcal D$ and $f$ is constant on a neighbourhood $U$ of $\supp g$.
If $\mathfrak C\setminus U$ is nonempty, its chemical distance from $\supp g$ is positive.
For large $k$, every trace edge has endpoints in a component smaller than this distance.
If $g$ differs at the endpoints, one lies in $\supp g$ and both lie in $U$, so the difference of $f$ is zero.
Every mixed trace summand vanishes.
Passing to the limit gives $\mathcal E(f,g)=0$.
If $U=\mathfrak C$ the conclusion is immediate.
Thus the form is strongly local.
\end{proof}

For later use, the construction also gives an exact decomposition into the energies of the retained components at every fixed generation.
Write $\mathcal E_A^{\mathrm g}$ for a component's contribution in global normalisation.
Finite traces split over disjoint socket interiors; monotone passage to their limits gives
$\mathcal E(f)=\sum_A\mathcal E_A^{\mathrm g}(f|_A)$.
The intrinsic resistance inequality in $A$ uses this local energy.
A finite-energy function on $A$ vanishing on its planted boundary extends by zero to $\mathcal D$, as is first seen on every finite trace.
A connected generation-$k$ component has the root energy law multiplied by $\lambda_{\mathfrak C}^k$.
For a $b$-socket these statements apply separately to its two components without making their laws independent.
\textcolor{black}{The intrinsic diameter bounds also pass to these limiting components.
For each fixed pair of internal skeleton vertices, the amplitude limits and continuity of finite-network resistance give convergence within the component.
The all-depth intrinsic bound of Lemma~\ref{lem:full-resistance-mesh} therefore passes first to every skeleton pair and then, by density, to the intrinsic completion, with the same moment bounds.}

\section{Diffusion limit and spectral dimension}\label{sec:diffusion}

Combining the chemical geometry, mass measure and resistance form, we construct a diffusion and prove that the simple random walks converge to it after suitable space and time rescaling.
Local mass and resistance estimates then determine its spectral dimension at $\mu$-almost every point, in terms of the chemical Minkowski and resistance dimensions.

The degree probability measure $\frac{1}{2|E(\mathfrak C_n)|}\sum_{x\in V(\mathfrak C_n)}\deg_{\mathfrak C_n}(x)\delta_x$ is reversible for simple random walk.

\begin{lemma}\label{lem:degree-measure}
Both $\frac{1}{|V(\mathfrak C_n)|}\sum_{x\in V(\mathfrak C_n)}\delta_x$ and $\frac{1}{2|E(\mathfrak C_n)|}\sum_{x\in V(\mathfrak C_n)}\deg_{\mathfrak C_n}(x)\delta_x$ converge to the probability measure $\mu$, in the common chemical and resistance realisations.
\end{lemma}

\begin{proof}
With the Perron normalisation of Section~\ref{sec:geometry},
$\displaystyle\lim_{n\to\infty}\rho(\mathbf M_{\mathfrak C})^{-n}|E(\mathfrak C_n)|=v_cW_c$ and $\displaystyle\lim_{n\to\infty}\rho(\mathbf M_{\mathfrak C})^{-n}|V(\mathfrak C_n)|=W_c$ almost surely.

Each microscopic $c$ seed contributes one edge; every other seed contributes none.
Thus $|E(\mathfrak C_n)|=Z_{n,c}$ and the edge limit is the Perron limit.
The almost-sure vertex limit is Lemma~\ref{lem:vertex-mass-limit}, with the mass normalisation of IGS~II.
The same edge and birth rewards work inside every fixed component, with the inherited mask for a side of $b$.

Assign each microscopic edge to its unique frontier component and move both endpoint half-masses to its anchor.
The projected degree probability mass is exactly its edge count divided by $|E(\mathfrak C_n)|$, even at old high-degree junctions.
For counting measure use first-birth ownership.
The allocation error involves only finitely many frontier boundary vertices and vanishes after normalisation.
Component edge and vertex ratios tend to the same descendant martingale ratios, because their deterministic reward constants cancel.
These ratios define $\mu$.
Both transports have displacement bounded by the component mesh in either metric.
First let $n\to\infty$ and then let the frontier depth tend to infinity.
Full support was proved in Section~\ref{sec:geometry}.

\end{proof}

With speed measure $\mu$, the regular strongly local resistance form defines a conservative recurrent diffusion $B$.
Indeed, $1\in\mathcal D$ and $\mathcal E(1)=0$.
Compactness and full support give its regular Dirichlet form, and positive capacity of each point defines the process from every starting point \cite[Theorems~9.4 and 9.9]{Kigami2012}.
Strong locality gives continuous paths.

On one event of probability one, the full clusters converge to $(\mathfrak C,d_{\mathrm{ch}},\mu)$ and $(\mathfrak C,R_*,\mu)$ in their respective marked Gromov--Hausdorff--Prokhorov topologies with common component correspondences.

\begin{theorem}\label{thm:natural-process}
On one event of probability one, for every starting-vertex sequence converging to $x\in \mathfrak C$, the conditional laws of
$(S^n_{\lfloor2|E(\mathfrak C_n)|\lambda_{\mathfrak C}^nt\rfloor})_{t\ge0}$ converge weakly to $B$ started at $x$, in the Skorokhod $J_1$ topology of a common chemical realisation.
Here $S^n$ is discrete simple random walk.
\end{theorem}

\begin{proof}
Fix an environment with all the address-wise limits, both uniform mesh bounds and the mass limits.
Theorem~\ref{thm:full-resistance-form} and Lemma~\ref{lem:degree-measure} give compact measured resistance convergence.
Croydon's theorem \cite[Theorem~1.2]{Croydon2018} applies.
Its escape condition is automatic: the compact convergent spaces have uniformly bounded diameters, so sufficiently large balls contain the entire spaces.
The associated process laws converge weakly in a common resistance embedding.

Use the common component correspondences to realise the chemical spaces in one complete separable space.
They have vanishing resistance distortion and chemical embedding error.
The identity from the limiting resistance space to its chemical carrier is continuous.
These are the spatial conditions of \cite[Theorem~7.1]{Croydon2018}, giving weak $J_1$ convergence of the embedded processes with continuous limiting paths.
The argument is deterministic after fixing the environment.
Every convergent starting sequence supplies a convergent pointed sequence, so the same environment event covers all such sequences.

For the discrete transition matrix $\mathbf P_n$, direct summation gives

\[
 \langle f-\mathbf P_n f,f\rangle_{\frac{1}{2|E(\mathfrak C_n)|}\sum_{x\in V(\mathfrak C_n)}\deg_{\mathfrak C_n}(x)\delta_x}=\mathcal E_n^0(f)/(2|E(\mathfrak C_n)|)\textcolor{black}{.}
\]
Thus the process for $\lambda_{\mathfrak C}^n\mathcal E_n^0$ and $\frac{1}{2|E(\mathfrak C_n)|}\sum_{x\in V(\mathfrak C_n)}\deg_{\mathfrak C_n}(x)\delta_x$ has generator $2|E(\mathfrak C_n)|\lambda_{\mathfrak C}^n(\mathbf P_n-\mathbf I)$.
The factor two comes from counting undirected edges once in the energy.

Attach independent mean-one exponential holding times $\xi_i$.
Set $\sigma_k:=\sum_{i\le k}\xi_i$, $N(t):=\max\{k:\sigma_k\le t\}$ and $Y_n(t):=S^n_{N((2|E(\mathfrak C_n)|\lambda_{\mathfrak C}^n)t)}$.
Since $\displaystyle\lim_{n\to\infty}2|E(\mathfrak C_n)|\lambda_{\mathfrak C}^n=\infty$, both $N((2|E(\mathfrak C_n)|\lambda_{\mathfrak C}^n)t)/(2|E(\mathfrak C_n)|\lambda_{\mathfrak C}^n)$ and
$\gamma_n(t):=\sigma_{\lfloor (2|E(\mathfrak C_n)|\lambda_{\mathfrak C}^n)t\rfloor}/(2|E(\mathfrak C_n)|\lambda_{\mathfrak C}^n)$ tend uniformly on compact intervals in probability to $t$.
The identity $Y_n(\gamma_n(t))=S^n_{\lfloor (2|E(\mathfrak C_n)|\lambda_{\mathfrak C}^n)t\rfloor}$ holds.
Convergence to a continuous limiting path law and the vanishing clock error give the same weak $J_1$ limit for the discrete walk.
No aperiodicity is required.

\textcolor{black}{Almost-sure metric--measure convergence together with these conditional path-law limits gives quenched Gromov--Hausdorff--Prokhorov--Skorokhod convergence under the clock $2|E(\mathfrak C_n)|\lambda_{\mathfrak C}^n$.}
It is not an assertion of almost-sure convergence of the original walk sample paths.
Integration over environments gives the corresponding annealed joint convergence.
\end{proof}

\begin{corollary}\label{cor:continuous-time-walk}
Under the hypotheses of Theorem~\ref{thm:natural-process}, rate-one continuous-time simple random walk, accelerated by $2|E(\mathfrak C_n)|\lambda_{\mathfrak C}^n$, has the same conditional diffusion limit.
\end{corollary}

\begin{proof}
The Poissonised walk in the proof of Theorem~\ref{thm:natural-process} is the process associated with the rescaled energy and the degree probability measure. Its convergence is obtained there before removing the exponential holding times.
\end{proof}

\begin{corollary}\label{cor:variable-speed-walk}
Under the hypotheses of Theorem~\ref{thm:natural-process}, the continuous-time walk with rate one along each incident edge, accelerated by $|V(\mathfrak C_n)|\lambda_{\mathfrak C}^n$, has the same conditional diffusion limit.
\end{corollary}

\begin{proof}
The generator is $L_nf(x)=\sum_{y\sim x}(f(y)-f(x))$. Its energy for the uniform vertex measure is $\mathcal E_n^0/|V(\mathfrak C_n)|$. \textcolor{black}{Acceleration by $|V(\mathfrak C_n)|\lambda_{\mathfrak C}^n$ gives $\lambda_{\mathfrak C}^n\mathcal E_n^0$,} and the same measured-resistance convergence and spatial comparison used in Theorem~\ref{thm:natural-process} apply.
\end{proof}

\begin{corollary}\label{cor:deterministic-clock}
Under deterministic acceleration $(\rho(\mathbf M_{\mathfrak C})\lambda_{\mathfrak C})^n$, simple random walk converges conditionally to $(B_{t/(2v_cW_c)})_{t\ge0}$, equivalently to the resistance-form process with speed measure $2v_cW_c\mu$.
For the variable-speed walk, replace $2v_cW_c$ by $W_c$.
\end{corollary}

\begin{proof}
The ratios of deterministic acceleration to the exact clocks tend to $(2v_cW_c)^{-1}$ and $W_c^{-1}$.
These positive limits give convergent linear changes of time.
Multiplying a speed measure by a positive constant divides the generator by that constant.
\end{proof}

The deterministic exponential clock is $(\rho(\mathbf M_{\mathfrak C})\lambda_{\mathfrak C})^n$, with a random scalar speed unless the finite-volume clock is used.
This clock statement does not identify an all-scale expected exit-time dimension.
The construction uses microscopic unit resistances throughout.
It should be distinguished from the resistance-normalised diffusion of \cite{hambly2010diffusion}, for which the local electrical normalisation is part of the model.

\subsection*{Spectral dimension}\label{sec:spectral}

The heat kernel below is relative to the probability speed measure $\mu$.

\begin{definition}\label{def:spectral-dimension}
Following the diffusion notation of IGS~I, at $x\in \mathfrak C$ the pointwise spectral dimension is
\[
 \dimspec^{(\mathrm N)}(\mathfrak C:x):=\lim_{t\downarrow0}
 \frac{2\log p_t(x,x)}{\log(1/t)},
\]
when this limit exists.
The resistance dimension relative to the chemical scale is
$\dimresis(\mathfrak C,d_{\mathrm{ch}}):=\log\lambda_{\mathfrak C}/\log\ell_{\mathfrak C}^{\mathrm{ch}}$.
\end{definition}

The auxiliary heat-kernel estimates are proved in Appendix~\ref{app:spectral}.
In particular, the heat kernel is jointly continuous by Lemma~\ref{lem:heat-kernel-continuity}, $\mu(V_*)=0$ by Lemma~\ref{lem:skeleton-zero-mass}, and at $\mu$-almost every point the containing generation-$k$ component has mass $\rho(\mathbf M_{\mathfrak C})^{-k}e^{o(k)}$ and intrinsic resistance diameter at most $\lambda_{\mathfrak C}^{-k}e^{o(k)}$.
The collar estimate also provides smaller cells separated from its planted boundary.
These facts give matching resolvent bounds at the natural clock scale.

\begin{theorem}\label{thm:heat-kernels}\label{lem:spectral-typical}
Almost surely, for $\mu$-almost every $x\in \mathfrak C$,
\begin{equation}\label{eq:spectral-dimension-form}
 \dimspec^{(\mathrm N)}(\mathfrak C:x)=\frac{2\log\rho(\mathbf M_{\mathfrak C})}{\log(\rho(\mathbf M_{\mathfrak C})\lambda_{\mathfrak C})}
 =\frac{2\dimminkowski(\mathfrak C,d_{\mathrm{ch}})}{\dimminkowski(\mathfrak C,d_{\mathrm{ch}})+\dimresis(\mathfrak C,d_{\mathrm{ch}})}.
\end{equation}
\end{theorem}

\begin{proof}

At the points of Lemma~\ref{lem:spectral-collars}, the containing component has mass $\rho(\mathbf M_{\mathfrak C})^{-k}e^{o(k)}$ and intrinsic resistance diameter at most $\lambda_{\mathfrak C}^{-k}e^{o(k)}$.
Equation~\eqref{eq:resolvent-component} gives
$G_{(\rho(\mathbf M_{\mathfrak C})\lambda_{\mathfrak C})^k}(x,x)\le\lambda_{\mathfrak C}^{-k}e^{o(k)}$.

Fix rational $\eta>0$ and $n=\lceil\eta k\rceil$.
Eventually the generation-$(k+n)$ component $D$ containing $x$ does not meet the boundary of $A=A_k(x)$.
On the refined skeleton in $A$, give the roots of $D$ value one and all other vertices value zero.
Extend harmonically within the descendant components and by zero outside $A$.
The finite trace construction gives $f\in\mathcal D$, with $0\le f\le1$ and $f(x)=1$.
For a side of $b$, the other side has its separate boundary prescription on the shared continuation.
There are at most $m^n$ connected refined trace edges inside the parent socket.
Therefore
\begin{equation}\label{eq:typical-cutoff}
 \mathcal E(f)\le
 \frac{m^n\lambda_{\mathfrak C}^{k+n}}{c_{\mathrm R}}
 \max_{\substack{|u|=k+n\\T_u=c}}X_u^{-1},
 \qquad \|f\|_2^2\le\mu(A).
\end{equation}
Equation~\eqref{eq:spectral-maxima} bounds the energy by
$C_{\omega,\varepsilon}\lambda_{\mathfrak C}^k
 e^{n\log(m\lambda_{\mathfrak C})+\varepsilon(k+n)}$,
whereas $(\rho(\mathbf M_{\mathfrak C})\lambda_{\mathfrak C})^k\|f\|_2^2\le\lambda_{\mathfrak C}^ke^{o(k)}$.

Insert $f$ in \eqref{eq:resolvent-variational}.
The exponential loss per generation is at most
$\eta\log(m\lambda_{\mathfrak C})+\varepsilon(1+\eta)$.
Take common events for rational parameters and then let $\eta,\varepsilon\downarrow0$.
This proves the lower bound $\lambda_{\mathfrak C}^{-k}e^{-o(k)}$.
Monotonicity interpolates between successive $(\rho(\mathbf M_{\mathfrak C})\lambda_{\mathfrak C})^k$.
Since $\lambda_{\mathfrak C}^{-k}=((\rho(\mathbf M_{\mathfrak C})\lambda_{\mathfrak C})^k)^{\log\rho(\mathbf M_{\mathfrak C})/\log(\rho(\mathbf M_{\mathfrak C})\lambda_{\mathfrak C})-1}$ and $0<\log\rho(\mathbf M_{\mathfrak C})/\log(\rho(\mathbf M_{\mathfrak C})\lambda_{\mathfrak C})<1$, Lemma~\ref{lem:spectral-resolvent} applies.

The resulting diagonal heat-kernel exponent gives the stated spectral dimension by Definition~\ref{def:spectral-dimension}.

\end{proof}

\begin{example}\label{ex:spectral-threshold}
For the fivefold subdivided diamond, the two terminal paths each have length ten.
Its fully open effective resistance is $5$, so Rayleigh monotonicity at every substitution depth gives $\lambda_{\mathfrak C}\ge5$.
The rule has twenty edges, whence $\rho(\mathbf M_{\mathfrak C})\le20$.
The spectral formula, increasing in $\rho(\mathbf M_{\mathfrak C})$ and decreasing in $\lambda_{\mathfrak C}>1$, yields
\[
 \dimspec^{(\mathrm N)}(\mathfrak C)\le\frac{2\log20}{\log20+\log5}
 =\frac{\log20}{\log10}\approx1.3010<\frac43.
\]
This calculation already gives a strict non-AO example without computing the random electrical eigenlaw.
\end{example}

The time exponent relative to chemical scale is
$\log(\rho(\mathbf M_{\mathfrak C})\lambda_{\mathfrak C})/\log\ell_{\mathfrak C}^{\mathrm{ch}}=\dimminkowski(\mathfrak C,d_{\mathrm{ch}})+\dimresis(\mathfrak C,d_{\mathrm{ch}})$.
Identification with an exponent defined through expected exits from arbitrary chemical balls would require further exit-time estimates and is not used here.
Multiplying the speed measure by a positive environment-dependent constant $c_\omega$ changes the kernel to
$p_t^{(c_\omega)}(x,y)=c_\omega^{-1}p_{t/c_\omega}(x,y)$, so the scalar time changes in Corollary~\ref{cor:deterministic-clock} do not alter the spectral dimension.
Our assertion concerns typical points of the compact limiting diffusion.
It neither gives a value at every finite-born vertex nor identifies a long-time return exponent on an infinite microscopic graph from fixed-time convergence alone.

\section{Alexander--Orbach conjecture on hierarchical lattices}\label{sec:ao}
Can the Alexander--Orbach conjecture hold on hierarchical lattices?
We address this question using the results of \textcolor{black}{Sections~\ref{sec:geometry}--\ref{sec:diffusion}}.

In the plane of chemical Minkowski and resistance dimensions, we prove that the typical spectral dimension differs from $4/3$ outside $(4/3,2)\times(2/3,1)$.
Equality remains possible only on an open segment of the Alexander--Orbach line; we conjecture that it never occurs on any classical EIGS hierarchical lattice.

Throughout this section, $\dimminkowski(\mathfrak C,d_{\mathrm{ch}})$ is the chemical Minkowski dimension of the complete critical cluster, and $\dimresis(\mathfrak C,d_{\mathrm{ch}})$ is the resistance dimension relative to the same chemical scale.
Theorem~\ref{thm:heat-kernels} gives
\begin{equation}\label{eq:ao-equality}
 \dimspec^{(\mathrm N)}(\mathfrak C)=\frac43
 \quad\Longleftrightarrow\quad
 \rho(\mathbf M_{\mathfrak C})=\lambda_{\mathfrak C}^2
 \quad\Longleftrightarrow\quad
 \dimminkowski(\mathfrak C,d_{\mathrm{ch}})=2\dimresis(\mathfrak C,d_{\mathrm{ch}}).
\end{equation}
Here and below $\dimspec^{(\mathrm N)}(\mathfrak C)$ denotes the deterministic value at $\mu$-almost every point, on an event of probability one.

The elementary inequality $\Reff_{\mathfrak C_{n}}(\mathfrak v_+,\mathfrak v_-)\le d_{\mathfrak C_{n}}(\mathfrak v_+,\mathfrak v_-)$ gives only a non-strict comparison of growth factors.
To exclude the boundary of the equality region, we first show why a local electrical improvement persists at the natural exponential scale.
For a nonnegative integrable function $Z$ of the original innovation tree, define
$\mathcal T Z:=F_H(Z_e:e\in E(H))$, where each $Z_e$ is the same function evaluated on the corresponding descendant tree and $F_H$ is the electrical functional.
Thus $\mathcal T X=\lambda_{\mathfrak C}X$ for the natural field of Section~\ref{sec:resistance}.

\begin{lemma}\label{lem:critical-supersolution}
Suppose $Z\ge0$ is measurable with respect to the innovation tree, $0<\E Z<\infty$, and $\mathcal T Z\le aZ$ almost surely for a deterministic $a>0$.
Then $a\ge\lambda_{\mathfrak C}$.
If equality holds, $Z$ is a deterministic positive multiple of $X$.
\end{lemma}
\begin{proof}
Independent thinning gives
$\Pp(Z=0)\le\Pp(\mathcal T Z=0)=g(\Pp(Z=0))$.
Since $\E Z>0$ and $g(u)<u$ on $(0,1)$, we have $\Pp(Z=0)=0$.
The small-resistance argument in Lemma~\ref{lem:electrical-eigenlaw} also gives
\[
 \Pp(Z\le t)\le g\bigl(\Pp(Z\le\max\{2,ma\}t)\bigr).
\]
Iteration, starting at any $t_0$ with $\Pp(Z\le t_0)<1$, yields every inverse moment of $Z$.

Iterate the supersolution inequality and condition on $\mathcal F_n$.
Concavity of effective conductance, followed by Jensen's inequality for reciprocals, gives
\[
 \E[Z\mid\mathcal F_n]
 \ge a^{-n}\E[F_{\mathcal H_n}(Z_u)\mid\mathcal F_n]
 \ge\frac{\Reff_{\mathfrak C_{n}}(\mathfrak v_+,\mathfrak v_-)}{a^n\E Z^{-1}}.
\]
Taking expectations and using Theorem~\ref{thm:resistance-exponent} proves $a\ge\lambda_{\mathfrak C}$.
If $a=\lambda_{\mathfrak C}$, the lower sandwich bound in Lemma~\ref{lem:electrical-sandwich} and martingale convergence imply
$Z\ge X/\E Z^{-1}$ almost surely.

Let $c:=\esssup(X/Z)$, so $0<c<\infty$.
Superadditivity of resistance shows that the nonnegative integrable function $cZ-X$ satisfies
$\mathcal T(cZ-X)\le\lambda_{\mathfrak C}(cZ-X)$.
If it is nonzero, the preceding argument applied to this function gives
\[
 cZ-X\ge\frac{X}{\E(cZ-X)^{-1}}.
\]
Consequently $X/Z\le c/(1+1/\E(cZ-X)^{-1})<c$ almost surely, contradicting the definition of $c$.
Thus $cZ=X$ almost surely.
\end{proof}

\begin{lemma}\label{lem:strict-chemical-resistance}
For every classical rule,
$1<\lambda_{\mathfrak C}<\ell_{\mathfrak C}^{\mathrm{ch}}$.
\end{lemma}
\begin{proof}
The first inequality is Theorem~\ref{thm:resistance-exponent}.
A unit flow along a shortest path with edge costs $X_e^{\mathrm{ch}}$ gives
$\mathcal T X^{\mathrm{ch}}\le\ell_{\mathfrak C}^{\mathrm{ch}}X^{\mathrm{ch}}$.
This inequality is strict on the all-open parent event.
Indeed, the terminal mincut is at least two, so a shortest terminal path has an alternative segment joining two of its vertices through edges outside that segment.
Diverting a sufficiently small positive amount of its unit current along the alternative segment decreases the energy: the saving on the original segment is linear in the diverted current, whereas the additional energy is quadratic.
All edge costs are positive, and the all-open event has conditional probability $p_c^{m-1}>0$.
Lemma~\ref{lem:critical-supersolution} therefore rules out $\ell_{\mathfrak C}^{\mathrm{ch}}=\lambda_{\mathfrak C}$.
\end{proof}

The full cluster and its terminal path core have different mass growth.
Let $H$ be the crossing-conditioned union of simple terminal paths.
This counts the edges supporting terminal routes; the complete mass matrix also retains connected dangling cells and noncrossing live cells.

\begin{lemma}\label{lem:mass-core-clock}
For every classical rule, $\dimminkowski(\mathfrak C,d_{\mathrm{ch}})+\dimresis(\mathfrak C,d_{\mathrm{ch}})>2$.
\end{lemma}

\begin{proof}
Every core edge is a connected child, so $M_{c,c}\ge\E|E(H)|$.
Primitivity and the positive right Perron vector of $\mathbf M_{\mathfrak C}$ give $\rho(\mathbf M_{\mathfrak C})>M_{c,c}$: the $c$ row has a positive off-diagonal entry, corresponding to a live noncrossing child.
Thus $\rho(\mathbf M_{\mathfrak C})>\E|E(H)|$.

For any finite network with unit resistances, orient its electrical unit flow along decreasing potential and decompose it into terminal path flows of total weight one.
The sum of the absolute edge currents is at least the shortest terminal distance.
Cauchy--Schwarz therefore gives, for the recursively retained core,
\[
 d_{\mathfrak C_{n}}(\mathfrak v_+,\mathfrak v_-)^2\le |E(\mathcal H_n)|\Reff_{\mathfrak C_{n}}(\mathfrak v_+,\mathfrak v_-).
\]
Its expected edge count is $(\E|E(H)|)^n$, by independent reproduction at each core address.
    Taking square roots and expectations gives
\[
 (\E d_{\mathfrak C_{n}}(\mathfrak v_+,\mathfrak v_-))^2\le(\E|E(H)|)^n\E\Reff_{\mathfrak C_{n}}(\mathfrak v_+,\mathfrak v_-).
\]
The growth limits in Theorems~\ref{thm:chemical-scalar} and~\ref{thm:resistance-exponent} give $(\ell_{\mathfrak C}^{\mathrm{ch}})^2\le(\E|E(H)|)\lambda_{\mathfrak C}$.
Combining the two inequalities gives $\rho(\mathbf M_{\mathfrak C})\lambda_{\mathfrak C}>(\ell_{\mathfrak C}^{\mathrm{ch}})^2$, which is precisely the asserted strict dimension bound.
\end{proof}

These two comparisons leave only an open segment of the AO line as a possible equality set.

\begin{aoresult}\label{thm:ao-region}
For every classical EIGS rule, if
\[
 (\dimminkowski(\mathfrak C,d_{\mathrm{ch}}),\dimresis(\mathfrak C,d_{\mathrm{ch}}))
 \notin (\frac{4}{3},2)\times(\frac{2}{3},1),
\]
then the typical spectral dimension of the natural limiting diffusion on its critical cluster satisfies
\[
 \dimspec^{(\mathrm N)}(\mathfrak C)\neq\frac{4}{3}.
\]
\end{aoresult}
\begin{proof}
Lemmas~\ref{lem:strict-chemical-resistance} and~\ref{lem:mass-core-clock} give
$0<\dimresis(\mathfrak C,d_{\mathrm{ch}})<1$ and $\dimminkowski(\mathfrak C,d_{\mathrm{ch}})+\dimresis(\mathfrak C,d_{\mathrm{ch}})>2$.
If $\dimspec^{(\mathrm N)}(\mathfrak C)=4/3$, equation~\eqref{eq:ao-equality} gives
$3\dimresis(\mathfrak C,d_{\mathrm{ch}})>2$ and $\dimminkowski(\mathfrak C,d_{\mathrm{ch}})=2\dimresis(\mathfrak C,d_{\mathrm{ch}})<2$.
Both inequalities are strict, proving the statement including its boundary cases.
\end{proof}

The proof also fixes the sign in two useful regions.
If $\dimminkowski(\mathfrak C,d_{\mathrm{ch}})\ge2$, then $\dimminkowski(\mathfrak C,d_{\mathrm{ch}})>2\dimresis(\mathfrak C,d_{\mathrm{ch}})$ and $\dimspec^{(\mathrm N)}(\mathfrak C)>4/3$.
If $\dimminkowski(\mathfrak C,d_{\mathrm{ch}})\le4/3$, the strict clock bound gives
$\dimresis(\mathfrak C,d_{\mathrm{ch}})>2-\dimminkowski(\mathfrak C,d_{\mathrm{ch}})\ge\dimminkowski(\mathfrak C,d_{\mathrm{ch}})/2$, and $\dimspec^{(\mathrm N)}(\mathfrak C)<4/3$.

\begin{example}\label{ex:long-diamond-region}
Replace each edge of the standard diamond by a path of five unit edges.
The resulting rule has twenty edges and two terminal paths of length ten, so $\ell_{\mathfrak C}^{\mathrm{ch}}=10$.
Since $\rho(\mathbf M_{\mathfrak C})<20$,
\[
 \dimminkowski(\mathfrak C,d_{\mathrm{ch}})<\frac{\log20}{\log10}\approx1.3010<\frac43.
\]
Its typical spectral dimension is therefore strictly smaller than $4/3$ by the preceding sign calculation.
The critical point throughout is that of this subdivided rule.
\end{example}

The preceding theorem excludes the Alexander--Orbach equality outside a narrow region, but leaves open the possibility of equality within it.
To seek an example attaining $4/3$, we ask whether the typical spectral dimensions of classical EIGS critical clusters are dense in $(1,2)$, so that a sequence of finite rules can approximate the desired value.
This leads to the following result.

\begin{theorem}\label{thm:spectral-density}
The typical spectral dimensions $\dimspec^{(\mathrm N)}(\mathfrak C)$ of the natural limiting diffusions on critical percolation clusters of classical EIGS hierarchical lattices are dense in $(1,2)$.
\end{theorem}
\begin{proof}
    See Appendix~\ref{app:spectral-density}.
\end{proof}

The density theorem allows approximation of $4/3$, but does not produce a finite rule whose critical cluster has exactly this spectral dimension.
In the construction of Appendix~\ref{app:spectral-density}, the rule size diverges along the approximating sequence.
We do not claim that this approach is impossible, but none of our attempts has yielded such a finite rule; this leads us to the following conjecture.

\begin{aoconjecture}\label{conj:ao-global}
For every classical EIGS hierarchical lattice at its own critical point, the typical spectral dimension of the natural limiting diffusion on the full terminal-connected cluster satisfies
\[
\dimspec^{(\mathrm N)}(\mathfrak C) \neq \frac{4}{3}.
\]
That is, the Alexander-Orbach conjecture does not hold on any classical EIGS hierarchical lattice.
\end{aoconjecture}

\section*{Acknowledgement}
This work was supported by the Additional Funding Programme for Mathematical Sciences, delivered by EPSRC (EP/V521917/1) and the Heilbronn Institute for Mathematical Research,
and also by the EPSRC Centre for Doctoral Training in Mathematics of Random Systems: Analysis, Modelling and Simulation (EP/S023925/1).

\bibliographystyle{amsplain}
\begingroup
\small
\bibliography{references}
\endgroup

\appendix
\newcommand{\appendixsection}[1]{%
  \refstepcounter{section}%
  \section*{Appendix \thesection. #1}%
}
\appendixsection{Chemical scaling estimates}\label{app:chemical}
\addcontentsline{toc}{section}{Appendix}

This appendix proves the scalar estimates used in Theorem~\ref{thm:chemical-scalar} and the component mesh used in the full-space construction.
All scalar functionals in this appendix are shortest-path functionals $F_H^{\mathrm{ch}}$.
For its stationary law, write $X^{\mathrm{ch}}$ for the root variable.

\begin{lemma}\label{lem:chemical-moments}
The laws of $d_{\mathfrak C_{n}}(\mathfrak v_+,\mathfrak v_-)/(\E d_{\mathfrak C_{n}}(\mathfrak v_+,\mathfrak v_-))$ have uniformly bounded moments of every fixed positive and negative order.
There is a mean-one law on $(0,\infty)$ and $\ell_{\mathfrak C}^{\mathrm{ch}}>1$ such that $F_H^{\mathrm{ch}}(X_e^{\mathrm{ch}})$ has the law of $\ell_{\mathfrak C}^{\mathrm{ch}} X^{\mathrm{ch}}$ for independent child copies.
This law has every positive and negative moment.
For each finite $t>0$, $\E e^{tX^{\mathrm{ch}}}$ and $\sup_n\E e^{td_{\mathfrak C_n}(\mathfrak v_+,\mathfrak v_-)/(\E d_{\mathfrak C_{n}}(\mathfrak v_+,\mathfrak v_-))}$ are finite.
\end{lemma}

\begin{proof}
Under $\mathcal Q_c$, the pivotal bonds counted by $N_{\mathrm{piv}}(R)$ in IGS~II are precisely the open edges present in every terminal path.
The conditional pivotal-count identity in \textcolor{black}{\cite[proof of Lemma~2.16]{NeroliIGSII2026}} gives $\E N_{\mathrm{piv}}(R)=\varphi_R'(p_c)>1$.
For independent mean-one nonnegative inputs, their sum is a lower bound for $F_H^{\mathrm{ch}}$.
Choose a deterministic simple path in each $H$, independently of the inputs, and let its length be $J(H)$.
Then $\E J(H)\le L$.
For any mean-one input law $\nu$ we obtain $\varphi_R'(p_c)\le \E F_H^{\mathrm{ch}}(X_e^{\mathrm{ch}})\le\E J(H)$.
This lower bound is $\varphi_R'(p_c)$, rather than a pathwise assertion that arbitrary weighted paths cost at least $\ell$.

These first-moment bounds allow us to construct an invariant class of laws.
Let $\widehat T\nu$ be the law of $F_H^{\mathrm{ch}}(X_e^{\mathrm{ch}})/(\E F_H^{\mathrm{ch}}(X_e^{\mathrm{ch}}))$.
For integer $p\ge2$, interpolation gives $\E (X^{\mathrm{ch}})^r\le \bigl(\E (X^{\mathrm{ch}})^p\bigr)^{(r-1)/(p-1)}$ for $1\le r\le p$.
The pure terms in a selected-path multinomial expansion contribute $(\E J(H))\bigl(\E (X^{\mathrm{ch}})^p\bigr)$.
A mixed monomial with $k\ge2$ distinct variables is bounded by $\bigl(\E (X^{\mathrm{ch}})^p\bigr)^{(p-k)/(p-1)}$.
Consequently
\[
 \int_0^\infty x^p\,(\widehat T\nu)(dx)\le \frac{\E J(H)}{\varphi_R'(p_c)^p}\left(\int_0^\infty x^p\,\nu(dx)\right)
       +A_p\left(\int_0^\infty x^p\,\nu(dx)\right)^{(p-2)/(p-1)}
\]
for a finite $A_p$.
For all sufficiently large integers $p$, choose $K_p\ge1$ so that this bound preserves $\bigl(\E (X^{\mathrm{ch}})^p\bigr)\le K_p$.
The mean-one laws satisfying these constraints form a nonempty convex weakly compact set.
Higher moments ensure tightness and preservation of the mean; lower semicontinuity preserves the constraints.
The Lipschitz bound for $F_H^{\mathrm{ch}}$ and uniform integrability make $\widehat T$ continuous there.
Schauder's theorem gives a fixed law with eigenvalue $\ell_{\mathfrak C}^{\mathrm{ch}}=\E F_H^{\mathrm{ch}}(X_e^{\mathrm{ch}})\in[\varphi_R'(p_c),\E J(H)]$.
It has every positive moment.
Iteration from $\delta_1$ gives the corresponding uniform moments of $d_{\mathfrak C_{n}}(\mathfrak v_+,\mathfrak v_-)/(\E d_{\mathfrak C_{n}}(\mathfrak v_+,\mathfrak v_-))$.

To control negative moments, we estimate the probability of small values.
If $z=\Pp(X^{\mathrm{ch}}=0)$, then $z=g(z)$, since zero terminal cost is a crossing through zero-cost child edges.
This is independent thinning of a crossing Bernoulli-$p_c$ configuration.
The mean excludes $z=1$, and $g(u)<u$ for $0<u<1$ excludes $0<z<1$.
Thus $X^{\mathrm{ch}}>0$.
For $F(t):=\Pp(X^{\mathrm{ch}}\le t)$, a path with cost at most $\ell_{\mathfrak C}^{\mathrm{ch}} t$ uses only edges of cost at most $\ell_{\mathfrak C}^{\mathrm{ch}} t$.
Hence $F(t)\le g(F(\ell_{\mathfrak C}^{\mathrm{ch}} t))$.
Choosing $F(t_0)<1$ and iterating gives $F(t_0(\ell_{\mathfrak C}^{\mathrm{ch}})^{-k})\le\exp(-c\ell^k)$ after finitely many initial terms.
This proves every inverse moment.

For
\[
 F_n(t):=\Pp(d_{\mathfrak C_{n}}(\mathfrak v_+,\mathfrak v_-)/(\E d_{\mathfrak C_{n}}(\mathfrak v_+,\mathfrak v_-))\le t)
\]
, the bound
\[
 (\E d_{\mathfrak C_{n+1}}(\mathfrak v_+,\mathfrak v_-))/(\E d_{\mathfrak C_{n}}(\mathfrak v_+,\mathfrak v_-))\le L
\]
 gives $F_{n+1}(t)\le g(F_n(Lt))$.
Uniform second moments and Paley--Zygmund give $F_n(1/2)\le r<1$.
For $n\ge k$, $F_n((1/2)L^{-k})\le g^{\circ k}(r)$.
For $n<k$, the event is empty because
\[
 d_{\mathfrak C_{n}}(\mathfrak v_+,\mathfrak v_-)/(\E d_{\mathfrak C_{n}}(\mathfrak v_+,\mathfrak v_-))\ge(\ell/L)^n>(1/2)L^{-k}\textcolor{black}{.}
\]
The same small-ball estimate is therefore uniform over $n$.

We next strengthen the positive-moment estimates to factorial bounds.
Choose $p_0$ so that $\binom{p+L-1}{L-1}<\varphi_R'(p_c)^p$ for $p\ge p_0$.
Choose $A\ge1$ to cover the finitely many lower-order moment bounds by $A^r r!$.
If the bound holds below order $p$, each mixed multinomial term is at most $A^pp!$.
A path of length $J$ has at most $\binom{p+L-1}{L-1}-J$ mixed terms.
Thus
\[\begin{aligned}
 &\E\left(\frac{d_{\mathfrak C_{n+1}}(\mathfrak v_+,\mathfrak v_-)}{\E d_{\mathfrak C_{n+1}}(\mathfrak v_+,\mathfrak v_-)}\right)^p\\
 &\quad\le \varphi_R'(p_c)^{-p}(\E J(H))\E\left(\frac{d_{\mathfrak C_{n}}(\mathfrak v_+,\mathfrak v_-)}{\E d_{\mathfrak C_{n}}(\mathfrak v_+,\mathfrak v_-)}\right)^p\\
 &\qquad{}+\varphi_R'(p_c)^{-p}A^pp!\left(\binom{p+L-1}{L-1}-\E J(H)\right).
\end{aligned}\]
Induction on $p$ and then $n$ yields $\E\left(\frac{d_{\mathfrak C_{n}}(\mathfrak v_+,\mathfrak v_-)}{\E d_{\mathfrak C_{n}}(\mathfrak v_+,\mathfrak v_-)}\right)^p\le A^pp!$.
For the eigenlaw, $((\ell_{\mathfrak C}^{\mathrm{ch}})^p-\E J(H))\bigl(\E (X^{\mathrm{ch}})^p\bigr)\le A^pp!(\binom{p+L-1}{L-1}-\E J(H))$ gives the same bound because $(\ell_{\mathfrak C}^{\mathrm{ch}})^p\ge \varphi_R'(p_c)^p>\binom{p+L-1}{L-1}$.

The factorial bounds first give exponential moments on a right neighbourhood of zero.
If $\psi_n(t):=\E e^{td_{\mathfrak C_n}(\mathfrak v_+,\mathfrak v_-)/(\E d_{\mathfrak C_{n}}(\mathfrak v_+,\mathfrak v_-))}$, the path bound and $(\E d_{\mathfrak C_{n+1}}(\mathfrak v_+,\mathfrak v_-))/(\E d_{\mathfrak C_{n}}(\mathfrak v_+,\mathfrak v_-))\ge \varphi_R'(p_c)$ give $\psi_{n+1}(\varphi_R'(p_c)t)\le\psi_n(t)^L$.
Together with $\psi_0(t)=e^t$, this extends uniform finiteness from $t$ to $\varphi_R'(p_c)t$.
Iteration reaches every finite positive $t$.
The eigenlaw obeys the analogous bound $\psi(\varphi_R'(p_c)t)\le\psi(t)^L$.
\end{proof}

\begin{lemma}\label{lem:chemical-paths}
There is $C_R<\infty$ such that a realised depth-$n$ path-core graph has at most $C_R^k$ simple terminal paths of $k$ microscopic edges.
If $\theta>0$ and independent leaf costs $U\ge0$ satisfy $C_Re^{\theta v}\E e^{-\theta U}<1$, then
\[
 \Pp\{F_{\mathcal H_n}^{\mathrm{ch}}(U)<vd_{\mathfrak C_n}(\mathfrak v_+,\mathfrak v_-)\mid\mathcal H_n\}\le\frac{(C_Re^{\theta v}\E e^{-\theta U})^{\ell^n}}{1-C_Re^{\theta v}\E e^{-\theta U}}.
\]
For $\theta>0$ and pairs $(U,V)$ independent across leaves, $C_R\E e^{\theta(U-cV)}<1$ gives the same bound with $C_R\E e^{\theta(U-cV)}$ in place of $C_Re^{\theta v}\E e^{-\theta U}$ for the event that some terminal path has $\sum U>c\sum V$.
\end{lemma}

\begin{proof}
A simple microscopic terminal path cannot revisit a planted vertex or re-enter a two-terminal child after leaving it.
It therefore projects recursively to simple coarse terminal paths.
Its ordered encoding tree has $k$ leaves and between $\ell$ and $L$ children at each internal vertex.
There are at most $k-1$ internal vertices.
The number of such plane trees is exponential in $k$, and each internal label has only finitely many possible oriented rule paths.
This gives $C_R^k$ after increasing the constant.

For a fixed path of $k$ leaves, exponential Markov bounds $\Pp(\sum U<vk)$ by $(e^{\theta v}\E e^{-\theta U})^k$.
Sum over all paths and $k\ge\ell^n$.
If every path costs at least $v$ times its length, the minimum is at least $vd_{\mathfrak C_n}(\mathfrak v_+,\mathfrak v_-)$.
Applying the same calculation to $U-cV$ gives the paired assertion.
The coordinates of each pair can be dependent.
\end{proof}

\begin{lemma}\label{lem:chemical-compactness}
The eigenlaw has a stationary realisation measurable with respect to the original innovation tree.
There are deterministic finite $C,C'$ such that
\[
 \begin{gathered}
 1\le\E \bigl((\ell_{\mathfrak C}^{\mathrm{ch}})^{-n}d_{\mathfrak C_{n}}(\mathfrak v_+,\mathfrak v_-)\bigr)\le C,\qquad
 \sup_n\E e^{t\bigl((\ell_{\mathfrak C}^{\mathrm{ch}})^{-n}d_{\mathfrak C_{n}}(\mathfrak v_+,\mathfrak v_-)\bigr)}<\infty\ (t>0),\qquad
 \\
 X^{\mathrm{ch}}\le\liminf_n\bigl((\ell_{\mathfrak C}^{\mathrm{ch}})^{-n}d_{\mathfrak C_{n}}(\mathfrak v_+,\mathfrak v_-)\bigr)\le\limsup_n\bigl((\ell_{\mathfrak C}^{\mathrm{ch}})^{-n}d_{\mathfrak C_{n}}(\mathfrak v_+,\mathfrak v_-)\bigr)\le C'X^{\mathrm{ch}}
 \quad\text{almost surely}.
 \end{gathered}
\]
The sequence is relatively compact in this same $L^2$ probability space.
\end{lemma}

\begin{proof}
Sample independent eigenlaw variables at a finite boundary and recurse upwards.
The eigenlaw identity makes the finite-dimensional distributions consistent.
Kolmogorov extension gives a stationary recursive field satisfying $F_{H_u}^{\mathrm{ch}}(X_{ui}^{\mathrm{ch}})=\ell_{\mathfrak C}^{\mathrm{ch}} X_u^{\mathrm{ch}}$, with independent descendant subtrees.
We now remove its possible extra boundary randomness.

Let $\mathcal F_k$ be the first $k$ innovation generations.
For a fixed outer graph and independent positive costs $Y_i$, put $Z=F_{\mathcal H_k}^{\mathrm{ch}}(Y)$ and choose a geodesic $P(Y)$ by deterministic tie-breaking.
Replacing $Y_i$ by an independent $Y_i'$ increases $Z$ by at most $\mathbf1_{\{i\in P(Y)\}}Y_i'$.
The one-sided Efron--Stein inequality gives
\[
 \E\Var(Z\mid\mathcal F_k)\le \E Y^2\,\E|P(Y)|.
\]
Since $|P(Y)|\le Z\max_iY_i^{-1}$ and there are at most $m^k$ leaves, Cauchy--Schwarz and a maximum bound give
\[
 \E|P(Y)|\le\|Z\|_2m^{k/p}\|Y^{-1}\|_p\qquad(p\ge2).
\]
With $Y=X^{\mathrm{ch}}$ the output is $(\ell_{\mathfrak C}^{\mathrm{ch}})^kX^{\mathrm{ch}}$.
Choose $p$ with $m^{1/p}<\ell_{\mathfrak C}^{\mathrm{ch}}$.
Then
\begin{equation}\label{eq:chemical-endogeny}
 \E\Var(X^{\mathrm{ch}}\mid\mathcal F_k)\le C_0(m^{1/p}/\ell_{\mathfrak C}^{\mathrm{ch}})^k.
\end{equation}
Thus $X^{\mathrm{ch}}$ is measurable with respect to all innovations.
The same holds at each address; this property is called endogeny \cite{AldousBandyopadhyay2005}.

Conditional concavity gives $\E[X^{\mathrm{ch}}\mid\mathcal F_n]\le \bigl((\ell_{\mathfrak C}^{\mathrm{ch}})^{-n}d_{\mathfrak C_{n}}(\mathfrak v_+,\mathfrak v_-)\bigr)$.
Martingale convergence implies $X^{\mathrm{ch}}\le\displaystyle \liminf \bigl((\ell_{\mathfrak C}^{\mathrm{ch}})^{-n}d_{\mathfrak C_{n}}(\mathfrak v_+,\mathfrak v_-)\bigr)$, and $\E \bigl((\ell_{\mathfrak C}^{\mathrm{ch}})^{-n}d_{\mathfrak C_{n}}(\mathfrak v_+,\mathfrak v_-)\bigr)\ge1$.
Choose $\theta,v>0$ with $r=C_Re^{\theta v}\E e^{-\theta X^{\mathrm{ch}}}<1$, using $X^{\mathrm{ch}}>0$.
The stationary identity and Lemma~\ref{lem:chemical-paths} give
\begin{equation}\label{eq:chemical-exception}
 \Pp(\bigl((\ell_{\mathfrak C}^{\mathrm{ch}})^{-n}d_{\mathfrak C_{n}}(\mathfrak v_+,\mathfrak v_-)\bigr)>X^{\mathrm{ch}}/v)\le C_1r^{\ell^n}.
\end{equation}
On this event use $\bigl((\ell_{\mathfrak C}^{\mathrm{ch}})^{-n}d_{\mathfrak C_{n}}(\mathfrak v_+,\mathfrak v_-)\bigr)\le(L/\ell_{\mathfrak C}^{\mathrm{ch}})^n$.
It follows that $\E \bigl((\ell_{\mathfrak C}^{\mathrm{ch}})^{-n}d_{\mathfrak C_{n}}(\mathfrak v_+,\mathfrak v_-)\bigr)\le v^{-1}+C_1(L/\ell_{\mathfrak C}^{\mathrm{ch}})^nr^{\ell^n}$, uniformly bounded.
Borel--Cantelli gives $\displaystyle \limsup \bigl((\ell_{\mathfrak C}^{\mathrm{ch}})^{-n}d_{\mathfrak C_{n}}(\mathfrak v_+,\mathfrak v_-)\bigr)\le X^{\mathrm{ch}}/v$.
Since $(\E d_{\mathfrak C_{n}}(\mathfrak v_+,\mathfrak v_-))\ge\ell^n$, bounded means also force $\ell_{\mathfrak C}^{\mathrm{ch}}\ge\ell$.
The uniform bound on $(\ell_{\mathfrak C}^{\mathrm{ch}})^{-n}\E d_{\mathfrak C_n}(\mathfrak v_+,\mathfrak v_-)$ therefore transfers the exponential bounds of Lemma~\ref{lem:chemical-moments} to the chemically rescaled terminal distances.

For $n\ge k$, condition on $\mathcal F_k$ and use independent inputs $Y=\bigl((\ell_{\mathfrak C}^{\mathrm{ch}})^{-{n-k}}d_{\mathfrak C_{n-k}}(\mathfrak v_+,\mathfrak v_-)\bigr)$.
Their positive and inverse moments are uniform, by the preceding mean bounds.
The output is $(\ell_{\mathfrak C}^{\mathrm{ch}})^k\bigl((\ell_{\mathfrak C}^{\mathrm{ch}})^{-n}d_{\mathfrak C_{n}}(\mathfrak v_+,\mathfrak v_-)\bigr)$, and the same variance estimate gives
\begin{equation}\label{eq:chemical-compactness}
 \sup_{n\ge k}\E\Var(\bigl((\ell_{\mathfrak C}^{\mathrm{ch}})^{-n}d_{\mathfrak C_{n}}(\mathfrak v_+,\mathfrak v_-)\bigr)\mid\mathcal F_k)\le C_2(m^{1/p}/\ell_{\mathfrak C}^{\mathrm{ch}})^k.
\end{equation}
For $n<k$ the conditional variance is zero.
Each $\mathcal F_k$ has finitely many atoms, so its bounded conditional expectations lie in a relatively compact finite-dimensional set.
Equation~\eqref{eq:chemical-compactness} approximates the whole sequence uniformly by these sets.
\end{proof}

For a nonnegative square-integrable function of the innovation tree, define $Tf:=(\ell_{\mathfrak C}^{\mathrm{ch}})^{-1}F_H^{\mathrm{ch}}(f_1,\ldots,f_m)$, where $f_i$ is the same function on the $i$th descendant tree.
Then $TX^{\mathrm{ch}}=X^{\mathrm{ch}}$ and $T^n1=\bigl((\ell_{\mathfrak C}^{\mathrm{ch}})^{-n}d_{\mathfrak C_{n}}(\mathfrak v_+,\mathfrak v_-)\bigr)$.

\begin{lemma}\label{lem:chemical-endpoint}
The $L^2$ omega-limit set $\mathcal O$ of $(\bigl((\ell_{\mathfrak C}^{\mathrm{ch}})^{-n}d_{\mathfrak C_{n}}(\mathfrak v_+,\mathfrak v_-)\bigr))$ contains $cX^{\mathrm{ch}}$ for a deterministic $c\ge1$.
\end{lemma}

\begin{proof}
The Lipschitz bound makes $T$ continuous in $L^2$.
Compactness gives a nonempty compact $\mathcal O$ with $T\mathcal O=\mathcal O$; the reverse inclusion follows by taking convergent predecessor subsequences.
Each $h\in\mathcal O$ satisfies $X^{\mathrm{ch}}\le h\le C'X^{\mathrm{ch}}$ almost surely, by a further almost-surely convergent subsequence.
Let $c:=\sup_{h\in\mathcal O}\esssup(h/X^{\mathrm{ch}})$.

Suppose $cX^{\mathrm{ch}}\notin\mathcal O$ and consider the nonnegative differences $cX^{\mathrm{ch}}-h$.
Compactness gives $\delta>0$ such that $\inf_h\E (cX^{\mathrm{ch}}-h)\ge\delta$, since otherwise a convergent subsequence would have limit $cX^{\mathrm{ch}}$.
The second moments have a common finite bound $M$.
Paley--Zygmund gives $\sup_h\Pp((cX^{\mathrm{ch}}-h)\le\delta/2)\le1-\delta^2/(4M)<1$.
Superadditivity yields $(cX^{\mathrm{ch}}-Th)\ge T(cX^{\mathrm{ch}}-h)$.
With $Q(t):=\sup_h\Pp((cX^{\mathrm{ch}}-h)\le t)$, predecessor invariance and thinning give $Q(t)\le g(Q(\ell_{\mathfrak C}^{\mathrm{ch}} t))$.
Thus $\displaystyle\lim_{t\downarrow0}Q(t)=0$.
The inequality $\sup_h\E e^{-\theta (cX^{\mathrm{ch}}-h)}\le Q(t)+e^{-\theta t}$ proves uniform decay of these Laplace transforms as $\theta\to\infty$.
Choose $\theta,v>0$ with $C_Re^{\theta v}\sup_h\E e^{-\theta (cX^{\mathrm{ch}}-h)}<1$.

For fixed $h$, choose deterministic predecessor functions $h_{-n}\in\mathcal O$ with $T^nh_{-n}=h$.
Then $(cX^{\mathrm{ch}}-h)\ge T^n(cX^{\mathrm{ch}}-h_{-n})$.
Their leaf copies are independent and independent of the exposed outer graph.
The uniform path bound and Borel--Cantelli give $(cX^{\mathrm{ch}}-h)\ge v\bigl((\ell_{\mathfrak C}^{\mathrm{ch}})^{-n}d_{\mathfrak C_{n}}(\mathfrak v_+,\mathfrak v_-)\bigr)$ eventually, hence $(cX^{\mathrm{ch}}-h)\ge vX^{\mathrm{ch}}$.
The constant $v$ is common to all $h$.
Apply the conclusion on a countable dense subset and use closedness of the order inequality.
Every $h\in\mathcal O$ then obeys $h\le(c-v)X^{\mathrm{ch}}$, contradicting the definition of $c$.
\end{proof}

\begin{lemma}\label{lem:full-mesh}
For every live type, the supremum over all remaining depths of each $\ell_{\mathfrak C}^{\mathrm{ch}}$-normalised retained-component diameter has every finite positive moment.
Almost surely there are exponentially decreasing bounds $\Gamma_k$ with $\displaystyle\lim_{k\to\infty}\Gamma_k=0$ for all generation-$k$ intrinsic component diameters in $(\ell_{\mathfrak C}^{\mathrm{ch}})^{-n}d_{\mathfrak C_n}$, simultaneously for all $n\ge k$.
\end{lemma}

\begin{proof}
Equation~\eqref{eq:chemical-exception} and $\bigl((\ell_{\mathfrak C}^{\mathrm{ch}})^{-n}d_{\mathfrak C_{n}}(\mathfrak v_+,\mathfrak v_-)\bigr)\le(L/\ell_{\mathfrak C}^{\mathrm{ch}})^n$ give
\[
 V:=\sup_n\bigl((\ell_{\mathfrak C}^{\mathrm{ch}})^{-n}d_{\mathfrak C_{n}}(\mathfrak v_+,\mathfrak v_-)\bigr)\le X^{\mathrm{ch}}/v+\sum_{n\ge0}(L/\ell_{\mathfrak C}^{\mathrm{ch}})^n\mathbf1_{\{\bigl((\ell_{\mathfrak C}^{\mathrm{ch}})^{-n}d_{\mathfrak C_{n}}(\mathfrak v_+,\mathfrak v_-)\bigr)>X^{\mathrm{ch}}/v\}}.
\]
Minkowski and the bound $C_1r^{\ell^n}$ show that $V$ has all positive moments.

Let $H_{u,r}$ be the maximum distance to an available terminal in a vertex's own retained component.
For $c$, minimise over the two terminals; for $b$, take the larger of the separate one-root radii.
The seed radii are zero.
Reach an available child terminal and follow a simple active coarse route to an available parent terminal.
Only connected children provide through-routes.
This yields
\[
 H_{u,r+1}\le\max_iH_{ui,r}+\sum_{i:T_{ui}=c}d_{\mathfrak C_{ui,r}}(\mathfrak v_+,\mathfrak v_-).
\]
A possible crossing of the initial connected child is included in the sum.
Iteration gives
\[
 \sup_r(\ell_{\mathfrak C}^{\mathrm{ch}})^{-r}H_{u,r}
 \le A_u:=m\sum_{j\ge1}(\ell_{\mathfrak C}^{\mathrm{ch}})^{-j}
          \max_{\substack{|w|=j\text{ below }u\\T_w=c}}V_w.
\]
An empty maximum is zero.
Conditional marginal moment bounds and Minkowski give
$\|A_u\|_p\le m\|V\|_p\sum_{j\ge1}(m^{1/p}/\ell_{\mathfrak C}^{\mathrm{ch}})^j<\infty$ for arbitrarily large $p$.
Thus all positive orders are finite.
The variable $G_u=2A_u+\mathbf1_{\{T_u=c\}}V_u$ bounds the diameter supremum for each component.
No independence of sibling types is used.

Put $\Gamma_k:=(\ell_{\mathfrak C}^{\mathrm{ch}})^{-k}\max_{|u|=k,\;u\text{ live}}G_u$.
For $m^{1/p}/\ell_{\mathfrak C}^{\mathrm{ch}}<\eta<1$, Markov gives
$\Pp(\Gamma_k>\eta^k)\le C[m/(\ell_{\mathfrak C}^{\mathrm{ch}}\eta)^p]^k$.
Choose $p$ large and apply Borel--Cantelli.
The supremum over remaining depths was already taken in $G_u$, proving the required uniformity.
\end{proof}

\appendixsection{Heat-kernel estimates}\label{app:spectral}

The component geometry and natural resistance form are those of Section~\ref{sec:resistance}.
The following estimates establish the mass, localisation and resolvent inputs to Theorem~\ref{thm:heat-kernels}.

A generation-$k$ socket has unnormalised mass $\rho(\mathbf M_{\mathfrak C})^{-k}W_{T_u}^{(u)}$.
For its retained component $A$, write $W_c\mu(A)=\rho(\mathbf M_{\mathfrak C})^{-k}Y_A$.
The two masked masses of $b$ sum to its socket mass and retain their original dependence.
The intrinsic resistance diameter of $A$ is at most $\lambda_{\mathfrak C}^{-k}H_A$, where the component laws of $H_A$ have all positive moments by Lemma~\ref{lem:full-resistance-mesh}.
Let $\mathcal E_A^{\mathrm g}$ denote its local energy in global normalisation.

By \cite[Lemma~3.2]{NeroliIGSII2026}, $\rho(\mathbf M_{\mathfrak C})>d_R$ and the normalised retained vertex counts have uniformly bounded moments of every positive order.
The $L^2$ mass limits in \cite[proof of Lemma~3.3]{NeroliIGSII2026} and Fatou's lemma therefore give every positive moment of $W_\sigma$.
Masked masses are bounded by socket masses.

\begin{lemma}\label{lem:spectral-moments}
For each $\varepsilon>0$, almost surely,
\begin{equation}\label{eq:spectral-maxima}
 \max_{A\in\mathcal A_k}(Y_A+H_A)
 +\max_{\substack{|u|=k\\T_u=c}}X_u^{-1}
 \le C_{\omega,\varepsilon}e^{\varepsilon k}\qquad(k\ge0).
\end{equation}
An empty maximum is zero.
\end{lemma}

\begin{proof}
There are at most $m^k$ sockets and two components per socket.
For each variable in \eqref{eq:spectral-maxima}, a moment bound gives frontier exceedance probability at most $C_pm^ke^{-p\varepsilon k}$.
Choose $p\varepsilon>\log m$ and apply Borel--Cantelli.
An intersection over positive rational $\varepsilon$ proves the assertion.
\end{proof}

For $A\in\mathcal A_k$, let $C_{k,n}(A)$ be the union of its generation-$(k+n)$ descendants meeting its planted boundary.
Set $C_{k,n}:=\bigcup_{A\in\mathcal A_k}C_{k,n}(A)$, using inherited masks throughout.

\begin{lemma}\label{lem:collar-expected-mass}
There is $C<\infty$ with $\E[W_c\mu(C_{k,n})]\le C(d_R/\rho(\mathbf M_{\mathfrak C}))^n$ for all $k,n$.
\end{lemma}

\begin{proof}
Each substitution multiplies the geometric socket degree at a planted vertex by $d_R$.
At most $2d_R^n$ descendant sockets contribute to the collar of one component.
For a side of $b$ its mass is bounded by the whole socket mass.
Conditional expected continuation masses are at most $\max_\sigma h_\sigma$, while $\E|\mathbf Z_k|\le C_0\rho(\mathbf M_{\mathfrak C})^k$.
Subadditivity gives
\begin{equation}\label{eq:collar-mass}
 \E[W_c\mu(C_{k,n})]\le C_1\rho(\mathbf M_{\mathfrak C})^kd_R^n\rho(\mathbf M_{\mathfrak C})^{-k-n}
 =C_1(d_R/\rho(\mathbf M_{\mathfrak C}))^n.
\end{equation}
This does not require independent $b$ sides or prior mass additivity at boundaries.
\end{proof}

\begin{lemma}\label{lem:skeleton-zero-mass}
Almost surely, $\mu(V_*)=0$.
\end{lemma}

\begin{proof}
Every retained vertex born by generation $k$ lies in $C_{k,n}$ for every $n$.
Since $\rho(\mathbf M_{\mathfrak C})>d_R$, \eqref{eq:collar-mass} gives $\E[W_c\mu(V_k)]=0$.
Countably many generations and $W_c>0$ give $\mu(V_*)=0$.
Components meet only at planted vertices, so their masses are additive and the containing component is unique outside $V_*$.
\end{proof}

\begin{lemma}\label{lem:spectral-collars}
On one event of probability one, for $\mu$-almost every $x$, its unique generation-$k$ component $A_k(x)$ has $\mu(A_k(x))=\rho(\mathbf M_{\mathfrak C})^{-k}e^{o(k)}$.
For each positive rational $\eta$, such points lie outside $C_{k,\lceil\eta k\rceil}$ eventually.
\end{lemma}

\begin{proof}
The union $B_k(\varepsilon)$ of components with $Y_A<e^{-\varepsilon k}$ has
\[
 \E[W_c\mu(B_k(\varepsilon))]
 \le\rho(\mathbf M_{\mathfrak C})^{-k}e^{-\varepsilon k}\E[2|\mathbf Z_k|]
 \le C_2e^{-\varepsilon k}.
\]
Tonelli and division by $W_c>0$ give summability of $\mu(B_k(\varepsilon))$ almost surely.
Measure Borel--Cantelli implies eventual avoidance at almost every point.
Together with \eqref{eq:spectral-maxima} and countably many rational $\varepsilon$, this proves the mass asymptotic.
No negative moment of the branching mass is used.

For fixed rational $\eta>0$, \eqref{eq:collar-mass} with $n=\lceil\eta k\rceil$ is summable over $k$.
The same argument proves eventual collar avoidance.
\end{proof}

\begin{lemma}\label{lem:heat-kernel-continuity}
The heat kernel is jointly continuous.
\end{lemma}

\begin{proof}
Averaging the resistance inequality gives
$|f(x)|^2\le2\|f\|_2^2+2\diam_{R_*}(\mathfrak C)\mathcal E(f)$.
A bounded set for the form norm is uniformly bounded and equicontinuous by
$|f(x)-f(y)|^2\le R_*(x,y)\mathcal E(f)$.
Arzelà--Ascoli gives compact form embedding into $L^2(\mu)$.
Let $(\zeta_j,\phi_j)$ be its spectral decomposition.
The eigenfunctions are continuous, $\displaystyle\lim_{j\to\infty}\zeta_j=\infty$, and bounded point evaluation gives
$\sum_j\phi_j(x)^2/(1+\zeta_j)\le C$ uniformly in $x$.
\textcolor{black}{For $t\ge t_0>0$, the diagonal tail after index $n$ is bounded uniformly in $x$ by $C\sup_{j>n}(1+\zeta_j)e^{-t_0\zeta_j}$.
This bound tends to zero as $n$ tends to infinity.}
Cauchy--Schwarz bounds off-diagonal tails.
Thus $\sum_je^{-t\zeta_j}\phi_j(x)\phi_j(y)$ defines a jointly continuous semigroup kernel.
\end{proof}

For $s>0$ and $x\in\mathfrak C$, define $G_s(x,x):=\int_0^\infty e^{-st}p_t(x,x)\,dt$.

\begin{lemma}\label{lem:resolvent-variational}
For $s>0$ and $x\in \mathfrak C$,
\begin{equation}\label{eq:resolvent-variational}
 G_s(x,x)=\sup_{f\in\mathcal D,\;f\ne0}
 \frac{f(x)^2}{\mathcal E(f)+s\|f\|_{L^2(\mu)}^2}.
\end{equation}
\end{lemma}

\begin{proof}
\textcolor{black}{For every $s>0$, $\mathcal E+s\|\cdot\|_2^2$ is the square of a norm equivalent to the form norm, and point evaluation is bounded for this norm.
Full support of $\mu$ makes the denominator in \eqref{eq:resolvent-variational} positive for every nonzero $f\in\mathcal D$.
The supremum is therefore the squared dual norm of point evaluation.
Spectral expansion identifies it with $\sum_j\phi_j(x)^2/(s+\zeta_j)$; Tonelli's theorem identifies this sum with the Laplace transform of the nonnegative diagonal heat series.}
\end{proof}

\begin{lemma}\label{lem:resolvent-component-bound}
For a positive-mass component $A$ containing $x$,
\begin{equation}\label{eq:resolvent-component}
 G_s(x,x)\le2\max\left\{\frac1{s\mu(A)},
               \diam^{\mathrm{intr}}_{R_*}(A)\right\}.
\end{equation}
\end{lemma}

\begin{proof}
For $y\in A$, the local resistance inequality gives
\begin{equation}\label{eq:component-evaluation}
 |f(x)|^2\le\frac2{\mu(A)}\int_Af^2\,d\mu
       +2\diam^{\mathrm{intr}}_{R_*}(A)\mathcal E_A^{\mathrm g}(f)
\end{equation}
after averaging.
Since the local energy is at most the global energy, this proves \eqref{eq:resolvent-component}.
\end{proof}

\begin{lemma}\label{lem:spectral-resolvent}
For $0<\beta<1$, a resolvent asymptotic $G_s(x,x)=s^{\beta-1+o(1)}$ implies $p_t(x,x)=t^{-\beta+o(1)}$.
Two-sided constant-factor resolvent bounds imply the corresponding constant-factor heat-kernel bounds.
\end{lemma}

\begin{proof}
The diagonal spectral series is nonincreasing in time.
Thus $p_t(x,x)\le eG_{1/t}(x,x)/t$.
The assumed resolvent estimate gives
$p_u(x,x)\le C_\varepsilon u^{-\beta-\varepsilon}$ for small $u$.
Splitting the Laplace integral at $t$ gives
$G_s(x,x)\le C_\varepsilon't^{1-\beta-\varepsilon}+s^{-1}p_t(x,x)$.
Choose $0<\varepsilon<\min\{\beta,(1-\beta)/2\}$, put
$s=t^{-(1-\beta-2\varepsilon)/(1-\beta+\varepsilon)}$.
The lower resolvent bound gives
$G_s(x,x)\ge t^{1-\beta-2\varepsilon}$ for small $t$.
The first upper-bound term is negligible.
Hence $p_t(x,x)\ge\frac12t^{-(1-\beta-2\varepsilon)(\beta-\varepsilon)/(1-\beta+\varepsilon)}$ eventually.
Let $\varepsilon\downarrow0$.

The same inequality gives the heat upper bound $Ct^{-\beta}$.
For a resolvent lower bound $G_{1/t}(x,x)\ge ct^{1-\beta}$, split the integral at $\delta t$.
Its first part is at most $C'\delta^{1-\beta}t^{1-\beta}$.
Choose fixed small $\delta$ to make this at most half the lower bound.
The remaining part is at most $t\,p_{\delta t}(x,x)$.
This gives the matching lower bound after a fixed time change.
\end{proof}

\appendixsection{Density of the spectral dimensions}\label{app:spectral-density}

We prove Theorem~\ref{thm:spectral-density} by constructing finite rules with fixed terminal geometry; for each rule, we take the natural growth limits before varying the rule.

\begin{proof}[Proof of Theorem~\ref{thm:spectral-density}]
Fix a rational number $B>1$.
For each sufficiently large integer $t$, take a series chain of $\lceil B^t\rceil$ blocks, each consisting of $\lceil t(4/3)^t\rceil$ internally disjoint paths of length $t$.
Add one edge at each end of this chain, and place a separate path of length two in parallel with the extended chain.
Use the resulting graph as the same deterministic rule at every substitution step.
It is finite, simple and canonical, and reflection exchanges its terminals.
The short and long branches give two edge-disjoint terminal paths, whereas each terminal has degree two.
Thus its terminal distance, terminal degrees and terminal mincut are all two.

All probabilities below use this rule's own critical parameter $p_c$.
Its critical equation is
\begin{equation}\label{eq:density-critical}
 \left[1-(1-p_c^t)^{\lceil t(4/3)^t\rceil}\right]^{\lceil B^t\rceil}
 =\frac{1}{p_c(1+p_c)}.
\end{equation}
Indeed, the full reliability is the short-path probability $p^2$ plus $(1-p^2)p^2$ times the chain-crossing probability.
The left side of \eqref{eq:density-critical} is strictly increasing in $p_c$, and the right side is strictly decreasing.
For all sufficiently large $t$, $7/10<p_c<4/5$.
At $p=7/10$, one block crosses with probability at most $2t(14/15)^t<1/2$, whereas the right side of \eqref{eq:density-critical} is $100/119$.
At $p=4/5$, the mean number of successful paths in a block is at least $t(16/15)^t$.
Once $(16/15)^t\ge B$, a union bound makes the entire chain's failure probability at most
$2B^te^{-tB}\le2e^{-t}<11/36$.
This gives the upper endpoint, where the right side is $25/36$.

Let $J$ be the number of successful paths in one block before conditioning.
It has distribution $\operatorname{Bin}(\lceil t(4/3)^t\rceil,p_c^t)$.
Equation~\eqref{eq:density-critical} implies
\[
 \Pp(J>0)>\frac{25}{36},
 \qquad
 \frac{1}{8\lceil B^t\rceil}<\Pp(J=0)<\frac{1}{\lceil B^t\rceil}.
\]
For the first failure bound, the chain fails with probability greater than $19/119>1/8$, and a union bound applies over its blocks.
For the second, a block failure probability at least $1/\lceil B^t\rceil$ would make the chain-crossing probability at most $e^{-1}<25/36$.
Using $(1-p_c^t)^{\lceil t(4/3)^t\rceil}\le e^{-\E J}$ and, eventually, $p_c^t<1/2$, we obtain
\begin{equation}\label{eq:density-path-count}
 \frac{t\log B}{2}<\E J
 <\log(8\lceil B^t\rceil)<(\lceil B\rceil+2)t,
 \qquad
 \E[J\mid J>0]<2(\lceil B\rceil+2)t.
\end{equation}
The lower bound follows from
$-\log(1-p_c^t)\le2p_c^t$ and the upper bound on $\Pp(J=0)$.

We next bound the natural resistance factor from below, uniformly in recursive depth.
\textcolor{black}{Give the successful edges independent positive resistances with a common finitely supported law $Y$, independent of their statuses, where $\E Y=1$ and $\Var(Y)\le8$.
Such laws include the normalised resistances at every finite recursive depth.}
For independent copies along a path,
\[
 \E e^{-Y/18}\le1-\frac1{18}+\frac9{2\cdot18^2}=\frac{23}{24},
 \qquad
 \Pp\left(\sum_{i=1}^tY_i<\frac t2\right)
 \le\left(\frac{69}{70}\right)^t.
\]
Here exponential Markov and $e^{1/36}\le36/35$ give the second estimate.
Let $U$ be the resistance of a block conditioned on crossing.
Given $J=j>0$, with probability at least $1-j(69/70)^t$ every successful path has resistance at least $t/2$.
Parallel addition and Jensen's inequality therefore give
\[
 \E U\ge\frac t2\E[J^{-1}\mid J>0]
          -\frac t2\left(\frac{69}{70}\right)^t
 >\frac{1}{4(\lceil B\rceil+2)}
          -\frac t2\left(\frac{69}{70}\right)^t
 >\frac{1}{8(\lceil B\rceil+2)}
\]
for sufficiently large $t$.
Choosing the first successful path bounds $\E U^2$ by $t^2+8t$.

Let $V$ be the resistance of the extended long branch conditioned on crossing, and let $Z$ be the resistance of the whole rule conditioned on crossing.
Under the long-branch conditioning, the blocks and exterior-edge inputs are independent, so
\[
 \E V>\frac{\lceil B^t\rceil}{8(\lceil B\rceil+2)},
 \qquad
 \Var(V)\le16+\lceil B^t\rceil(t^2+8t).
\]
\textcolor{black}{Before conditioning on whole-rule crossing, the short-branch status is independent of the long-branch states and resistances.
The whole rule crosses with probability $p_c$, whereas the short branch crosses with probability $p_c^2$.
Thus, given whole-rule crossing, the short branch crosses with probability $p_c^2/p_c=p_c$.
On this event, its two input resistances remain independent with law $Y$; their sum bounds $Z$ from above and has second moment $2\E Y^2+2(\E Y)^2\le20$.
On the complementary event, of conditional probability $1-p_c$, the short branch fails and the long branch crosses.
By the independence before conditioning, the long-branch resistance on this event has the law of $V$, and it equals the whole-rule resistance.}
Consequently
\[
 \E Z\ge(1-p_c)\E V,\qquad
 \E Z^2\le20p_c+(1-p_c)\bigl((\E V)^2+\Var(V)\bigr).
\]
Dividing by the squared mean lower bound yields
\begin{equation}\label{eq:density-variance}
 \frac{\Var(Z)}{(\E Z)^2}
 <4+\frac{1600(\lceil B\rceil+2)^2(t^2+8t+36)}{\lceil B^t\rceil}
 <8
\end{equation}
for sufficiently large $t$.
\textcolor{black}{All choices of sufficiently large $t$ above depend only on $B$, uniformly over the finitely supported input laws with the stated two moments.
Fix one such $t$.
Starting from unit input, apply the estimate at depth $n$ with $Y$ distributed as $\Reff_{\mathfrak C_n}(\mathfrak v_+,\mathfrak v_-)/\E\Reff_{\mathfrak C_n}(\mathfrak v_+,\mathfrak v_-)$.
The exact cell recursion and homogeneity identify the output $Z$ with $\Reff_{\mathfrak C_{n+1}}(\mathfrak v_+,\mathfrak v_-)/\E\Reff_{\mathfrak C_n}(\mathfrak v_+,\mathfrak v_-)$ in law.
Thus \eqref{eq:density-variance} preserves the variance bound at depth $n+1$, while $\E Z>\lceil B^t\rceil/[40(\lceil B\rceil+2)]$ gives}
{\color{black}
\[
 \frac{\E\Reff_{\mathfrak C_{n+1}}(\mathfrak v_+,\mathfrak v_-)}{\E\Reff_{\mathfrak C_n}(\mathfrak v_+,\mathfrak v_-)}
 \ge\frac{\lceil B^t\rceil}{40(\lceil B\rceil+2)}.
\]
}
\textcolor{black}{Induction and Theorem~\ref{thm:resistance-exponent}, applied to this fixed rule, yield $\lambda_{\mathfrak C}\ge\lceil B^t\rceil/[40(\lceil B\rceil+2)]$.}

For an upper bound, use unit inputs in one generation.
The binomial identity
\[
 \E\frac1{J+1}
 =\frac{1-(1-p_c^t)^{\lceil t(4/3)^t\rceil+1}}
        {(\lceil t(4/3)^t\rceil+1)p_c^t}
 \le\frac1{\E J}
\]
and $j^{-1}\le2/(j+1)$ for $j\ge1$ imply
$\E[J^{-1}\mid J>0]\le4/\E J$.
The successful block then has resistance $t/J$, so \eqref{eq:density-path-count} gives $\E U<8/\log B$.
Routing along the short branch when it crosses and otherwise along the long branch bounds the whole one-generation mean by $2+8\lceil B^t\rceil/\log B$.
Concavity and homogeneity of resistance bound every subsequent mean-growth factor by this unit-input mean.
Theorem~\ref{thm:resistance-exponent} and the lower bound above now give
\begin{equation}\label{eq:density-resistance-asymptotic}
 \frac{\lceil B^t\rceil}{40(\lceil B\rceil+2)}
 \le\lambda_{\mathfrak C}
 \le\left(2+\frac8{\log B}\right)\lceil B^t\rceil,
 \qquad
 \log\lambda_{\mathfrak C}=t\log B+O_B(1).
\end{equation}

It remains to estimate the full mass, including the dangling parts.
Given whole-rule crossing, the long branch crosses with probability $1/(1+p_c)$ by \eqref{eq:density-critical}.
On this event each block has its crossing-conditioned law.
In each of its paths, the first edge has conditional opening probability at least $p_c$, because block crossing is an increasing event.
Whenever this edge is open it is attached to the retained block terminal, so it contributes a connected child even if the rest of its path fails.
The connected diagonal entry of the full mass matrix therefore gives
\[
 \rho(\mathbf M_{\mathfrak C})\ge M_{c,c}
 \ge\frac{p_c}{1+p_c}\lceil B^t\rceil\lceil t(4/3)^t\rceil
 >\frac13\lceil B^t\rceil\lceil t(4/3)^t\rceil.
\]

For the upper bound, first examine the union of both terminal components of a block without conditioning.
On any length-$t$ path, its retained open edges lie in the two open prefixes from the path endpoints, whose expected total length is at most $2p_c/(1-p_c)<8$.
At most two closed frontier edges are also retained.
Hence a path contributes fewer than ten retained child cells on average.
Every globally retained child in a block lies in this local two-terminal union, since the block meets the rest of the rule only at its two terminals.
For each full parent state $c,b,u$, conditioning divides by $p_c$ or $1-p_c$, both greater than $1/5$.
The same bound applies to either inherited active-terminal choice for state $u$.
Each row sum of $\mathbf M_{\mathfrak C}$ is therefore less than
$50\lceil B^t\rceil\lceil t(4/3)^t\rceil+4$, including the two exterior edges and the short path.
A child in state $b$ is counted once, with both sides inherited from the same failed configuration.
The row-sum bound and the preceding lower bound yield
\begin{equation}\label{eq:density-mass-asymptotic}
 \frac13\lceil B^t\rceil\lceil t(4/3)^t\rceil
 <\rho(\mathbf M_{\mathfrak C})
 <50\lceil B^t\rceil\lceil t(4/3)^t\rceil+4,
\end{equation}
and hence
$\log\rho(\mathbf M_{\mathfrak C})
=t(\log B+\log(4/3))+\log t+O(1)$.

For fixed $B$, Theorem~\ref{thm:heat-kernels} and these two logarithmic estimates show that the typical spectral dimensions of the constructed rules satisfy
\[
 \lim_{t\to\infty}\dimspec^{(\mathrm N)}(\mathfrak C)
 =\frac{2(\log B+\log(4/3))}{2\log B+\log(4/3)}.
\]
Here $\mathfrak C$ is the critical cluster associated with the rule constructed at parameter $t$.
The expression on the right is continuous and strictly decreasing in $B>1$, with range $(1,2)$.
Given any nonempty open interval contained in $(1,2)$, choose a rational $B>1$ whose limiting value lies in that interval, and then choose a sufficiently large finite $t$.
The corresponding fixed classical rule has its typical spectral dimension in the prescribed interval, proving the theorem.
\end{proof}

\end{document}

%% file: figure/dhl_construction.tex
\begin{tikzpicture}[line width=0.45pt,line cap=round]
\begin{scope}[xshift=-3.8cm]
\node[anchor=east] at (4.4,2.65) {Rule:};
\draw[-{Stealth[length=4pt]}] (4.7,2.65) -- (5.7,2.65);
\fill (4.7,2.65) circle (1.4pt);
\fill (5.7,2.65) circle (1.4pt);
\draw[-{Stealth[length=5pt]}] (6.0,2.65) -- (6.7,2.65);
\draw[-{Stealth[length=3pt]}] (7.1,2.65) -- (8.3,3.25);
\draw[-{Stealth[length=3pt]}] (8.3,3.25) -- (9.5,2.65);
\draw[-{Stealth[length=3pt]}] (7.1,2.65) -- (8.3,2.05);
\draw[-{Stealth[length=3pt]}] (8.3,2.05) -- (9.5,2.65);
\foreach \x/\y in {7.1/2.65,8.3/3.25,9.5/2.65,8.3/2.05}{\fill (\x,\y) circle (1.4pt);}
\node[below left,font=\scriptsize] at (7.1,2.65) {$\beta^+$};
\node[below right,font=\scriptsize] at (9.5,2.65) {$\beta^-$};
\end{scope}
\draw[dashed,gray] (0,1.6) -- (14.7,1.6);

\begin{scope}[xshift=0cm]
\coordinate (v0) at (0.000000,0.000000);
\coordinate (v1) at (1.400000,0.000000);
\draw[-{Stealth[length=2pt,width=1.4pt]}] (v0) -- (v1);
\fill (v0) circle (0.85pt);
\fill (v1) circle (0.85pt);
\node at (0.7,-1.5) {$\Xi^{0}$};
\end{scope}
\draw[-{Stealth[length=4pt]}] (1.5999999999999999,0) -- (2.0999999999999996,0);
\begin{scope}[xshift=2.3cm]
\coordinate (v0) at (0.000000,0.000000);
\coordinate (v1) at (2.200000,0.000000);
\coordinate (v2) at (1.100000,0.550000);
\coordinate (v3) at (1.100000,-0.550000);
\draw[-{Stealth[length=2pt,width=1.4pt]}] (v0) -- (v2);
\draw[-{Stealth[length=2pt,width=1.4pt]}] (v2) -- (v1);
\draw[-{Stealth[length=2pt,width=1.4pt]}] (v0) -- (v3);
\draw[-{Stealth[length=2pt,width=1.4pt]}] (v3) -- (v1);
\fill (v0) circle (0.85pt);
\fill (v1) circle (0.85pt);
\fill (v2) circle (0.85pt);
\fill (v3) circle (0.85pt);
\node at (1.1,-1.5) {$\Xi^{1}$};
\end{scope}
\draw[-{Stealth[length=4pt]}] (4.7,0) -- (5.2,0);
\begin{scope}[xshift=5.4cm]
\coordinate (v0) at (0.000000,0.000000);
\coordinate (v1) at (3.100000,0.000000);
\coordinate (v2) at (1.550000,0.775000);
\coordinate (v3) at (1.550000,-0.775000);
\coordinate (v4) at (0.664286,0.608929);
\coordinate (v5) at (0.885714,0.166071);
\coordinate (v6) at (2.435714,0.608929);
\coordinate (v7) at (2.214286,0.166071);
\coordinate (v8) at (0.885714,-0.166071);
\coordinate (v9) at (0.664286,-0.608929);
\coordinate (v10) at (2.214286,-0.166071);
\coordinate (v11) at (2.435714,-0.608929);
\draw[-{Stealth[length=2pt,width=1.4pt]}] (v0) -- (v4);
\draw[-{Stealth[length=2pt,width=1.4pt]}] (v4) -- (v2);
\draw[-{Stealth[length=2pt,width=1.4pt]}] (v0) -- (v5);
\draw[-{Stealth[length=2pt,width=1.4pt]}] (v5) -- (v2);
\draw[-{Stealth[length=2pt,width=1.4pt]}] (v2) -- (v6);
\draw[-{Stealth[length=2pt,width=1.4pt]}] (v6) -- (v1);
\draw[-{Stealth[length=2pt,width=1.4pt]}] (v2) -- (v7);
\draw[-{Stealth[length=2pt,width=1.4pt]}] (v7) -- (v1);
\draw[-{Stealth[length=2pt,width=1.4pt]}] (v0) -- (v8);
\draw[-{Stealth[length=2pt,width=1.4pt]}] (v8) -- (v3);
\draw[-{Stealth[length=2pt,width=1.4pt]}] (v0) -- (v9);
\draw[-{Stealth[length=2pt,width=1.4pt]}] (v9) -- (v3);
\draw[-{Stealth[length=2pt,width=1.4pt]}] (v3) -- (v10);
\draw[-{Stealth[length=2pt,width=1.4pt]}] (v10) -- (v1);
\draw[-{Stealth[length=2pt,width=1.4pt]}] (v3) -- (v11);
\draw[-{Stealth[length=2pt,width=1.4pt]}] (v11) -- (v1);
\fill (v0) circle (0.85pt);
\fill (v1) circle (0.85pt);
\fill (v2) circle (0.85pt);
\fill (v3) circle (0.85pt);
\fill (v4) circle (0.85pt);
\fill (v5) circle (0.85pt);
\fill (v6) circle (0.85pt);
\fill (v7) circle (0.85pt);
\fill (v8) circle (0.85pt);
\fill (v9) circle (0.85pt);
\fill (v10) circle (0.85pt);
\fill (v11) circle (0.85pt);
\node at (1.55,-1.5) {$\Xi^{2}$};
\end{scope}
\draw[-{Stealth[length=4pt]}] (8.7,0) -- (9.2,0);
\begin{scope}[xshift=9.4cm]
\coordinate (v0) at (0.000000,0.000000);
\coordinate (v1) at (4.600000,0.000000);
\coordinate (v2) at (2.300000,1.150000);
\coordinate (v3) at (2.300000,-1.150000);
\coordinate (v4) at (0.985714,0.903571);
\coordinate (v5) at (1.314286,0.246429);
\coordinate (v6) at (3.614286,0.903571);
\coordinate (v7) at (3.285714,0.246429);
\coordinate (v8) at (1.314286,-0.246429);
\coordinate (v9) at (0.985714,-0.903571);
\coordinate (v10) at (3.285714,-0.246429);
\coordinate (v11) at (3.614286,-0.903571);
\coordinate (v12) at (0.436384,0.513393);
\coordinate (v13) at (0.549330,0.390178);
\coordinate (v14) at (1.627455,1.108928);
\coordinate (v15) at (1.658259,0.944643);
\coordinate (v16) at (0.641741,0.205357);
\coordinate (v17) at (0.672545,0.041072);
\coordinate (v18) at (1.750670,0.759822);
\coordinate (v19) at (1.863616,0.636607);
\coordinate (v20) at (2.972545,1.108928);
\coordinate (v21) at (2.941741,0.944643);
\coordinate (v22) at (4.163616,0.513393);
\coordinate (v23) at (4.050670,0.390178);
\coordinate (v24) at (2.849330,0.759822);
\coordinate (v25) at (2.736384,0.636607);
\coordinate (v26) at (3.958259,0.205357);
\coordinate (v27) at (3.927455,0.041072);
\coordinate (v28) at (0.672545,-0.041072);
\coordinate (v29) at (0.641741,-0.205357);
\coordinate (v30) at (1.863616,-0.636607);
\coordinate (v31) at (1.750670,-0.759822);
\coordinate (v32) at (0.549330,-0.390178);
\coordinate (v33) at (0.436384,-0.513393);
\coordinate (v34) at (1.658259,-0.944643);
\coordinate (v35) at (1.627455,-1.108928);
\coordinate (v36) at (2.736384,-0.636607);
\coordinate (v37) at (2.849330,-0.759822);
\coordinate (v38) at (3.927455,-0.041072);
\coordinate (v39) at (3.958259,-0.205357);
\coordinate (v40) at (2.941741,-0.944643);
\coordinate (v41) at (2.972545,-1.108928);
\coordinate (v42) at (4.050670,-0.390178);
\coordinate (v43) at (4.163616,-0.513393);
\draw[-{Stealth[length=2pt,width=1.4pt]}] (v0) -- (v12);
\draw[-{Stealth[length=2pt,width=1.4pt]}] (v12) -- (v4);
\draw[-{Stealth[length=2pt,width=1.4pt]}] (v0) -- (v13);
\draw[-{Stealth[length=2pt,width=1.4pt]}] (v13) -- (v4);
\draw[-{Stealth[length=2pt,width=1.4pt]}] (v4) -- (v14);
\draw[-{Stealth[length=2pt,width=1.4pt]}] (v14) -- (v2);
\draw[-{Stealth[length=2pt,width=1.4pt]}] (v4) -- (v15);
\draw[-{Stealth[length=2pt,width=1.4pt]}] (v15) -- (v2);
\draw[-{Stealth[length=2pt,width=1.4pt]}] (v0) -- (v16);
\draw[-{Stealth[length=2pt,width=1.4pt]}] (v16) -- (v5);
\draw[-{Stealth[length=2pt,width=1.4pt]}] (v0) -- (v17);
\draw[-{Stealth[length=2pt,width=1.4pt]}] (v17) -- (v5);
\draw[-{Stealth[length=2pt,width=1.4pt]}] (v5) -- (v18);
\draw[-{Stealth[length=2pt,width=1.4pt]}] (v18) -- (v2);
\draw[-{Stealth[length=2pt,width=1.4pt]}] (v5) -- (v19);
\draw[-{Stealth[length=2pt,width=1.4pt]}] (v19) -- (v2);
\draw[-{Stealth[length=2pt,width=1.4pt]}] (v2) -- (v20);
\draw[-{Stealth[length=2pt,width=1.4pt]}] (v20) -- (v6);
\draw[-{Stealth[length=2pt,width=1.4pt]}] (v2) -- (v21);
\draw[-{Stealth[length=2pt,width=1.4pt]}] (v21) -- (v6);
\draw[-{Stealth[length=2pt,width=1.4pt]}] (v6) -- (v22);
\draw[-{Stealth[length=2pt,width=1.4pt]}] (v22) -- (v1);
\draw[-{Stealth[length=2pt,width=1.4pt]}] (v6) -- (v23);
\draw[-{Stealth[length=2pt,width=1.4pt]}] (v23) -- (v1);
\draw[-{Stealth[length=2pt,width=1.4pt]}] (v2) -- (v24);
\draw[-{Stealth[length=2pt,width=1.4pt]}] (v24) -- (v7);
\draw[-{Stealth[length=2pt,width=1.4pt]}] (v2) -- (v25);
\draw[-{Stealth[length=2pt,width=1.4pt]}] (v25) -- (v7);
\draw[-{Stealth[length=2pt,width=1.4pt]}] (v7) -- (v26);
\draw[-{Stealth[length=2pt,width=1.4pt]}] (v26) -- (v1);
\draw[-{Stealth[length=2pt,width=1.4pt]}] (v7) -- (v27);
\draw[-{Stealth[length=2pt,width=1.4pt]}] (v27) -- (v1);
\draw[-{Stealth[length=2pt,width=1.4pt]}] (v0) -- (v28);
\draw[-{Stealth[length=2pt,width=1.4pt]}] (v28) -- (v8);
\draw[-{Stealth[length=2pt,width=1.4pt]}] (v0) -- (v29);
\draw[-{Stealth[length=2pt,width=1.4pt]}] (v29) -- (v8);
\draw[-{Stealth[length=2pt,width=1.4pt]}] (v8) -- (v30);
\draw[-{Stealth[length=2pt,width=1.4pt]}] (v30) -- (v3);
\draw[-{Stealth[length=2pt,width=1.4pt]}] (v8) -- (v31);
\draw[-{Stealth[length=2pt,width=1.4pt]}] (v31) -- (v3);
\draw[-{Stealth[length=2pt,width=1.4pt]}] (v0) -- (v32);
\draw[-{Stealth[length=2pt,width=1.4pt]}] (v32) -- (v9);
\draw[-{Stealth[length=2pt,width=1.4pt]}] (v0) -- (v33);
\draw[-{Stealth[length=2pt,width=1.4pt]}] (v33) -- (v9);
\draw[-{Stealth[length=2pt,width=1.4pt]}] (v9) -- (v34);
\draw[-{Stealth[length=2pt,width=1.4pt]}] (v34) -- (v3);
\draw[-{Stealth[length=2pt,width=1.4pt]}] (v9) -- (v35);
\draw[-{Stealth[length=2pt,width=1.4pt]}] (v35) -- (v3);
\draw[-{Stealth[length=2pt,width=1.4pt]}] (v3) -- (v36);
\draw[-{Stealth[length=2pt,width=1.4pt]}] (v36) -- (v10);
\draw[-{Stealth[length=2pt,width=1.4pt]}] (v3) -- (v37);
\draw[-{Stealth[length=2pt,width=1.4pt]}] (v37) -- (v10);
\draw[-{Stealth[length=2pt,width=1.4pt]}] (v10) -- (v38);
\draw[-{Stealth[length=2pt,width=1.4pt]}] (v38) -- (v1);
\draw[-{Stealth[length=2pt,width=1.4pt]}] (v10) -- (v39);
\draw[-{Stealth[length=2pt,width=1.4pt]}] (v39) -- (v1);
\draw[-{Stealth[length=2pt,width=1.4pt]}] (v3) -- (v40);
\draw[-{Stealth[length=2pt,width=1.4pt]}] (v40) -- (v11);
\draw[-{Stealth[length=2pt,width=1.4pt]}] (v3) -- (v41);
\draw[-{Stealth[length=2pt,width=1.4pt]}] (v41) -- (v11);
\draw[-{Stealth[length=2pt,width=1.4pt]}] (v11) -- (v42);
\draw[-{Stealth[length=2pt,width=1.4pt]}] (v42) -- (v1);
\draw[-{Stealth[length=2pt,width=1.4pt]}] (v11) -- (v43);
\draw[-{Stealth[length=2pt,width=1.4pt]}] (v43) -- (v1);
\fill (v0) circle (0.85pt);
\fill (v1) circle (0.85pt);
\fill (v2) circle (0.85pt);
\fill (v3) circle (0.85pt);
\fill (v4) circle (0.85pt);
\fill (v5) circle (0.85pt);
\fill (v6) circle (0.85pt);
\fill (v7) circle (0.85pt);
\fill (v8) circle (0.85pt);
\fill (v9) circle (0.85pt);
\fill (v10) circle (0.85pt);
\fill (v11) circle (0.85pt);
\fill (v12) circle (0.85pt);
\fill (v13) circle (0.85pt);
\fill (v14) circle (0.85pt);
\fill (v15) circle (0.85pt);
\fill (v16) circle (0.85pt);
\fill (v17) circle (0.85pt);
\fill (v18) circle (0.85pt);
\fill (v19) circle (0.85pt);
\fill (v20) circle (0.85pt);
\fill (v21) circle (0.85pt);
\fill (v22) circle (0.85pt);
\fill (v23) circle (0.85pt);
\fill (v24) circle (0.85pt);
\fill (v25) circle (0.85pt);
\fill (v26) circle (0.85pt);
\fill (v27) circle (0.85pt);
\fill (v28) circle (0.85pt);
\fill (v29) circle (0.85pt);
\fill (v30) circle (0.85pt);
\fill (v31) circle (0.85pt);
\fill (v32) circle (0.85pt);
\fill (v33) circle (0.85pt);
\fill (v34) circle (0.85pt);
\fill (v35) circle (0.85pt);
\fill (v36) circle (0.85pt);
\fill (v37) circle (0.85pt);
\fill (v38) circle (0.85pt);
\fill (v39) circle (0.85pt);
\fill (v40) circle (0.85pt);
\fill (v41) circle (0.85pt);
\fill (v42) circle (0.85pt);
\fill (v43) circle (0.85pt);
\node at (2.3,-1.5) {$\Xi^{3}$};
\end{scope}
\node at (14.55,0) {$\cdots$};
\end{tikzpicture}

%% file: figure/random_DHL_rules.tex
\begingroup
\definecolor{eigsred}{RGB}{205,46,53}
\definecolor{eigsblue}{RGB}{38,100,190}
\definecolor{eigsgreen}{RGB}{29,139,76}
\begin{tikzpicture}[x=1cm,y=1cm,line cap=round,line join=round,
 every node/.style={font=\small},
 bond/.style={line width=.85pt,-{Stealth[length=2.5pt,width=2.3pt]}},
 map/.style={black!45,line width=.65pt,-{Stealth[length=4pt]}}]
\path[use as bounding box] (0,-8.7) rectangle (16,.75);
\node[anchor=west,font=\small\bfseries] at (0,.62) {Rules};
\draw[bond,eigsred] (0.1500,-0.3500) -- (1.1000,-0.3500);
\fill (0.1500,-0.3500) circle (1.15pt);
\fill (1.1000,-0.3500) circle (1.15pt);
\draw[map] (1.40,-0.35) -- (2.05,-0.35);
\node at (4.25,0.18000000000000005) {$\frac{p^4}{2p^2-p^4}$};
\draw[bond,eigsred] (3.8700,-0.3500) -- (4.2500,-0.1100);
\draw[bond,eigsred] (4.2500,-0.1100) -- (4.6300,-0.3500);
\draw[bond,eigsred] (3.8700,-0.3500) -- (4.2500,-0.5900);
\draw[bond,eigsred] (4.2500,-0.5900) -- (4.6300,-0.3500);
\fill (3.8700,-0.3500) circle (1.0pt);
\fill (4.2500,-0.1100) circle (1.0pt);
\fill (4.2500,-0.5900) circle (1.0pt);
\fill (4.6300,-0.3500) circle (1.0pt);
\node at (8.85,0.18000000000000005) {$\frac{p^3(1-p)}{2p^2-p^4}$};
\draw[bond,eigsred] (6.9700,-0.3500) -- (7.3500,-0.1100);
\draw[bond,eigsred] (7.3500,-0.1100) -- (7.7300,-0.3500);
\draw[bond,eigsred] (6.9700,-0.3500) -- (7.3500,-0.5900);
\draw[bond,eigsblue] (7.3500,-0.5900) -- (7.7300,-0.3500);
\fill (6.9700,-0.3500) circle (1.0pt);
\fill (7.3500,-0.1100) circle (1.0pt);
\fill (7.3500,-0.5900) circle (1.0pt);
\fill (7.7300,-0.3500) circle (1.0pt);
\draw[bond,eigsred] (7.9700,-0.3500) -- (8.3500,-0.1100);
\draw[bond,eigsred] (8.3500,-0.1100) -- (8.7300,-0.3500);
\draw[bond,eigsblue] (7.9700,-0.3500) -- (8.3500,-0.5900);
\draw[bond,eigsred] (8.3500,-0.5900) -- (8.7300,-0.3500);
\fill (7.9700,-0.3500) circle (1.0pt);
\fill (8.3500,-0.1100) circle (1.0pt);
\fill (8.3500,-0.5900) circle (1.0pt);
\fill (8.7300,-0.3500) circle (1.0pt);
\draw[bond,eigsred] (8.9700,-0.3500) -- (9.3500,-0.1100);
\draw[bond,eigsblue] (9.3500,-0.1100) -- (9.7300,-0.3500);
\draw[bond,eigsred] (8.9700,-0.3500) -- (9.3500,-0.5900);
\draw[bond,eigsred] (9.3500,-0.5900) -- (9.7300,-0.3500);
\fill (8.9700,-0.3500) circle (1.0pt);
\fill (9.3500,-0.1100) circle (1.0pt);
\fill (9.3500,-0.5900) circle (1.0pt);
\fill (9.7300,-0.3500) circle (1.0pt);
\draw[bond,eigsblue] (9.9700,-0.3500) -- (10.3500,-0.1100);
\draw[bond,eigsred] (10.3500,-0.1100) -- (10.7300,-0.3500);
\draw[bond,eigsred] (9.9700,-0.3500) -- (10.3500,-0.5900);
\draw[bond,eigsred] (10.3500,-0.5900) -- (10.7300,-0.3500);
\fill (9.9700,-0.3500) circle (1.0pt);
\fill (10.3500,-0.1100) circle (1.0pt);
\fill (10.3500,-0.5900) circle (1.0pt);
\fill (10.7300,-0.3500) circle (1.0pt);
\node at (13.45,0.18000000000000005) {$\frac{p^2(1-p)^2}{2p^2-p^4}$};
\draw[bond,eigsred] (12.5700,-0.3500) -- (12.9500,-0.1100);
\draw[bond,eigsred] (12.9500,-0.1100) -- (13.3300,-0.3500);
\draw[bond,eigsgreen] (12.5700,-0.3500) -- (12.9500,-0.5900);
\draw[bond,eigsgreen] (13.3300,-0.3500) -- (12.9500,-0.5900);
\fill (12.5700,-0.3500) circle (1.0pt);
\fill (12.9500,-0.1100) circle (1.0pt);
\fill (12.9500,-0.5900) circle (1.0pt);
\fill (13.3300,-0.3500) circle (1.0pt);
\draw[bond,eigsgreen] (13.5700,-0.3500) -- (13.9500,-0.1100);
\draw[bond,eigsgreen] (14.3300,-0.3500) -- (13.9500,-0.1100);
\draw[bond,eigsred] (13.5700,-0.3500) -- (13.9500,-0.5900);
\draw[bond,eigsred] (13.9500,-0.5900) -- (14.3300,-0.3500);
\fill (13.5700,-0.3500) circle (1.0pt);
\fill (13.9500,-0.1100) circle (1.0pt);
\fill (13.9500,-0.5900) circle (1.0pt);
\fill (14.3300,-0.3500) circle (1.0pt);
\draw[black!12,line width=.35pt] (0,-0.9299999999999999) -- (16,-0.9299999999999999);
\draw[bond,eigsblue] (0.1500,-1.7000) -- (1.1000,-1.7000);
\fill (0.1500,-1.7000) circle (1.15pt);
\fill (1.1000,-1.7000) circle (1.15pt);
\draw[map] (1.40,-1.7000000000000002) -- (2.05,-1.7000000000000002);
\node at (4.25,-1.1700000000000002) {$\frac{p(1-p)^3}{(1-p^2)^2}$};
\draw[bond,eigsred] (2.3700,-1.7000) -- (2.7500,-1.4600);
\draw[bond,eigsblue] (2.7500,-1.4600) -- (3.1300,-1.7000);
\draw[bond,eigsgreen] (2.3700,-1.7000) -- (2.7500,-1.9400);
\draw[bond,eigsgreen] (3.1300,-1.7000) -- (2.7500,-1.9400);
\fill (2.3700,-1.7000) circle (1.0pt);
\fill (2.7500,-1.4600) circle (1.0pt);
\fill (2.7500,-1.9400) circle (1.0pt);
\fill (3.1300,-1.7000) circle (1.0pt);
\draw[bond,eigsblue] (3.3700,-1.7000) -- (3.7500,-1.4600);
\draw[bond,eigsred] (3.7500,-1.4600) -- (4.1300,-1.7000);
\draw[bond,eigsgreen] (3.3700,-1.7000) -- (3.7500,-1.9400);
\draw[bond,eigsgreen] (4.1300,-1.7000) -- (3.7500,-1.9400);
\fill (3.3700,-1.7000) circle (1.0pt);
\fill (3.7500,-1.4600) circle (1.0pt);
\fill (3.7500,-1.9400) circle (1.0pt);
\fill (4.1300,-1.7000) circle (1.0pt);
\draw[bond,eigsgreen] (4.3700,-1.7000) -- (4.7500,-1.4600);
\draw[bond,eigsgreen] (5.1300,-1.7000) -- (4.7500,-1.4600);
\draw[bond,eigsred] (4.3700,-1.7000) -- (4.7500,-1.9400);
\draw[bond,eigsblue] (4.7500,-1.9400) -- (5.1300,-1.7000);
\fill (4.3700,-1.7000) circle (1.0pt);
\fill (4.7500,-1.4600) circle (1.0pt);
\fill (4.7500,-1.9400) circle (1.0pt);
\fill (5.1300,-1.7000) circle (1.0pt);
\draw[bond,eigsgreen] (5.3700,-1.7000) -- (5.7500,-1.4600);
\draw[bond,eigsgreen] (6.1300,-1.7000) -- (5.7500,-1.4600);
\draw[bond,eigsblue] (5.3700,-1.7000) -- (5.7500,-1.9400);
\draw[bond,eigsred] (5.7500,-1.9400) -- (6.1300,-1.7000);
\fill (5.3700,-1.7000) circle (1.0pt);
\fill (5.7500,-1.4600) circle (1.0pt);
\fill (5.7500,-1.9400) circle (1.0pt);
\fill (6.1300,-1.7000) circle (1.0pt);
\node at (8.85,-1.1700000000000002) {$\frac{p^2(1-p)^2}{(1-p^2)^2}$};
\draw[bond,eigsred] (6.9700,-1.7000) -- (7.3500,-1.4600);
\draw[bond,eigsblue] (7.3500,-1.4600) -- (7.7300,-1.7000);
\draw[bond,eigsred] (6.9700,-1.7000) -- (7.3500,-1.9400);
\draw[bond,eigsblue] (7.3500,-1.9400) -- (7.7300,-1.7000);
\fill (6.9700,-1.7000) circle (1.0pt);
\fill (7.3500,-1.4600) circle (1.0pt);
\fill (7.3500,-1.9400) circle (1.0pt);
\fill (7.7300,-1.7000) circle (1.0pt);
\draw[bond,eigsblue] (7.9700,-1.7000) -- (8.3500,-1.4600);
\draw[bond,eigsred] (8.3500,-1.4600) -- (8.7300,-1.7000);
\draw[bond,eigsred] (7.9700,-1.7000) -- (8.3500,-1.9400);
\draw[bond,eigsblue] (8.3500,-1.9400) -- (8.7300,-1.7000);
\fill (7.9700,-1.7000) circle (1.0pt);
\fill (8.3500,-1.4600) circle (1.0pt);
\fill (8.3500,-1.9400) circle (1.0pt);
\fill (8.7300,-1.7000) circle (1.0pt);
\draw[bond,eigsred] (8.9700,-1.7000) -- (9.3500,-1.4600);
\draw[bond,eigsblue] (9.3500,-1.4600) -- (9.7300,-1.7000);
\draw[bond,eigsblue] (8.9700,-1.7000) -- (9.3500,-1.9400);
\draw[bond,eigsred] (9.3500,-1.9400) -- (9.7300,-1.7000);
\fill (8.9700,-1.7000) circle (1.0pt);
\fill (9.3500,-1.4600) circle (1.0pt);
\fill (9.3500,-1.9400) circle (1.0pt);
\fill (9.7300,-1.7000) circle (1.0pt);
\draw[bond,eigsblue] (9.9700,-1.7000) -- (10.3500,-1.4600);
\draw[bond,eigsred] (10.3500,-1.4600) -- (10.7300,-1.7000);
\draw[bond,eigsblue] (9.9700,-1.7000) -- (10.3500,-1.9400);
\draw[bond,eigsred] (10.3500,-1.9400) -- (10.7300,-1.7000);
\fill (9.9700,-1.7000) circle (1.0pt);
\fill (10.3500,-1.4600) circle (1.0pt);
\fill (10.3500,-1.9400) circle (1.0pt);
\fill (10.7300,-1.7000) circle (1.0pt);
\node at (13.45,-1.1700000000000002) {$\frac{(1-p)^4}{(1-p^2)^2}$};
\draw[bond,eigsgreen] (13.0700,-1.7000) -- (13.4500,-1.4600);
\draw[bond,eigsgreen] (13.8300,-1.7000) -- (13.4500,-1.4600);
\draw[bond,eigsgreen] (13.0700,-1.7000) -- (13.4500,-1.9400);
\draw[bond,eigsgreen] (13.8300,-1.7000) -- (13.4500,-1.9400);
\fill (13.0700,-1.7000) circle (1.0pt);
\fill (13.4500,-1.4600) circle (1.0pt);
\fill (13.4500,-1.9400) circle (1.0pt);
\fill (13.8300,-1.7000) circle (1.0pt);
\draw[black!12,line width=.35pt] (0,-2.2800000000000002) -- (16,-2.2800000000000002);
\draw[bond,eigsgreen] (0.1500,-3.0500) -- (1.1000,-3.0500);
\fill (0.1500,-3.0500) circle (1.15pt);
\fill (1.1000,-3.0500) circle (1.15pt);
\draw[map] (1.40,-3.0500000000000003) -- (2.05,-3.0500000000000003);
\node at (4.25,-2.5200000000000005) {$\frac{p(1-p)^2}{(1-p^2)^2}$};
\draw[bond,eigsred] (3.3700,-3.0500) -- (3.7500,-2.8100);
\draw[bond,eigsgreen] (3.7500,-2.8100) -- (4.1300,-3.0500);
\draw[bond,eigsgreen] (3.3700,-3.0500) -- (3.7500,-3.2900);
\draw[bond,black!65] (3.7500,-3.2900) -- (4.1300,-3.0500);
\fill (3.3700,-3.0500) circle (1.0pt);
\fill (3.7500,-2.8100) circle (1.0pt);
\fill (3.7500,-3.2900) circle (1.0pt);
\fill (4.1300,-3.0500) circle (1.0pt);
\draw[bond,eigsgreen] (4.3700,-3.0500) -- (4.7500,-2.8100);
\draw[bond,black!65] (4.7500,-2.8100) -- (5.1300,-3.0500);
\draw[bond,eigsred] (4.3700,-3.0500) -- (4.7500,-3.2900);
\draw[bond,eigsgreen] (4.7500,-3.2900) -- (5.1300,-3.0500);
\fill (4.3700,-3.0500) circle (1.0pt);
\fill (4.7500,-2.8100) circle (1.0pt);
\fill (4.7500,-3.2900) circle (1.0pt);
\fill (5.1300,-3.0500) circle (1.0pt);
\node at (8.85,-2.5200000000000005) {$\frac{p^2(1-p)^2}{(1-p^2)^2}$};
\draw[bond,eigsred] (8.4700,-3.0500) -- (8.8500,-2.8100);
\draw[bond,eigsgreen] (8.8500,-2.8100) -- (9.2300,-3.0500);
\draw[bond,eigsred] (8.4700,-3.0500) -- (8.8500,-3.2900);
\draw[bond,eigsgreen] (8.8500,-3.2900) -- (9.2300,-3.0500);
\fill (8.4700,-3.0500) circle (1.0pt);
\fill (8.8500,-2.8100) circle (1.0pt);
\fill (8.8500,-3.2900) circle (1.0pt);
\fill (9.2300,-3.0500) circle (1.0pt);
\node at (13.45,-2.5200000000000005) {$\frac{(1-p)^2}{(1-p^2)^2}$};
\draw[bond,eigsgreen] (13.0700,-3.0500) -- (13.4500,-2.8100);
\draw[bond,black!65] (13.4500,-2.8100) -- (13.8300,-3.0500);
\draw[bond,eigsgreen] (13.0700,-3.0500) -- (13.4500,-3.2900);
\draw[bond,black!65] (13.4500,-3.2900) -- (13.8300,-3.0500);
\fill (13.0700,-3.0500) circle (1.0pt);
\fill (13.4500,-2.8100) circle (1.0pt);
\fill (13.4500,-3.2900) circle (1.0pt);
\fill (13.8300,-3.0500) circle (1.0pt);
\draw[black!12,line width=.35pt] (0,-3.6300000000000003) -- (16,-3.6300000000000003);
\draw[bond,black!65] (0.1500,-4.4000) -- (1.1000,-4.4000);
\fill (0.1500,-4.4000) circle (1.15pt);
\fill (1.1000,-4.4000) circle (1.15pt);
\draw[map] (1.40,-4.4) -- (2.05,-4.4);
\node at (4.25,-3.87) {$1$};
\draw[bond,black!65] (3.8700,-4.4000) -- (4.2500,-4.1600);
\draw[bond,black!65] (4.2500,-4.1600) -- (4.6300,-4.4000);
\draw[bond,black!65] (3.8700,-4.4000) -- (4.2500,-4.6400);
\draw[bond,black!65] (4.2500,-4.6400) -- (4.6300,-4.4000);
\fill (3.8700,-4.4000) circle (1.0pt);
\fill (4.2500,-4.1600) circle (1.0pt);
\fill (4.2500,-4.6400) circle (1.0pt);
\fill (4.6300,-4.4000) circle (1.0pt);
\draw[densely dashed,black!45,line width=.6pt] (0,-5.0) -- (16,-5.0);
\node[anchor=west,font=\small\bfseries] at (0,-5.38) {One realisation};
\draw[bond,eigsred,line width=1pt] (0.2000,-6.6500) -- (1.0000,-6.6500);
\fill (0.2000,-6.6500) circle (1.15pt);
\fill (1.0000,-6.6500) circle (1.15pt);
\node at (.60,-8.45) {$\Xi^0$};
\draw[bond,eigsred] (1.8500,-6.6500) -- (2.6000,-6.3350);
\draw[bond,eigsred] (2.6000,-6.3350) -- (3.3500,-6.6500);
\draw[bond,eigsred] (1.8500,-6.6500) -- (2.6000,-6.9650);
\draw[bond,eigsblue] (2.6000,-6.9650) -- (3.3500,-6.6500);
\fill (1.8500,-6.6500) circle (1.0pt);
\fill (2.6000,-6.3350) circle (1.0pt);
\fill (2.6000,-6.9650) circle (1.0pt);
\fill (3.3500,-6.6500) circle (1.0pt);
\node at (2.6,-8.45) {$\Xi^1$};
\draw[bond,eigsred] (4.5000,-6.6500) -- (5.2328,-6.1540);
\draw[bond,eigsred] (5.2328,-6.1540) -- (6.1000,-5.9780);
\draw[bond,eigsgreen] (4.5000,-6.6500) -- (5.3672,-6.4740);
\draw[bond,eigsgreen] (6.1000,-5.9780) -- (5.3672,-6.4740);
\draw[bond,eigsred] (6.1000,-5.9780) -- (6.9672,-6.1540);
\draw[bond,eigsred] (6.9672,-6.1540) -- (7.7000,-6.6500);
\draw[bond,eigsred] (6.1000,-5.9780) -- (6.8328,-6.4740);
\draw[bond,eigsred] (6.8328,-6.4740) -- (7.7000,-6.6500);
\draw[bond,eigsgreen] (4.5000,-6.6500) -- (5.3672,-6.8260);
\draw[bond,eigsgreen] (6.1000,-7.3220) -- (5.3672,-6.8260);
\draw[bond,eigsred] (4.5000,-6.6500) -- (5.2328,-7.1460);
\draw[bond,eigsred] (5.2328,-7.1460) -- (6.1000,-7.3220);
\draw[bond,eigsred] (6.1000,-7.3220) -- (6.8328,-6.8260);
\draw[bond,eigsblue] (6.8328,-6.8260) -- (7.7000,-6.6500);
\draw[bond,eigsgreen] (6.1000,-7.3220) -- (6.9672,-7.1460);
\draw[bond,eigsgreen] (7.7000,-6.6500) -- (6.9672,-7.1460);
\fill (4.5000,-6.6500) circle (1.0pt);
\fill (6.1000,-5.9780) circle (1.0pt);
\fill (6.1000,-7.3220) circle (1.0pt);
\fill (7.7000,-6.6500) circle (1.0pt);
\fill (5.2328,-6.1540) circle (1.0pt);
\fill (5.3672,-6.4740) circle (1.0pt);
\fill (6.9672,-6.1540) circle (1.0pt);
\fill (6.8328,-6.4740) circle (1.0pt);
\fill (5.3672,-6.8260) circle (1.0pt);
\fill (5.2328,-7.1460) circle (1.0pt);
\fill (6.8328,-6.8260) circle (1.0pt);
\fill (6.9672,-7.1460) circle (1.0pt);
\node at (6.1,-8.45) {$\Xi^2$};
\draw[bond,eigsred,line width=.75pt] (8.9000,-6.6500) -- (9.4761,-6.1198);
\draw[bond,eigsred,line width=.75pt] (9.4761,-6.1198) -- (10.1824,-5.7820);
\draw[bond,eigsred,line width=.75pt] (8.9000,-6.6500) -- (9.6063,-6.3122);
\draw[bond,eigsblue,line width=.75pt] (9.6063,-6.3122) -- (10.1824,-5.7820);
\draw[bond,eigsred,line width=.75pt] (10.1824,-5.7820) -- (10.9181,-5.5142);
\draw[bond,eigsred,line width=.75pt] (10.9181,-5.5142) -- (11.7000,-5.4740);
\draw[bond,eigsgreen,line width=.75pt] (10.1824,-5.7820) -- (10.9643,-5.7418);
\draw[bond,eigsgreen,line width=.75pt] (11.7000,-5.4740) -- (10.9643,-5.7418);
\draw[bond,eigsred,line width=.75pt] (8.9000,-6.6500) -- (9.6357,-6.3822);
\draw[bond,eigsgreen,line width=.75pt] (9.6357,-6.3822) -- (10.4176,-6.3420);
\draw[bond,eigsred,line width=.75pt] (8.9000,-6.6500) -- (9.6819,-6.6098);
\draw[bond,eigsgreen,line width=.75pt] (9.6819,-6.6098) -- (10.4176,-6.3420);
\draw[bond,eigsgreen,line width=.75pt] (11.7000,-5.4740) -- (11.1239,-6.0042);
\draw[bond,black!65,line width=.75pt] (11.1239,-6.0042) -- (10.4176,-6.3420);
\draw[bond,eigsgreen,line width=.75pt] (11.7000,-5.4740) -- (10.9937,-5.8118);
\draw[bond,black!65,line width=.75pt] (10.9937,-5.8118) -- (10.4176,-6.3420);
\draw[bond,eigsred,line width=.75pt] (11.7000,-5.4740) -- (12.4819,-5.5142);
\draw[bond,eigsred,line width=.75pt] (12.4819,-5.5142) -- (13.2176,-5.7820);
\draw[bond,eigsred,line width=.75pt] (11.7000,-5.4740) -- (12.4357,-5.7418);
\draw[bond,eigsblue,line width=.75pt] (12.4357,-5.7418) -- (13.2176,-5.7820);
\draw[bond,eigsred,line width=.75pt] (13.2176,-5.7820) -- (13.9239,-6.1198);
\draw[bond,eigsred,line width=.75pt] (13.9239,-6.1198) -- (14.5000,-6.6500);
\draw[bond,eigsgreen,line width=.75pt] (13.2176,-5.7820) -- (13.7937,-6.3122);
\draw[bond,eigsgreen,line width=.75pt] (14.5000,-6.6500) -- (13.7937,-6.3122);
\draw[bond,eigsred,line width=.75pt] (11.7000,-5.4740) -- (12.4063,-5.8118);
\draw[bond,eigsred,line width=.75pt] (12.4063,-5.8118) -- (12.9824,-6.3420);
\draw[bond,eigsred,line width=.75pt] (11.7000,-5.4740) -- (12.2761,-6.0042);
\draw[bond,eigsred,line width=.75pt] (12.2761,-6.0042) -- (12.9824,-6.3420);
\draw[bond,eigsgreen,line width=.75pt] (12.9824,-6.3420) -- (13.7643,-6.3822);
\draw[bond,eigsgreen,line width=.75pt] (14.5000,-6.6500) -- (13.7643,-6.3822);
\draw[bond,eigsred,line width=.75pt] (12.9824,-6.3420) -- (13.7181,-6.6098);
\draw[bond,eigsred,line width=.75pt] (13.7181,-6.6098) -- (14.5000,-6.6500);
\draw[bond,eigsred,line width=.75pt] (8.9000,-6.6500) -- (9.6819,-6.6902);
\draw[bond,eigsgreen,line width=.75pt] (9.6819,-6.6902) -- (10.4176,-6.9580);
\draw[bond,eigsgreen,line width=.75pt] (8.9000,-6.6500) -- (9.6357,-6.9178);
\draw[bond,black!65,line width=.75pt] (9.6357,-6.9178) -- (10.4176,-6.9580);
\draw[bond,eigsgreen,line width=.75pt] (11.7000,-7.8260) -- (10.9937,-7.4882);
\draw[bond,black!65,line width=.75pt] (10.9937,-7.4882) -- (10.4176,-6.9580);
\draw[bond,eigsred,line width=.75pt] (11.7000,-7.8260) -- (11.1239,-7.2958);
\draw[bond,eigsgreen,line width=.75pt] (11.1239,-7.2958) -- (10.4176,-6.9580);
\draw[bond,eigsred,line width=.75pt] (8.9000,-6.6500) -- (9.6063,-6.9878);
\draw[bond,eigsred,line width=.75pt] (9.6063,-6.9878) -- (10.1824,-7.5180);
\draw[bond,eigsred,line width=.75pt] (8.9000,-6.6500) -- (9.4761,-7.1802);
\draw[bond,eigsred,line width=.75pt] (9.4761,-7.1802) -- (10.1824,-7.5180);
\draw[bond,eigsgreen,line width=.75pt] (10.1824,-7.5180) -- (10.9643,-7.5582);
\draw[bond,eigsgreen,line width=.75pt] (11.7000,-7.8260) -- (10.9643,-7.5582);
\draw[bond,eigsred,line width=.75pt] (10.1824,-7.5180) -- (10.9181,-7.7858);
\draw[bond,eigsred,line width=.75pt] (10.9181,-7.7858) -- (11.7000,-7.8260);
\draw[bond,eigsred,line width=.75pt] (11.7000,-7.8260) -- (12.2761,-7.2958);
\draw[bond,eigsred,line width=.75pt] (12.2761,-7.2958) -- (12.9824,-6.9580);
\draw[bond,eigsred,line width=.75pt] (11.7000,-7.8260) -- (12.4063,-7.4882);
\draw[bond,eigsblue,line width=.75pt] (12.4063,-7.4882) -- (12.9824,-6.9580);
\draw[bond,eigsred,line width=.75pt] (12.9824,-6.9580) -- (13.7181,-6.6902);
\draw[bond,eigsblue,line width=.75pt] (13.7181,-6.6902) -- (14.5000,-6.6500);
\draw[bond,eigsred,line width=.75pt] (12.9824,-6.9580) -- (13.7643,-6.9178);
\draw[bond,eigsblue,line width=.75pt] (13.7643,-6.9178) -- (14.5000,-6.6500);
\draw[bond,eigsred,line width=.75pt] (11.7000,-7.8260) -- (12.4357,-7.5582);
\draw[bond,eigsgreen,line width=.75pt] (12.4357,-7.5582) -- (13.2176,-7.5180);
\draw[bond,eigsred,line width=.75pt] (11.7000,-7.8260) -- (12.4819,-7.7858);
\draw[bond,eigsgreen,line width=.75pt] (12.4819,-7.7858) -- (13.2176,-7.5180);
\draw[bond,eigsgreen,line width=.75pt] (14.5000,-6.6500) -- (13.9239,-7.1802);
\draw[bond,black!65,line width=.75pt] (13.9239,-7.1802) -- (13.2176,-7.5180);
\draw[bond,eigsgreen,line width=.75pt] (14.5000,-6.6500) -- (13.7937,-6.9878);
\draw[bond,black!65,line width=.75pt] (13.7937,-6.9878) -- (13.2176,-7.5180);
\fill (8.9000,-6.6500) circle (0.75pt);
\fill (11.7000,-5.4740) circle (0.75pt);
\fill (11.7000,-7.8260) circle (0.75pt);
\fill (14.5000,-6.6500) circle (0.75pt);
\fill (10.1824,-5.7820) circle (0.75pt);
\fill (10.4176,-6.3420) circle (0.75pt);
\fill (13.2176,-5.7820) circle (0.75pt);
\fill (12.9824,-6.3420) circle (0.75pt);
\fill (10.4176,-6.9580) circle (0.75pt);
\fill (10.1824,-7.5180) circle (0.75pt);
\fill (12.9824,-6.9580) circle (0.75pt);
\fill (13.2176,-7.5180) circle (0.75pt);
\fill (9.4761,-6.1198) circle (0.75pt);
\fill (9.6063,-6.3122) circle (0.75pt);
\fill (10.9181,-5.5142) circle (0.75pt);
\fill (10.9643,-5.7418) circle (0.75pt);
\fill (9.6357,-6.3822) circle (0.75pt);
\fill (9.6819,-6.6098) circle (0.75pt);
\fill (11.1239,-6.0042) circle (0.75pt);
\fill (10.9937,-5.8118) circle (0.75pt);
\fill (12.4819,-5.5142) circle (0.75pt);
\fill (12.4357,-5.7418) circle (0.75pt);
\fill (13.9239,-6.1198) circle (0.75pt);
\fill (13.7937,-6.3122) circle (0.75pt);
\fill (12.4063,-5.8118) circle (0.75pt);
\fill (12.2761,-6.0042) circle (0.75pt);
\fill (13.7643,-6.3822) circle (0.75pt);
\fill (13.7181,-6.6098) circle (0.75pt);
\fill (9.6819,-6.6902) circle (0.75pt);
\fill (9.6357,-6.9178) circle (0.75pt);
\fill (10.9937,-7.4882) circle (0.75pt);
\fill (11.1239,-7.2958) circle (0.75pt);
\fill (9.6063,-6.9878) circle (0.75pt);
\fill (9.4761,-7.1802) circle (0.75pt);
\fill (10.9643,-7.5582) circle (0.75pt);
\fill (10.9181,-7.7858) circle (0.75pt);
\fill (12.2761,-7.2958) circle (0.75pt);
\fill (12.4063,-7.4882) circle (0.75pt);
\fill (13.7181,-6.6902) circle (0.75pt);
\fill (13.7643,-6.9178) circle (0.75pt);
\fill (12.4357,-7.5582) circle (0.75pt);
\fill (12.4819,-7.7858) circle (0.75pt);
\fill (13.9239,-7.1802) circle (0.75pt);
\fill (13.7937,-6.9878) circle (0.75pt);
\node at (11.7,-8.45) {$\Xi^3$};
\draw[map] (1.2,-6.65) -- (1.65,-6.65);
\draw[map] (3.6,-6.65) -- (4.1,-6.65);
\draw[map] (8.0,-6.65) -- (8.55,-6.65);
\draw[map] (14.8,-6.65) -- (15.35,-6.65);
\node at (15.75,-6.65) {$\cdots$};
\end{tikzpicture}
\endgroup

%% file: main.bbl
\providecommand{\bysame}{\leavevmode\hbox to3em{\hrulefill}\thinspace}
\providecommand{\MR}{\relax\ifhmode\unskip\space\fi MR }
\providecommand{\MRhref}[2]{%
  \href{http://www.ams.org/mathscinet-getitem?mr=#1}{#2}
}
\providecommand{\href}[2]{#2}
\begin{thebibliography}{10}

\bibitem{aharony1984breakdown}
Amnon Aharony and Dietrich Stauffer, \emph{Possible breakdown of the
  {Alexander--Orbach} rule at low dimensionalities}, Physical Review Letters
  \textbf{52} (1984), no.~26, 2368--2370.

\bibitem{AldousBandyopadhyay2005}
David~J. Aldous and Antar Bandyopadhyay, \emph{A survey of max-type recursive
  distributional equations}, The Annals of Applied Probability \textbf{15}
  (2005), no.~2, 1047--1110.

\bibitem{alexander1982density}
S.~Alexander and R.~Orbach, \emph{Density of states on fractals: ``fractons''},
  Journal de Physique Lettres \textbf{43} (1982), no.~17, L625--L631.

\bibitem{alves2026percolation}
Caio Alves, Rangel Baldasso, Carlos~Gustavo Moreira, and Augusto Teixeira,
  \emph{Percolation on hierarchical lattices}, 2026, arXiv:2606.11503
  [math.PR].

\bibitem{anttila2024replacement}
Riku Anttila and Sylvester Eriksson-Bique, \emph{On constructions of fractal
  spaces using replacement and the combinatorial {Loewner} property}, 2024,
  arXiv:2406.08062v2.

\bibitem{anttila2026loewner}
\bysame, \emph{The combinatorial {Loewner} property and super-multiplicativity
  inequalities for symmetric self-similar metric spaces}, Potential Analysis
  \textbf{65} (2026), Paper No.~9, 65 pp.

\bibitem{anttila2025energy}
Riku Anttila, Sylvester Eriksson-Bique, and Ryosuke Shimizu, \emph{Construction
  of self-similar energy forms and singularity of {Sobolev} spaces on
  {Laakso}-type fractal spaces}, 2025, arXiv:2503.13258v2.

\bibitem{berker1979renormalisation}
A.~N. Berker and S.~Ostlund, \emph{Renormalisation-group calculations of finite
  systems: order parameter and specific heat for epitaxial ordering}, Journal
  of Physics C: Solid State Physics \textbf{12} (1979), no.~22, 4961--4975.

\bibitem{Croydon2018}
D.~A. Croydon, \emph{Scaling limits of stochastic processes associated with
  resistance forms}, Annales de l'Institut Henri Poincar\'e, Probabilit\'es et
  Statistiques \textbf{54} (2018), no.~4, 1939--1968.

\bibitem{croydon2008selfsimilarity}
D.~A. Croydon and B.~M. Hambly, \emph{Self-similarity and spectral asymptotics
  for the continuum random tree}, Stochastic Processes and their Applications
  \textbf{118} (2008), no.~5, 730--754.

\bibitem{croydon2010spectral}
\bysame, \emph{Spectral asymptotics for stable trees}, Electronic Journal of
  Probability \textbf{15} (2010), no.~57, 1772--1801.

\bibitem{fitzner2017meanfield}
Robert Fitzner and Remco van~der Hofstad, \emph{Mean-field behavior for
  nearest-neighbor percolation in {$d>10$}}, Electronic Journal of Probability
  \textbf{22} (2017), Paper No.~43, 65 pp.

\bibitem{gefen1983anomalous}
Yuval Gefen, Amnon Aharony, and Shlomo Alexander, \emph{Anomalous diffusion on
  percolating clusters}, Physical Review Letters \textbf{50} (1983), no.~1,
  77--80.

\bibitem{grassberger1999conductivity}
Peter Grassberger, \emph{Conductivity exponent and backbone dimension in {2-d}
  percolation}, Physica A: Statistical Mechanics and its Applications
  \textbf{262} (1999), no.~3--4, 251--263.

\bibitem{griffiths1982spin}
Robert~B. Griffiths and Miron Kaufman, \emph{Spin systems on hierarchical
  lattices. {Introduction} and thermodynamic limit}, Physical Review B
  \textbf{26} (1982), no.~9, 5022--5032.

\bibitem{hambly2010diffusion}
B.~M. Hambly and T.~Kumagai, \emph{Diffusion on the scaling limit of the
  critical percolation cluster in the diamond hierarchical lattice},
  Communications in Mathematical Physics \textbf{295} (2010), no.~1, 29--69.

\bibitem{hauser1986geometrical}
P.~R. Hauser and V.~K. Saxena, \emph{Geometrical phase transitions on
  hierarchical lattices and universality}, Physical Review B \textbf{34}
  (1986), no.~11, 8193--8195.

\bibitem{kaufman1981exactly}
Miron Kaufman and Robert~B. Griffiths, \emph{Exactly soluble {Ising} models on
  hierarchical lattices}, Physical Review B \textbf{24} (1981), no.~1,
  496--498.

\bibitem{Kigami2003}
Jun Kigami, \emph{Harmonic analysis for resistance forms}, Journal of
  Functional Analysis \textbf{204} (2003), no.~2, 399--444.

\bibitem{Kigami2012}
\bysame, \emph{Resistance forms, quasisymmetric maps and heat kernel
  estimates}, Memoirs of the American Mathematical Society \textbf{216} (2012),
  no.~1015, vi+132.

\bibitem{kozma2009alexander}
Gady Kozma and Asaf Nachmias, \emph{The {Alexander--Orbach} conjecture holds in
  high dimensions}, Inventiones mathematicae \textbf{178} (2009), no.~3,
  635--654.

\bibitem{KurtzLyonsPemantlePeres1997}
Thomas~G. Kurtz, Russell Lyons, Robin Pemantle, and Yuval Peres, \emph{A
  conceptual proof of the {Kesten--Stigum} theorem for multi-type branching
  processes}, Classical and Modern Branching Processes (Krishna~B. Athreya and
  Peter Jagers, eds.), The IMA Volumes in Mathematics and its Applications,
  vol.~84, Springer, New York, 1997, pp.~181--185.

\bibitem{li2024on}
Nero Li and Thomas Britz, \emph{On the scale-freeness of random colored
  substitution networks}, Proceedings of the American Mathematical Society
  \textbf{152} (2024), no.~4, 1377--1389.

\bibitem{li2025reducible}
Nero~Ziyu Li, Frank~Xin Hu, and Thomas Britz, \emph{Reducible {Iterated Graph
  Systems}: multiscale-freeness and multifractals}, Nonlinearity \textbf{39}
  (2026), no.~4, 045009.

\bibitem{lobb1984percolative}
C.~J. Lobb and D.~J. Frank, \emph{Percolative conduction and the
  {Alexander--Orbach} conjecture in two dimensions}, Physical Review B
  \textbf{30} (1984), no.~7, 4090--4092.

\bibitem{neroli2024fractal}
Ziyu Neroli, \emph{Fractal dimensions for iterated graph systems}, Proceedings
  of the Royal Society A: Mathematical, Physical and Engineering Sciences
  \textbf{480} (2024), no.~2300, 20240406.

\bibitem{Neroli2026}
\bysame, \emph{{Iterated Graph Systems} ({I}): random walks and diffusion
  limits}, 2026, arXiv:2603.13798v3.

\bibitem{NeroliIGSII2026}
\bysame, \emph{{Iterated Graph Systems} ({II}): {Bernoulli} percolation and
  critical-exponent universality classes on hierarchical lattices}, 2026,
  arXiv:2610.10469v1.

\bibitem{stauffer1994percolation}
Dietrich Stauffer and Amnon Aharony, \emph{Introduction to percolation theory},
  revised second ed., Taylor \& Francis, London, 1994.

\bibitem{xi2017fractality}
Lifeng Xi, Lihong Wang, Songjing Wang, Zhouyu Yu, and Qin Wang,
  \emph{Fractality and scale-free effect of a class of self-similar networks},
  Physica A: Statistical Mechanics and its Applications \textbf{478} (2017),
  31--40.

\end{thebibliography}
